\documentclass{article}

\usepackage[all]{xy}
\usepackage{amsmath} %Aqu\'{\i} esta la tipograf\'{\i}a \mathbb
\usepackage{amssymb}
\usepackage{theorem}
\usepackage{pst-all}

\newtheorem{theorem}{Theorem}[section]
\newtheorem{lemma}[theorem]{Lemma}
\newtheorem{proposition}[theorem]{Proposition}
\newtheorem{corollary}[theorem]{Corollary}

\theoremstyle{definition}
\newtheorem{definition}[theorem]{Definition}
\newtheorem{example}[theorem]{Example}

\theoremstyle{remark}
\newtheorem{remark}[theorem]{Remark}

\numberwithin{equation}{section}

\newenvironment{proof}
  {\begin{trivlist}
  \item[\textit{\noindent\textsc{ Proof.}}]}
  {{$\square$}\end{trivlist}}

\newcommand{\Adj}{\ensuremath{\mbox{\rm Adj}}}
\newcommand{\Const}{\ensuremath{\mbox{\rm Const}}}

\newcommand{\Hom}{\ensuremath{\mbox{\rm Hom}}}

\newcommand{\Spec}{\ensuremath{\mbox{\rm Spec}}}
\newcommand{\DiffSpec}{\ensuremath{\mbox{\rm DiffSpec}}}
\newcommand{\Diff}{\ensuremath{\mbox{\rm Diff}}}
\newcommand{\Aut}{\ensuremath{\mbox{\rm Aut}}}
\newcommand{\Der}{\ensuremath{\mbox{\rm Der}}}
\newcommand{\Gal}{\ensuremath{\mbox{\rm Gal}}}

\newcommand{\Gr}{\ensuremath{\mbox{\rm Gr}}}

\begin{document}

\title{Differential Galois Theory of Algebraic Lie-Vessiot Systems
} 
\author{David Bl\'azquez-Sanz \& Juan Jos\'e Morales-Ruiz\footnote{This research has been partially
 supported by grant MCyT-FEDER MTM2006-00478of Spanish goverment, and the Sergio Arboleda University 
Research Agency CIVILIZAR.}}

\maketitle

\begin{abstract}
  In this paper we develop a differential Galois theory for
algebraic Lie-Vessiot systems in algebraic homogeneous spaces.
Lie-Vessiot systems are non autonomous vector fields
that are linear combinations with time-dependent coefficients of 
fundamental vector fields of an algebraic Lie group action. Those systems
are the building blocks for differential equations that admit
superposition of solutions. Lie-Vessiot systems in algebraic 
homogeneous spaces include the case of linear differential equations. 
Therefore, the differential Galois theory for Lie-Vessiot systems is an
extension of the classical Picard-Vessiot theory. In particular,
algebraic Lie-Vessiot systems are solvable in terms of Kolchin's strongly
normal extensions. Therefore, strongly normal extensions are geometrically
interpreted as the fields of functions on principal homogeneous spaces
over the Galois group. Finally we consider the problem
of integrability and solvability of automorphic differential equations. Our
main tool is a classical method of reduction, somewhere cited as Lie reduction. We
develop and algebraic version of this method, that we call Lie-Kolchin reduction.
Obstructions to the application are related to Galois cohomology.
% of the
%group of definition of the differential equation.  
\end{abstract}

\section{Introduction}

  A Lie-Vessiot system, as defined in \cite{BM2008}, is a system
of non-autonomous differential equations,
\begin{equation}\label{GeneralNA} \dot x_i = F_i(t,x_1,\ldots,x_n),\end{equation}
such that there exist $r$ functions $f(t)$ of the parameter
$t$ verifying:
$$F_i(t,x_1,\ldots,x_n) = \sum_{j=1}^rf_j(t)(A_jx_i),$$
where $A_1,\ldots,A_s$ are autonomous vector fields which infinitesimally span
a \emph{pretransitive} Lie group action. Such systems were introduced
by S. Lie at the end of 19th century (see, for instance \cite{Lie1893c}). 
The differential equation \eqref{GeneralNA}, interpreted
as a non-autonomous vector field, in a manifold $M$, is a linear combination
of the infinitesimal generators of the action of $G$ in $M$:
$$\vec X = \frac{\partial}{\partial t} + \sum f_j(t)A_j.$$

  In \cite{BM2008}, it is proven that a differential equation
admits a \emph{superposition law} if and only if it is 
a Lie-Vessiot system related to a pretransitive Lie group action 
(this is the global version of a classical result 
exposed in \cite{Lie1893c}). The
orbits by a pretransitive group action are homogeneous $G$-spaces, 
so that we can decompose a Lie-Vessiot system in a family of 
systems on homogeneous spaces. Therefore, Lie-Vessiot systems
on homogeneous spaces are the building blocks of differential equations
admitting superpostion laws.

Here, we study Lie-Vessiot systems on algebraic
homogeneous spaces $M$ with coefficients $f_i$ in a differential
field $\mathcal K$ whose field of constants $\mathcal C$ is the field of definition
of the phase space $M$. In this frame, a Lie-Vessiot system is seen as a derivation
of the scheme $M_{\mathcal K}$, compatible with the canonical 
derivation of $\mathcal K$.

\subsection*{Notation and Conventions}

  We denote differential and ordinary fields and rings
by calligraphic letters $\mathcal C, \mathcal K, \ldots $
The canonical derivation of a differential ring $\mathcal K$ is denoted by $\partial_{\mathcal K}$
or just $\partial$ whenever it does not lead to confussion. 
Algebraic varieties are denoted by capital letters $M, G, \ldots$
The structure sheaf of $M$ is denoted by $\mathcal O_M$. If $M$ is
a $\mathcal C$-algebraic variety and $\mathcal C\subset \mathcal K$, 
the space of $\mathcal K$-points of an algebraic variety $M$ is
denoted by $M(\mathcal K)$. We write $M_{\mathcal K}$ for the $\mathcal K$-algebraic
variety obtained after base change $M\times_{\mathcal C}\Spec(\mathcal K)$.
If $p$ is a point of $M$ we denote by $\kappa(p)$ its quotient field and 
$p^\natural$ the valuation morphism
$p^\natural\colon \mathcal O_{M,p}\to \kappa(p)$.

\setcounter{tocdepth}{1}
\tableofcontents

%%%%%%%%%%%%%%%%%%%%%%%%%%%%%%%%%%%%%%%%%%%%%%%%%%%%%%%%%%%%%%%%%%%%%%%%%
%%%%%%%%%%%%%%%%%%%%%%%%%%%%%%%%%%%%%%%%%%%%%%%%%%%%%%%%%%%%%%%%%%%%%%%%%
%%%%%%%%%%%%%%%%%%%%%%%%%%%%%%%%%%%%%%%%%%%%%%%%%%%%%%%%%%%%%%%%%%%%%%%%%

\section{Algebraic Groups and Homogeneous Spaces}

\subsection{Algebraic Groups}

Let us consider a field $\mathcal C$ and its algebraic closure $\bar{\mathcal C}$.
By an \emph{algebraic variety} over $\mathcal C$ we mean
a reduced and separated scheme of finite type over $\mathcal C$. Along this
text an \emph{algebraic group} means an algebraic variety endowed with an
algebraic group law and inversion morphism. In particular, algebraic groups 
over fields of characteristic zero are
smooth varieties (\cite{Mum} pp. 101--102).

  The functor of points of an algebraic group takes values on the category
of groups. If $G$ is a $\mathcal C$-algebraic group, and $\mathcal K$ is a 
$\mathcal C$-algebra, then the set $G(\mathcal K)$ of $\mathcal K$-points of 
$G$ is naturally endowed with an structure of group.

  An algebraic group is an \emph{affine group} if it is an affine algebraic
variety. The main example of an affine algebraic group is the General Linear Group,
$$GL(n,\mathcal C) = \Spec\left( \mathcal C[x_{ij},\Delta]\right), \quad \Delta = \frac{1}{|x_{ij}|} .$$
We call algebraic linear groups to the Zariski closed subgroups of $GL(n,\mathcal C)$.
It is well known that any affine algebraic group is isomorphic to an algebraic linear
group.

\subsection{Lie Algebra of an Algebraic Group}

  Let us consider $\mathfrak X(G)$ the space of regular vector
fields in $G$, \emph{id est}, derivations of the sheaf $\mathcal O_G$
vanishing on $\mathcal C$. The Lie bracket of regular vector fields is a regular vector
field, so $\mathfrak X(G)$ is a Lie algebra.

\begin{definition}
  Let $A$ be a regular vector field in $G$, and $\psi\colon G\to G$ an
  automorphism of algebraic variety. Then, we define $\psi(A)$ the
  transformed vector field $\psi(A) = (\psi^\sharp)^{-1}\circ A\circ
  \psi^\sharp$.
  $$\xymatrix{\mathcal O_G\ar[r]^-{\psi(A)}\ar[d]_-{\psi^{\sharp}} & \mathcal O_G \\
  \mathcal O_G \ar[r]^-{A} & \mathcal O_G
  \ar[u]_-{(\psi^\sharp)^{-1}}}$$
\end{definition}

  Any $\mathcal C$-point $\sigma$ of $G$ induces right and left translations, $R_{\sigma}$ and $L_{\sigma}$,
which are automorphisms of the algebraic variety $G$. A $\bar{\mathcal C}$-point $\bar\sigma$ of $G$,
induces translations in $G_{\bar{\mathcal C}}$. 

\begin{definition}
  The Lie algebra $\mathcal R(G)$ of $G$ is the space of all regular
  vector fields $A\in\mathfrak X(G)$ such that for all $\bar{\mathcal C}$-point $\sigma\in G(\bar{\mathcal C})$,
  $R_\sigma(A \otimes 1) =A \otimes 1$. In the same way, we define the
  Lie algebra $\mathcal L(G)$  of left invariant vector fields.
\end{definition}

  The Lie bracket of two right invariant vector field is a right
invariant vector field. The same is true for left invariant vector
fields, so $\mathcal R(G)$ and $\mathcal L(G)$ are Lie
sub-algebras of $\mathfrak X(G)$. For a point $x\in G$ its tangent space $T_\sigma G$ 
is defined as the space of $\mathcal C$-derivations from the ring of germs of regular functions, 
$\mathcal O_{G,\sigma}$ with values in its quotient field $\kappa(\sigma)$. 
It is a $\kappa(\sigma)$-vector space
of the same dimentsion than $G$. Any regualr vector field $\vec X$ in $\mathfrak X(G)$,
can be seen as a map $\sigma\mapsto \vec X_{\sigma}\in T_\sigma(G)$. Let us consider
$e$ the identity element of $G$. If $\mathcal C$ is algebraically closed, for any vector
$\vec v\in T_{e}G$ there are unique invariant vector fields $\vec R\in \mathcal R(G)$ and 
$L\in\mathcal L(G)$ such that $\vec R_e = \vec L_e = \vec v$ (see \cite{Mum} pp. 98--99).  %

\subsection{Algebraic Homogeneous spaces}

\begin{definition}
    Let $G$ be a $\mathcal C$-\emph{algebraic} group. 
A $G$-space $M$ is an algebraic variety over $\mathcal C$ endowed with an
  algebraic action of $G$,
  $$G\times_{\mathcal C} M \xrightarrow{a} M,\quad (\sigma,x)\mapsto \sigma\cdot x.$$
\end{definition}

  Let $M$ be a $G$-space. Then for each extension $\mathcal C\subset \mathcal K$, the
group $G(\mathcal K)$ acts on the set $M(\mathcal K)$. Therefore it is a $G(\mathcal K)$-set in 
the set theoretic sense. Given a point $x\in M$ its \index{isotropy subgroup} 
\emph{isotropy subgroup} is an \emph{algebraic} subgroup of $G$ 
that we denote by $H_x$. It is defined by equation $H_x\cdot x = x$. 
Note that it is not necessary for $x$ to be a rational point.

  The intersection of the isotropy subgroups of all \emph{closed} points of $M$ is
a normal \emph{algebraic} subgroup $H_M \triangleleft G$. The action
of $G$ is called \emph{faithful} if $H_M$ is the identity element
$\{e\}$, and it is called \emph{free} if for any rational point $x$, 
$H_x=\{e\}$. It is called \emph{transitive} if for each pair of
rational points $x,y\in M$ there is a $\sigma\in G$ such that
$\sigma\cdot x = y$; \emph{id est} there is only one orbit.

\begin{definition}
Let us consider the induced morphism,
$$(a\times Id)\colon G\times_{\mathcal C} M\to M\times_{\mathcal C} M,\quad
(\sigma,x)\mapsto (\sigma x, x)$$ then,
\begin{enumerate}
\item[(1)] $M$ is an homogeneous $G$-space if $(a\times Id)$ is
surjective.
\item[(2)] $M$ is a principal homogeneous $G$-space if $(a\times Id)$ is
an isomorphism.
\end{enumerate}
\end{definition}

If $\mathcal C$ is algebraically closed, an homogeneous $G$-space is simply a \emph{transitive}
$G$-space and a principal homogeneous $G$-space is a \emph{free and transitive}
$G$-space. In such case, any principal homogeneous $G$-space over is isomorphic to $G$.

\subsection{Existence of quotients: Chevalley's theorem}

Let $V$ ve a $\mathcal C$-vector space, and $GL(V)$ the group of linear 
transformations of $V$. It is an $\mathcal C$-algebraic group, and it
acts algebraically on any tensor space over $V$. Given a tensor $T$ we
call \emph{stabilizer subgroup} of $T$ to the group of linear
transformations $\sigma\in GL(V)$ for whom there exist a scalar
$\lambda\in C$ such that $\sigma(t) = \lambda T$.
In other words, the stabilizer subgroup of $T$ is the isotropy
subgroup of the line $\langle T \rangle$ spanned by $T$ in the
projectivization of the tensor space.

\begin{theorem}[Chevalley, see \cite{Humphreys} p. 80]\label{ThChevalley}
  Let $V$ be a $\mathcal C$-vector space of finite dimension, and let
$H\subset GL(V)$ be an algebraic subgroup. There exist a tensor,
$$T\in \bigoplus_i \left(V^{\otimes n_i}\otimes_{\mathcal C} \left(V^{\otimes m_i}\right)^*\right)$$
such that $H$ is the stabilizer of $T$,
  $$H = \{\sigma\in GL(V) |\langle\sigma(T)\rangle = \langle T \rangle \}$$
\end{theorem}

  From this result we obtain that for a linear algebraic group $G$ and
an algebraic subgroup $H$, the quotient space $G/H$ is isomorphic to
the orbit $O_{\langle T\rangle}$ in the projective space $\mathbb
P\left(\bigoplus_i \left(V^{\otimes n_i}\otimes_C \left(V^{\otimes
m_i}\right)^*\right)\right)$. It is a quasiprojective algebraic
variety. 

  There is a lack in the literature of an existence
theorem for arbitrary quotients of an non-linear algebraic group
over an arbitrary field. However, there is 
a result, due to M. Rosenlicht \cite{Rosenlicht1963}, saying
that for any action of an algebraic group $G$ on an
algebraic variety $V$, there exist a $G$-invariant open
subset $U\subset V$ such that the geometrical quotient
$U/G$ in the sense of Mumford exists. In the case of 
a subgroup $G'$ acting on $G$, this open subset 
must be right-invariant, and then it coincides with $G$.

\subsection{Galois Cohomology}

  In this section, we assume that $\mathcal C$ is a \emph{perfect field}; note that this holds
if $\mathcal C$ is of characteristic zero, which is the case we are interested in. 
In such case, any algebraic extension can be embedded into a Galois
extension. Therefore, the algebraic closure $\bar{\mathcal C}$
is the inductive limit of all Galois extensions of $\mathcal C$. The group of $\mathcal C$-automorphisms 
of $\bar{\mathcal C}$ is then identified with the projective limit of all Galois groups, of
algebraic extensions of $\mathcal C$. With the initial topology of the family
of projections onto finite Galois groups, this is a compact totally disconnected
group, that we denote $\Gal(\bar{\mathcal C}/\mathcal C)$. 

 Let $G$ be a $\mathcal C$-algebraic group. The group of automorphisms acts on $G(\bar{\mathcal C})$ by
composition. Let us consider $\mathbf G^k$ the set of continuous maps from
$\Gal(\bar{\mathcal C}/ \mathcal C)^k$ onto $G(\bar{\mathcal C})$. 
In such case $\mathbf G^0 = G(\bar{\mathcal C})$.
We consider the sequence:
\begin{equation}\label{GCsequence}
0 \to \mathbf G^0 \xrightarrow{\delta_0} \mathbf G^1 \xrightarrow{\delta_1} \mathbf G^2,
\end{equation} 
where the codifferential of $x\in \mathbf G^0$ is $(\delta_0 x)(\sigma) = x^{-1}\cdot
\sigma(x)$, and the codifferential of $\varphi\in \mathbf G^1$ is 
$(\delta_1\varphi)(\sigma,\tau) = \varphi(\sigma\cdot\tau)^{-1}\cdot\varphi(\sigma)\cdot\sigma(\varphi(\tau))$.
An element in the image of $\delta_0$ is called a \emph{coboundary}, the set of coboundaries is
denoted by $B^1(G,C)$. An element $\varphi\in\mathbf G^1$ is called a \emph{$1$-cocycle} if $\delta_1\varphi$ vanish.
The set of $1$-cocycles is denoted $Z^1(G,\mathcal C)$. Two $1$-cocycles are called 
\emph{cohomologous} if there is $x\in \mathbf G^0$ such that $\varphi(\sigma) = x^{-1}\cdot \psi(\sigma) \cdot
\sigma(x)$. This is an \emph{equivalence relation} in $Z^1(G,\mathcal C)$. The quotient
set $\mathbf Z^1(G,\mathcal C)/\sim$ is a pointed set, with distingished point the class of \emph{coboundaries}. 
Note that when $G$ is an abelian group the sequence \eqref{GCsequence}
is a \emph{differential complex} and this quotient is the first cohomology group. 

\begin{definition}
  The zero Galois cohomology set of $G$ with coefficients in $\mathcal C$, $H^0(G,\mathcal C)$ is the kernel of
$\delta_0$. It is a pointed set with distinguised point the identity. The first Galois 
cohomology set of $G$ with coefficients in $\mathcal C$, $H^1(G,\mathcal C)$, is the pointed set 
$Z^1(G,\mathcal C)/\sim$.
\end{definition}
 
  From the definition of $\delta_0$ it is clear that $x\in H^0(G,\mathcal C)$ if an only if
it is invariant under the action of $\Gal(\bar{\mathcal C}/\mathcal C)$. The fixed field of $\bar{\mathcal C}$ in
precisely $\mathcal C$, therefore the zero Galois cohomology set coincides with the set
of $\mathcal C$-points $G(\mathcal C)$. Therefore, we define the zero cohomology set $H^0(V,\mathcal C)$ of any
$\mathcal C$-algebraic variety $V$ to be the set of $\mathcal C$-points $V(\mathcal C)$.

  Let $G'$ be an algebraic subgroup of $G$. In such case $H^0(G/G',\mathcal C)$ is a pointed
set, with distinguised point the class of the identity. An element $x\in H^0(G/G',\mathcal C)$
is a $\mathcal C$-point of the homogeneous space $G/G'$. This $x$ is the class of a unique 
$\bar{\mathcal C}$-point $\bar x$ of $G$. The \emph{coboundary} $\partial_0 \bar x$ is a cocycle
in $G'$, and its cohomology class $[\bar x]\in H^1(G',\mathcal C)$ does not depends on the election
of $x$. We have a morphism of pointed sets $H^0(G/G',\mathcal C)\to H^1(G,\mathcal C)$ called the \emph{connecting
morphism}.  We obtain an exact sequence of pointed sets:
$$0 \to H^0(G',\mathcal C) \to H^0(G,\mathcal C) \to H^0(G/G', \mathcal C) \to H^1(G',\mathcal C) \to H^1(G,\mathcal C)$$
and when $G'$ is a normal subgroup of $G$, the sequence
$$H^1(G',\mathcal C) \to H^1(G,\mathcal C) \to H^1(G/G', \mathcal C)$$
is also exact (see \cite{Ko1}, p. 277--288).

Using the previous exact sequence it is relatively easy to compute the first
Galois cohomology set of several algebraic groups. We say that the first cohomology
set of $G$ with coefficients in $\mathcal C$ \emph{vanish} if it consists of an only point. 
In particular the following results
are well known:
\begin{itemize}
\item The first cohomology set of the additive group $H^1((\mathcal C,+),\mathcal C)$ vanish.
\item The first cohomology set of the multiplicative group $H^1(\mathcal C^*,\cdot), \mathcal C)$ vanish.
\item $H^1(GL(n,\mathcal C),\mathcal C)$ vanish.
\item $H^1(SL(n,\mathcal C),\mathcal C)$ vanish.
\item If $G$ is linear connected solvable group then $H^1(G,\mathcal C)$ vanish.
\item If $\mathcal C$ is algebraically closed then for any algebraic group $H^1(G, {\mathcal C})$ vanish.
\item If $S$ is a Riemann surface and $\mathcal M(S)$ is its field of meromorphic
      function then for any linear connected $\mathcal M(S)$-algebraic group $G$, $H^1(G, \mathcal M(S))$ vanish
      (this is a particular case of fields of dimension lower or equal than one, treated in
      \cite{Serre}).
\item If $S$ is an open Riemann surface then for any connected $\mathcal M(S)$-algebraic 
      group $H^1(G, \mathcal M(S))$ vanish (Grauert theorem, see \cite{Sibuya}).
\end{itemize}

The first Galois cohomology set classifies the \emph{principal homogeneous spaces
over $G$.} This classification was first obtained by Ch\^atelet for some
particular cases, here we follow Kolchin \cite{Ko1} (see p. 281--283). The main fact is that 
if the first Galois cohomology set vanish then all principal homogeneous spaces have
rational points. 

\begin{theorem}
  Let $G$ be a $\mathcal C$-algebraic group and $M$ a principal homogeneous $G$-space. Then
$M$ defines a class $[M]$ in $H^1(G,\mathcal C)$. This cohomology class classifies $M$ up to $\mathcal C$-isomorphisms.
$M$ is isomorphic to $G$ if and only if $[M]$ is the distinguised point of $H^1(G,\mathcal C)$. Reciprocally
any cohomology class of $H^1(G,\mathcal C)$ is the class of certain homogeneous $G$-space. 
\end{theorem}

\subsection{Fundamental Fields}

  Consider a right invariant vector field $A\in \mathcal R(G)$. Then,
$\vec A\otimes 1$ is a regular vector field in $G\times_{\mathcal C} M$. This vector field
is projectable by the action of $G$ in $M$,
$$a\colon G\times_{\mathcal C} M \to M, \quad \vec A \otimes 1 \mapsto \vec A^M.$$

\begin{definition}\label{DFfundamentalfield}
  We call algebra of fundamental field $\mathcal R(G,M)$ to the Lie algebra
of regular vector fields in $M$ spanned by the projections  $\vec A^M$ of vector 
fields $\vec A\otimes 1$, being $\vec A$  right invariant vector field in $G$. 
\end{definition}

  There is a canonical surjective Lie algebra morphism, 
$$\mathcal R(G)\to \mathcal R(G,M), \quad \vec A \to \vec A^M,$$
the kernel of this morphism is the Lie algebra of the kernel of the
action $H_M$, $\mathcal R(H_M)\subset \mathcal R(G)$. In particular, the Lie 
algebra of fundamental fields $\mathcal R(G,G)$ in $G$ coincides with $\mathcal R(G)$.

\section{Differential Algebraic Geometry}

 We can state that the differential algebraic geometry is with
respect to the differential algebra the same than the classical
algebraic geometry is with respect to the commutative algebra. In this sense, the
differential algebraic geometry is the study of geometric objects
associated with differential rings. Here we present the theory 
of schemes with derivations, which has been developed by Buium \cite{Bu}, 
and the theory of differential schemes, which is due to Keigher \cite{Ke1, Ke2},
Carra' Ferro (see \cite{Ca0}), and Kovacic \cite{Kov1}.

\subsection{Differential Algebra}
  We present here some preliminaries in differential algebra. 
The main references for this subject are \cite{Ritt1950}, 
\cite{Ka}, \cite{Ko1}. 

  A differential ring is a commutative ring $\mathcal A$
and a derivation $\partial_{\mathcal A}$. By a derivation we mean an application
verifying the Leibnitz rule, 
$\partial_{\mathcal A}(ab) = a\cdot\partial_{\mathcal A}(b) + b\cdot\partial_{\mathcal A}(a)$.
An element $a\in \mathcal A$ is called a constant if it has vanishing
derivative $\partial a = 0$. Whenever it does not lead to confusion, we will write $\partial$
instead of $\partial_{\mathcal A}$. The subset $C_{\mathcal A}$ of 
constants elements is a subring of $\mathcal A$. When $\mathcal A$ is a field we call
it a \emph{differential field}. In such a case, the constant ring
$C_{\mathcal A}$ is a subfield of $\mathcal A$.
An ideal $\mathfrak I\subset \mathcal A$ is a differential ideal if
$\partial (\mathfrak I)\subset \mathfrak I$.

  Note that if $\mathfrak I$ is a differential ideal, then the
quotient $\mathcal A/\mathfrak I$ is also a differential ring. For a
subset $S\subset\mathcal A$ we denote $[S]$ for the smallest
differential ideal containing $S$, and $\{S\}$ for the smallest
radical differential ideal containing $S$. For an ideal $\mathfrak
I\subset \mathcal A$ we denote $\mathfrak I'$ for the smallest
differential ideal containing $\mathfrak I$, namely:
$\mathfrak I' = \sum_i \partial^i(\mathfrak I).$ Localization by
arbitrary multiplicative sytems is also suitable in differential rings.
A ring morphism is called \emph{differential} if it is compatible with
the derivation. In the category of differential rings, tensor product is also
well defined. 

%
%  
%  A ring morphism $\psi\colon \mathcal A\to\mathcal B$ 
%is called a differential ring morphism if
%$\psi\circ\partial_{\mathcal A} = \partial_{\mathcal B}\circ\psi$.
%  Whenever it does not lead to confusion we denote by the same symbol $\partial$
%the derivations in $\mathcal A$ and $\mathcal B$. If $\psi$ is a differential 
%ring morphism, then for each differential ideal $\mathfrak I\subset \mathcal B$, the preimage
%$\psi^{-1}(\mathfrak I)\subset \mathcal A$ is also a differential
%ideal.
%

  Consider $\mathcal K$ a differential field. A differential ring
$\mathcal A$ endowed with a morphism $\mathcal
K\hookrightarrow\mathcal A$ is called a \emph{differential $\mathcal
K$-algebra}. If $\mathcal A$ is a differential field then we say
that it is a \emph{differential extension} of $\mathcal K$.
\index{differential!$\mathcal
K$-algebra}\index{differential!extension}

%%
%
%  Derivations extend to localized rings in a canonical way. 
%Let $S\subset \mathcal A$ be a multiplicative system. The localized
%ring $S^{-1}\mathcal A$ admits a unique derivation such that the
%localization morphism is a morphism of differential rings. This
%derivation is defined as follows:
%$$\partial\colon S^{-1}\mathcal A \to S^{-1}\mathcal A
%,\quad \frac{a}{b}\mapsto \frac{\partial a\cdot b + a \cdot\partial
%b}{b^2}.$$
%
%%
%
%  There is also a unique way to extend derivations to the
%tensor product over a differential field.  Let us consider a differential 
%field $\mathcal K$ and two
%differential $\mathcal K$-algebras $\mathcal A$, $\mathcal B$. Using
%the Leibnitz rule we define a derivation in the tensor product
%$\mathcal A\otimes_{\mathcal K}\mathcal B$ in the following 
%way $\partial(a\otimes b) = \partial a \otimes b + a \otimes \partial b$.
%This is the only derivation of $\mathcal A \otimes_{\mathcal
%K}\mathcal B$ such that the canonical injections $\mathcal A \to
%\mathcal A \otimes_{\mathcal K}\mathcal B$ and $\mathcal B \to
%\mathcal A \otimes_{\mathcal K} \mathcal B$ are morphisms of
%differential $\mathcal K$-algebras. Hence, it occurs that the tensor
%product is the direct sum in the category of differential algebras.
%

\subsection{Keigher Rings}

  If $\mathfrak I\subset \mathcal A$ is an ideal, we denote 
its radical ideal by $\sqrt{\mathfrak I}$, 
the intersection of all prime ideals containing $\mathfrak I$. In algebraic geometry,
there is a one-to-one correspondence between the set of radical
ideals of $\mathcal A$ and the set of Zariski closed subsets of
$\Spec(\mathcal A)$, the prime spectrum of $\mathcal A$. In order to
perform an analogous systematical study of the set of differential
ideals - \emph{id est} differential algebraic geometry - we should
require radicals of differential ideals to be also differential
ideals. This property does not hold in the general case. We have to
introduce a suitable class of differential rings. This class was
introduced by Keigher (see \cite{Ke1}); we call them Keigher rings.

\begin{definition}
 A Keigher ring is a differential ring verifying that for each
differential ideal $\mathfrak I$, its radical $\sqrt{\mathfrak I}$
is also a differential ideal.
\end{definition}

\begin{definition}
  For any ideal $\mathfrak I\subset \mathcal A$ we define its differential
  core as
  $\mathfrak I_\sharp = \{a\in\mathfrak I\colon \forall n(\partial^na\in\mathfrak
  I)\}$.
\end{definition}

  Keigher rings can be defined in several equivalent ways. The
following theorem of characterization includes different possible
definitions (see \cite{Kov1}, proposition 2.2.).

\begin{theorem}\label{C2THE2.1.6}
Let $\mathcal A$ be a differential ring.
  The following are equivalent:
\begin{enumerate}
\item[(a)] If $\mathfrak p\subset\mathcal A$ is a prime ideal, then $\mathfrak
p_\sharp$ is a prime differential ideal.
\item[(b)] If $\mathfrak I\subset\mathcal A$ is a differential ideal, and $S$ is a
multiplicative system disjoint from $\mathfrak I$, then there is a
prime maximal differential ideal containing $\mathfrak I$ disjoint
with $S$.
\item[(c)] If $\mathfrak I\subset\mathcal A$ is a differential ideal, then so is $\sqrt{\mathfrak
I}$.
\item[(d)] If $S$ is any subset, then $\{S\} = \sqrt{[S]}$.
\item[(e)] $\mathcal A$ is a Keigher ring.
\end{enumerate}
\end{theorem}

  By a \emph{Ritt algebra} we mean a differential ring including
the field $\mathbb Q$ of rational numbers. When studing differential
equations in characteristic zero, differential rings considered are
mainly Ritt algebras. A main property of Ritt algebras is that
the radical of a differential ideals is a differential ideal (see for
instance \cite{Ka}), therefore \emph{Ritt algebras are Keigher rings}.

\begin{proposition}
  If $\mathcal A$ is a Keigher ring then for any differential ideal
$\mathfrak I$, $\mathcal A/\mathfrak I$ is Keigher and for any
multiplicative system $S$, $S^{-1}\mathcal A$ is Keigher.
\end{proposition}

\begin{proof}
  Assume $\mathcal A$ is Keigher.
  First, let us prove that $\mathcal A/\mathfrak I$ is Keigher. Consider the projection
  $\pi\colon\mathcal A\to\mathcal A/\mathfrak I$. Let $\mathfrak a$ be
a differential ideal of $\mathcal A / \mathfrak I$. Then
$\sqrt{\mathfrak a} = \pi(\sqrt{\pi^{-1}(\mathfrak a)})$ is a
differential ideal.

  Second, consider a localization morphism $l\colon \mathcal A \to S^{-1}\mathcal A$. Let
$\mathfrak a\subset S^{-1}\mathcal A$ be a differential ideal. Let
us denote by $\mathfrak b$ the preimage $l^{-1}(\mathfrak a)$; it is
a differential ideal and $l(\mathfrak b)\cdot S^{-1}\mathcal A =
\mathfrak a$.

Let us consider $\frac{a}{s}\in\sqrt{\mathfrak a}$.
$\frac{a}{s}\frac{s}{1} = \frac{a}{1}\in\sqrt{\mathfrak a}$. For
certain $n$, hence $\frac{a^n}{1}\in\mathfrak a$, $a^n\in\mathfrak
b$ and $a\in\sqrt{\mathfrak b}$. $\mathcal A$ is Keigher, and then
$\partial a\in\sqrt{\mathfrak b}$. Therefore $(\partial a)^m\in
\mathfrak b$, so that $\left(\frac{\partial
a}{1}\right)^m\in\mathfrak a$ and, for instance, $\frac{\partial
a}{1}\in\sqrt{\mathfrak a}$. Finally,
$$\partial\left(\frac{a}{s}\right) = \frac{\partial a}{1}\frac{1}{s} -
\frac{a}{1}\frac{\partial s}{s^2}\in\sqrt{\mathfrak a},$$ and by
\emph{(c)} of Theorem \ref{C2THE2.1.6} $S^{-1}\mathcal A$ is
Keigher.
\end{proof}

\subsection{New Constants}

  From now on let $\mathcal K$ be a differential field, and let
$\mathcal C$ be its field of constants. We assume that $\mathcal C$
is algebraically closed. A classical lemma of differential algebra
(see \cite{Ko1} p. 87 Corollary 1) says that if $\mathcal A$ is a
differential $\mathcal K$-algebra, then the ring of constant
$C_{\mathcal A}$ is linearly disjoint over $\mathcal C$ with
$\mathcal K$. Let us set this classical lemma in a more geometric
frame.

\begin{lemma}\label{LmDisjoint}
  Let $\mathcal A$ be an integral finitely generated
differential $\mathcal K$-algebra. Then there is an affine subset
$U\subset \Spec(\mathcal A)$ such that the ring of constants
$C_{\mathcal A_U}$ is a finitely generated algebra over $\mathcal
C$.
\end{lemma}

\begin{proof}
  Consider $Q(\mathcal A)$ the field of fractions of $\mathcal A$.
The extension $\mathcal K \subset Q(\mathcal A)$ is of finite
transcendency degree. Then, $\mathcal K \subset \mathcal K \cdot
C_{Q(\mathcal A)} \subset Q(\mathcal A$) are extensions of finite
transcendency degree, and there are $\lambda_1,\ldots\lambda_s$ in
$C_{Q(\mathcal A)}$ such that $\mathcal
K(\lambda_1,\ldots,\lambda_s) = \mathcal K\cdot C_{Q(\mathcal A)}$.
Constants $\lambda_1$,$\ldots$,$\lambda_s$ are fractions
$\frac{f_i}{g_i}$. Consider the affine open subset obtained by
removing from $\Spec(\mathcal A)$ the zeroes of the denominators,
$$U = \Spec A \setminus \bigcup_{i=1}^s (g_i)_0.$$
Then, $\lambda_i\in \mathcal A_U$ and $\mathcal K[C_{\mathcal A_U}]
= \mathcal K[\lambda_1,\ldots,\lambda_s]$. We will prove that
$C_{\mathcal A_U} = \mathcal
C[\lambda_1,\ldots,\lambda_s]$. Let $\lambda\in C_{\mathcal A_U}$.
It is certain polynomial in the variables $\lambda_i$ with
coefficients in $\mathcal K$:
$$\lambda = \sum_{I\in\Lambda}a_I\lambda^I, \quad a_I\in \mathcal
K;$$ where $\Lambda$ is a suitable finite set of
multi-indices. We can take this set in such way that the
$\{\lambda^{I}\}_{I\in\Lambda}$ are linearly independent over
$\mathcal K$, and then so they are over $\mathcal C$.
$\{\lambda,\lambda^I\}_{I\in\Lambda}$ is a subset of $\mathcal
K$-linearly dependents elements of $C_{\mathcal A_U}$. By \cite{Ko1}
(p. 87 corollary 1) then they are $\mathcal C$-linearly dependent.
Hence, $\lambda$ is $\mathcal C$-linear combination of
$\{\lambda_I\}_{I\in\Lambda}$, $\lambda\in \mathcal
C[\lambda_1,\ldots,\lambda_s]$ and finally $C_{\mathcal A_U} =
\mathcal C[\lambda_1,\ldots,\lambda_s]$.
\end{proof}

\subsection{Differential Spectra}

\index{differential!spectrum}
\begin{definition}
Let $\mathcal A$ be a differential ring. 
$\DiffSpec(\mathcal A)$ is the set of all prime differential ideals
 $\mathfrak p\subset \mathcal A$.
\end{definition}

Let $S\subset\mathcal A$ any subset. We define the differential
locus of zeroes of $S$, $\{S\}_{0}\subset\DiffSpec(\mathcal A)$ as
the subset of prime differential ideals containing $S$. This family
of subsets define a topology (having these subsets as closed
subsets), that we call the \emph{Kolchin topology} or
\emph{differential Zariski topology.} Note that $\{S\}_0 = (S)_0\cap
\DiffSpec(\mathcal A)$. From that if follows:

\begin{proposition}
  $\DiffSpec(\mathcal A)$ with Kolchin topology is a topological subspace of
$\Spec(\mathcal A)$ with Zariski topology.
\end{proposition}

From now on, let us consider the following notation: $X =
\Spec(\mathcal A)$, and $X' = \DiffSpec(\mathcal A)$.

  Let us recall that a topological space is said \emph{reducible}
if it is the non-trivial union of two closed subsets. It is said
\emph{irreducible} if it is not reducible. A point of an irreducible
topological space is said \emph{generic} if it is included in each
open subset. The following properties of the differential spectrum
are proven in \cite{Ke2} (see Proposition 2.1).

\begin{proposition}
$X'$ verifies:
\begin{enumerate}
\item[(1)] $X'$ is quasicompact.
\item[(2)] $X'$ is $T_0$ separated.
\item[(3)] Every closed irreducible subspace of $X'$ admits a unique generic
  point. The map $X'\to 2^{X'}$, that maps each point $x$ to its Kolchin closure $\overline{\{x\}}$
  is a bijection between points of $X'$ and irreducible closed subspaces of $X'$.
\end{enumerate}
\end{proposition}

 Here we review some of the topological properties of
the differential spectrum of Keigher rings. 

\begin{lemma}\label{LM3.8}
Assume that $\mathcal A$ is a Keigher ring. Then each minimal prime
ideal is a differential ideal.
\end{lemma}

\begin{proof}
  Then, let $\mathfrak p$ be a minimal prime. By Theorem \ref{C2THE2.1.6} (a),
$\mathfrak p_\sharp$ is a prime differential ideal and $\mathfrak
p_\sharp\subseteq \mathfrak p$.
\end{proof}

\begin{proposition}
Assume that $\mathcal A$ is Keigher. Then, $X$ is an irreducible
topological space if and only if $X'$ is an irreducible topological
space.
\end{proposition}

\begin{proof}
  Just note that the irreducible components of $X'$ are the Kolchin closure of
  minimal prime ideals of $\mathcal A$.
\end{proof}

\begin{proposition}
Assume $\mathcal A$ is Keigher. If $X'$ is connected, then $X$ is
connected.
\end{proposition}

\begin{proof}
  Assume that $X = Y \sqcup Z$, then we have an isomorphism of rings
  $$(p_1,p_2)\colon\mathcal A \mapsto \mathcal O_X(Y) \times \mathcal O_X(Z), \quad
  a\mapsto (a|_X,a|_Y),$$
  the kernel of each restriction $p_i$ is intersection of minimal prime ideals, so
  by Lemma \ref{LM3.8} they are differential ideals. Hence, the rings $\mathcal O_X(Y)$ and $\mathcal O_X(Z)$ are also
  differential rings. Then,
  $$X' = Y' \sqcup Z',$$
  being $Y' = \DiffSpec(\mathcal O_X(Y))$, $Z' = \DiffSpec(\mathcal
  O_X(Z))$. We have proven that if $X$ disconnects, then $X'$
  disconnects.
\end{proof}

\subsection{Structure Sheaf}

  We define the structure sheaf $\mathcal O_{X'}$ as in \cite{Kov1}.
Let us consider the projection,
   $$\pi\colon\bigsqcup_{x\in X'} \mathcal A_x \to X'.$$
being $\bigsqcup_{x\in X'} \mathcal A_x$ the disjoint union of all
the localized rings $\mathcal A_x$. We say that a section $s$ of
$\pi$ defined in an open subset $U\subset X'$ is \emph{a regular
function} if it verifies the following: for all $x\in U$ there exist
an open neighborhood $x\in U_x$ and $a,b\in \mathcal A$ with
$b(x)\neq 0$ $(b\not\in x)$, such that for all $y\in U_x$ with
$b(y)\neq 0$, $s(y) = \frac{a}{b}\in \mathcal A_y$. Thus, a regular
function is a section which is locally representable as a quotient.
We write $\mathcal O_{X'}$ for the sheaf of regular functions in
$X'$. By the above construction we can state:

\begin{proposition}
  The stalk $\mathcal O_{X',x}$ is a ring isomorphic to $\mathcal
  A_x$.
\end{proposition}

\begin{theorem}
  Let us consider the natural inclusion $j\colon X'\hookrightarrow
  X$. The sheaf of regular functions  $\mathcal O_{X'}$ is the restriction $\mathcal
  O_X|_{X'}$ of the sheaf of regular function in $X$.
\end{theorem}

\begin{proof}
First, let us define a natural morphism of presheaves of rings on
$X'$ between the inverse image presheaf $j^{-1}\mathcal O_{X}$ and
$\mathcal O_{X'}$. Let us consider an open subset $U\subset X'$ and
a section $s$ of the presheaf $j^{-1}\mathcal O_X$ defined in $U$. By
definition of inverse image, there is an open subset $W$ of $X$ such
that $W\cap X' \cap U$ and for what $s$ is written as a fraction
$\frac{a}{b}\in\mathcal A_W$. This fraction is a section of
$\mathcal O_{X'}(U)$, and it defines the presheaf morphism
$$j^{-1}\mathcal O_X \to \mathcal O_{X'}.$$
This presheaf morphism induces a morphism between associated sheaves
$\mathcal O_X|_{X'}$ and $\mathcal O_{X'}$. It is clear that this
natural morphism induce the identity between fibers $\mathcal
(j^{-1}\mathcal O_X)_x = \mathcal A_x \to \mathcal O_{X',x} =
\mathcal A_x$, and then it is an isomorphism.
\end{proof}

\subsection{Global Sections}

  One of the main facts of the differential algebraic
geometry is that the ring of global regular sections of $X'$
does not coincide with the differential ring $\mathcal A$. Of course
there is a canonical morphism from $\mathcal A$ to $\mathcal O_{X'}(X')$.
However there are non-vanishing elements giving rise to the zero section and
non invertible elements giving rise to invertible sections. An element $a$
of $\mathcal A$ is called a \emph{differential zero} if its annihilator ideal is
not contained in any proper differential ideal. The set of 
differential zeroes is denoted by $\mathfrak Z$. An element is called
a differential unit if it is not contained in any proper differential ideal.
The set of differential units is denoted by $\mathfrak U$. Then, there
is a canonical \emph{injective} morphism,
$\mathfrak U^{-1}\mathcal A/\mathfrak Z \hookrightarrow \mathcal O_{X'}(X').$
But in general this morphism is not surjective, \emph{id est}, there are regular
functions that are not representable as fractions of $\mathcal A$. Therefore,
the differential spectrum of $\mathcal O_{X'}(X')$ is not always isomorphic to
$X'$. This problem is extensively discussed in \cite{Be2008}.

\subsection{Differential Schemes}

  The study of differential schemes started within
the work of Keigher  \cite{Ke1, Ke2} and was continued by Carra'
Ferro \cite{Ca0}, Buium  \cite{Bu} and Kovacic \cite{Kov1}.
Definitions are slightly different in each author approach, here we
follow Kovacic.

  Let us remind that a \emph{locally ringed space} is a topological space
$X$ endowed with an structure sheaf of rings $\mathcal O_X$ such that for all
$x\in X$ the stalk $\mathcal O_{X,x}$ is a local ring. Thus, a 
\emph{locally differential ringed space} is a locally ringed space whose
structure sheaf $\mathcal O_X$ is a sheaf of differential rings. A morphism 
of locally differential ringed spaces $f\colon X\to Y$ consist of a continous
map together with a sheaves morphism $f^\natural\colon \mathcal O_X \to f_*\mathcal O_Y$.
For the differential ring $\mathcal A$ it is clear that its differential spectrum
$X'$ endowed with the structure sheaf $\mathcal O_{X'}$ is a locally differential 
ringed space. 

\begin{definition}
  An affine differential scheme is a locally differentially ringed space $X$
  which is isomorphic to $\DiffSpec(\mathcal A)$ for
  some differential ring $\mathcal A$.
\end{definition}

\begin{definition}
  A differential scheme is a locally differentially ringed space
  $X$ in which every point has a neighborhood
  that is an affine differential scheme.
\end{definition}

\begin{remark} 
Schemes are differential schemes, endowed with the
trivial derivation. The category of differential schemes is an
extension of the category of schemes, in the same way that the
category of differential rings is an extension of the category of
rings.
\end{remark}

%  We say that a differential scheme is \emph{irreducible},
%or \emph{connected} if it is irreducible, or connected, as
%topological space. We say that it is \emph{reduced} if its sheaf of
%regular functions is a sheaf of reduced rings, and we say that it is
%\emph{noetherian} if it is quasicompact and the sheaf of regular
%functions is a sheaf of noetherian rings.
%
%

  By a \emph{morphism of differential schemes} $f\colon X\to Y$ we mean a
morphism of locally ringed spaces, such that $f^{\sharp}\colon \mathcal O_Y \to
f_*\mathcal O_X$ is a morphism of sheaves of differential rings.

  Let $\mathcal K$ be a differential field. A \emph{$\mathcal K$-differential
scheme} is a differential scheme $X$ provided with a morphism $X\to
\DiffSpec(\mathcal K)$, it means that $\mathcal O_X$ is a sheaf of
differential $\mathcal K$-algebras.

  A morphism of differential schemes $f\colon X \to Y$ between two
differential $\mathcal K$-schemes is a \emph{morphism of
differential $\mathcal K$-schemes} if the sheaf morphism $f^\sharp
\colon \mathcal O_Y \to f_*\mathcal O_X$ is a morphism of sheaves of
differential $\mathcal K$-algebras.

\subsection{Product of Differential Schemes}

  There is not a direct product in the category of differential
schemes relative to a given basic differential scheme. This problem
is discussed in \cite{Kov1}. However, in the case of differential
schemes over a differential field $\mathcal K$ we can construct
the direct product by patching tensor products, as it is usually
done in algebraic geometry. Therefore,
$$\DiffSpec(\mathcal A)\times_{\mathcal K}\DiffSpec(\mathcal B) =
\DiffSpec(\mathcal A\otimes_{\mathcal K}\mathcal B).$$
Moreover, if $X$ and $Y$ are reduced differential $\mathcal K$-schemes
then $X\times_{\mathcal K} Y$ is also reduced (see \cite{Kov2} Proposition
25.2).  

\subsection{Split of Differential Schemes}

\begin{definition}
  Let $X$ be a differential scheme. Define the presheaf of rings $\mathcal C_X$ on
$X$ by the formula,
  $$\mathcal C_X(U) = C_{\mathcal O_X(U)},$$
for any open subset $U\subseteq X$.
\end{definition}

  From this definition it follows that $\mathcal C_X$ is a sheaf
of rings and its fiber $\mathcal C_{X,x}$ is isomorphic to
the ring of constants $\mathcal C_{\mathcal O_{X,x}}$.  In particular,
if $X$ is a $\mathcal K$-differential scheme $\mathcal C_X$ is a sheaf
of $\mathcal C_{\mathcal K}$-algebras. 

\begin{definition}
  We call space of constants of $X$, $\Const(X)$ to the locally
ringed space $(X,\mathcal C_X)$.
\end{definition}

\begin{definition}
  We say that $X$ is an almost-constant differential scheme if
its space of constants $\Const(X)$ is a scheme.
\end{definition}

  Let $X$ be an almost-constant scheme. Then, each open subset
$U\subset X$ is also almost-constant. If $Y$ is a reduced closed
subscheme of $X$ then $Y$ is almost-constant. In this way if $Y$ is
a locally closed reduced subscheme of $X$, then $Y$ is
almost-constant.

Let $\mathcal K$ be a differential field, and $\mathcal C$ its field
of constants.

\begin{definition}\label{C2DEFsplitDS}
  A differential $\mathcal K$-scheme $X$ splits if there is a
$\mathcal C$-scheme $Y$ and an isomorphism of $\mathcal
K$-differential schemes,
$$\phi\colon X \xrightarrow{\sim} Y \times_{\mathcal C} \DiffSpec(\mathcal K).$$
The isomorphism $\phi$ is called an splitting isomorphism for $X$.
\end{definition}

%\begin{proposition}\label{Kov2.27.1}
%  Let $X = \DiffSpec(\mathcal A)$. Suppose that $\mathcal A$ is reduced.
%Then $\mathcal A$ is almost-constant if and only if $\Const(X)$ is
%an affine scheme. In such case $\Const(X) \simeq \Spec(\mathcal
%C_A)$.
%\end{proposition}
%
%\begin{proof}
%\cite{Kov2} proposition 27.1.
%\end{proof}

\begin{proposition}\label{Kov28.2}
  If $X$ is reduced and splits, then it is almost-constant and
$$X \xrightarrow{\sim} \Const(X)\times_{\mathcal C} \DiffSpec(\mathcal K).$$
\end{proposition}

\begin{proof}
\cite{Kov2} proposition  28.2.
\end{proof}

%\begin{proposition}
%  Suppose that a differential $\mathcal K$-algebra
%$\mathcal A$ is reduced and almost-constant. Then
%$\DiffSpec(\mathcal A)$ splits if and only if for all
%$x\in\DiffSpec(\mathcal A)$:
%$$\mathcal A_{x} = \mathcal K[C_{\mathcal A}]_{x\cap C_{\mathcal A}}$$
%\end{proposition}
%
%\begin{proof}
%\cite{Kov2} proposition 28.3.
%\end{proof}

\subsection{Strongly Normal Extensions}

  Strongly normal extensions are introduced by Kolchin \cite{Ko0}.
They are differential field extensions whose group of automorphisms
admits an structure of algebraic group. This notion has been
recently characterized in terms of differential schemes by
Kovacic \cite{Kov3}. This characterization is more convenient
for our presentation of differential Galois theory, so that
we will use it as a new definition. 

\begin{definition}
$\mathcal K\to\mathcal L$ is a strongly normal extension if and only
if the differential scheme  $\DiffSpec(\mathcal L\otimes_{\mathcal
K} \mathcal L)$ splits. In such case denote $\Gal(\mathcal
L/\mathcal K)$ to the scheme $\Const(\DiffSpec(\mathcal
L\otimes_{\mathcal K}\mathcal L))$.
\end{definition}

  Note that prime differential ideals of $\mathcal L\otimes_{\mathcal
K} \mathcal L$ whose quotient field is $\mathcal L$, correspond to 
$\mathcal K$-automorphisms of $\mathcal L$.
If $\sigma$ is a $\mathcal K$-automorphism of $\mathcal L$, the kernel 
of the differential $\mathcal K$-algebra morphism,
$$\mathcal L \otimes_{\mathcal K}\mathcal L \to \mathcal L,\quad a\otimes b \mapsto a\sigma(b),$$
is a prime differential ideal $\mathfrak p_{\sigma}$. Then, the set of rational
points of $\DiffSpec(\mathcal L\otimes_{\mathcal
K} \mathcal L)$ is naturally endowed with a group structure. This group
structure descent to a structure of $\mathcal C$-algebraic group precisely
when  $\DiffSpec(\mathcal L\otimes_{\mathcal
K} \mathcal L)$ splits. In such case the space of constant $\Gal(\mathcal L/\mathcal K)$ is
endowed with an structure of algebraic group. This problem is axhaustively treated in \cite{Kov3}.

%\begin{definition}
%We say that a strongly normal extension $\mathcal K \to \mathcal L$
%is a Picard-Vessiot extension if $\Gal(\mathcal L/\mathcal K)$ is an
%affine scheme.
%\end{definition}

  This approach gives us a parallelism with Galois extensions in classical
theory of fields. Note that a field extension $k\to K$ is a Galois
extension if and only if $\Spec(K\otimes_k K) = G \times_k \Spec(K)$
(see \cite{Sa}).
  We also obtain the scheme structure of the Galois group:
it is the scheme of constants of $\DiffSpec(\mathcal
L\otimes_{\mathcal K}\mathcal L)$.

\subsection{Galois Correspondence for Strongly Normal Extensions}

Let us consider as above $\mathcal K\subset\mathcal L$ a strongly
normal extension of differential fields. To each subgroup $H\subset
\Gal(\mathcal L/\mathcal K)$ we assign the intermediate extension
$\mathcal K\subset \mathcal L^H\subset \mathcal L$ of
$H$-invariants. Reciprocally to each intermediate extension
$\mathcal K \subset \mathcal F \subset \mathcal L$ we assign the
subgroup $\Gal(\mathcal L/\mathcal F)\subset \Gal(\mathcal
L/\mathcal K)$ of automorphisms of $\mathcal L$ that are
differential $\mathcal F$-algebra automorphism. The Galois
correspondence between closed subgroups and intermediate extensions
is first shown by Kolchin (see \cite{Ko0} and \cite{Ko1}).

\index{Galois!correspondence}
\begin{theorem}\label{ThGaloisCorrespondence}
 The maps
$$ H \mapsto \mathcal L^H \subset \mathcal L$$
from group subschemes of $\Gal(\mathcal L/\mathcal K)$ to
intermediate differential extensions and
$$\mathcal F \mapsto \Gal(\mathcal L/\mathcal F) \subset \Gal(\mathcal L/\mathcal K)$$
from intermediate differential extensions subgroup schemes, are
bijective and inverse each other. The extension $\mathcal K\subset
\mathcal F$ is strongly normal if and only if $\Gal(\mathcal
L/\mathcal F)$ is a normal subgroup of $\Gal(\mathcal L/\mathcal
K)$. In such case $\Gal(\mathcal F/\mathcal K)$ is isomorphic to the
quotient $\Gal(\mathcal L/\mathcal K)/\Gal(\mathcal L/\mathcal F)$.
\end{theorem}

\subsection{Lie Extensions}

  The algebraic differential approach to Lie-Vessiot systems,
in terms of differential fields, was initiated by K. Nishioka
\cite{Ni3}. He relates the differential extensions generated by
solutions of a Lie-Vessiot system with \emph{algebraic dependence on
initial conditions}; a concept introduced by H. Umemura
\cite{Umemura1985} in relation with the analysis of Painlev\'e differential
equations. He also
introduces the notion of \emph{Lie extension}, a differential field
extension that carry the infinitesimal structure of  a Lie-Vessiot
system. Here we review some of his results, in order to relate them
with the Galois theory of automorphic systems. Consider a
differential field $\mathcal K$ of characteristic zero with
algebraically closed constant field $\mathcal C$. Any considered
differential extension of $\mathcal K$ is a subfield of certain
fixed universal extension of $\mathcal K$.

\begin{definition}\label{DefRationalDependence}
 We say that a differential extension
$\mathcal K  \subset \mathcal R$ depends rationally on arbitrary
constants if there exist a differential field extension $\mathcal
K\subset \mathcal M$ such that $\mathcal R$ and $\mathcal M$ are
free over $\mathcal K$ and $\mathcal R\cdot \mathcal M = \mathcal M
\cdot C_{\mathcal R \cdot \mathcal M}$.
\end{definition}

% It follows that intermediate differential field extensions of a
%strongly normal extension depend rationally on arbitrary constants.
%It is expected that, under reasonable conditions, it gives a
%characterization of those differential field extensions. There are
%some partial results for the converse.

%\begin{theorem}[\cite{Umemura1985}]
%  Let $\mathcal K$ be an extension of the complex numbers $\mathbb C$ generated by a
%  finite number of meromorphic functions in some domain of $\mathbb
%  C$. Consider $\eta$ the general solution of an algebraic differential
%  equation, and $\mathcal K\subset \mathcal R$ the differential
%  field extension generated by $\eta$. If it depends rationally on
%  arbitrary constants, then $\mathcal R$ is contained
%  in the terminal $\mathcal K_m$ of a finite tower of strongly normal
%  extensions,
%  $$\mathcal K \subset \mathcal K_1 \subset \ldots \subset \mathcal
%  K_m.$$
%\end{theorem}

%\begin{theorem}[\cite{Ni2}]\label{ThNishioka1}
%  Let $\mathcal K$ be algebraically closed and $\mathcal K \subset
%  \mathcal R$ a differential field extension generated by a single element which is
%  differentially algebraic over $\mathcal K$. The following are equivalent:
%  \begin{enumerate}
%  \item[(i)] $\mathcal C = C_{\mathcal R}$ and $\mathcal R$ depends
%  on arbitrary constants;
%  \item[(ii)] there exists a strongly normal extension of $\mathcal
%  K$ which contains $\mathcal R$.
%  \end{enumerate}
%\end{theorem}

  For  a differential extension $\mathcal K \subset \mathcal L$
denote $\Der_{\mathcal K}(\mathcal L)$ the space of derivations of
$\mathcal L$ that vanish over $\mathcal K$. This space is a
$\mathcal K$-Lie algebra.

\begin{definition}\label{DefLieExtension}
  We say that a differential extension $\mathcal K\subset \mathcal
  L$ is a Lie extension if $\mathcal C = C_{\mathcal L}$, there
  exists a $\mathcal C$-Lie sub algebra $\mathfrak g\subset \Der_{\mathcal
  K}(\mathcal L)$ such that $[\partial, \mathfrak g] \subset
  \mathcal K \mathfrak g$, and $\mathcal L \mathfrak g =
  \Der_{\mathcal K}(\mathcal L)$.
\end{definition}

\begin{theorem}[\cite{Ni3}]\label{ThNishioka2}
  Suppose that $\mathcal K$ is algebraically closed. Then every
  intermediate differential field of a strongly normal extension of $\mathcal K$
  is a Lie extension.
\end{theorem}

%\begin{theorem}[\cite{Ni3}]
%  Let $\mathcal K\subset \mathcal L$ be a Lie extension and
%  $\mathcal R$ the maximum between intermediate differential fields
%  depending rationally on arbitrary constants. Then $\mathcal R$ is
%  also a Lie extension.
%\end{theorem}

\subsection{Schemes with Derivation}

  In this section we present some facts of the theory of schemes
with derivations. This is mainly the point of view of \cite{Bu}.
However we consider only regular derivations whereas A. Buium
considers the more general case of meromorphic derivations. Our
purpose is to relate schemes with derivations to differential
schemes. Note that the regularity of the derivation is essential to
Theorem \ref{C2THE2.3.2} below; hence it does not hold under Buium's
definition.

Let $X$ be a scheme. A derivation $\partial_X$ of the structure
sheaf $\mathcal O_X$ is a law that assigns to each open subset
$U\subset X$ a derivation $\partial_X(U)$ of the ring $\mathcal
O_X(U)$. This law is assumed to be compatible with restriction
morphisms.

\begin{definition}\index{scheme!with derivation}
  A scheme with derivation is a pair $(X,\partial_X)$ consisting of a scheme $X$
and a derivation $\partial_X$ of the structure sheaf $\mathcal O_X$.
\end{definition}

  Thus, a scheme with derivation is a scheme such that its structure
sheaf is a sheaf of differential rings. A \emph{morphism of schemes
with derivation} is a scheme morphism such that induces a morphism
of sheaves of differential rings.

  Let $\mathcal K$ be a differential field. A \emph{$\mathcal
K$-scheme with derivation} is a scheme with derivation
$(X,\partial)$ together with a morphism
$(X,\partial)\to(\Spec(\mathcal K),\partial)$. Thus, the structure
sheaf of $X$ is a sheaf of differential $\mathcal K$-algebras.

  Let $(X,\partial_X)$, $(Y, \partial_Y)$ be two $\mathcal K$-schemes with
derivation. Then the direct product $X\times_{\mathcal K}Y$ admits
the derivation $\partial_X\otimes 1 + 1\otimes \partial_Y$. Then,
$$(X\times_{\mathcal K} Y, \partial_X\otimes 1 + 1\otimes \partial_Y)$$
is the direct product of $(X,\partial_X)$ and $(Y,\partial_Y)$ in
the category of schemes with derivation.

\subsection{Differential Schemes and Schemes with Derivation}

\begin{theorem}\label{C2THE2.3.2}
  Given a scheme with derivation $(X,\partial)$ there exist a
unique topological subspace $X' \subset X$ verifying
\begin{enumerate}
\item[(1)] $X'$ endowed with the structure sheaf $\mathcal
O_X|_{X'}$ and the derivation $\partial|_{X'}$ is a differential
scheme. This differential scheme will be denoted
$\Diff(X,\partial)$.
\item[(2)] For each open affine subset $U\subset X$, $U\cap X'
\simeq \DiffSpec(\mathcal O_X(U),\partial)$.
\end{enumerate}
Furthermore, each morphism of schemes with derivation $\mathcal
(X,\partial_X) \to (Y, \partial_Y)$ induces a morphism of
differential schemes $\Diff(X,\partial_X) \to \Diff(Y,\partial_Y)$.
The assignation $(X,\partial)\leadsto \Diff(X,\partial)$ is
functorial.
\end{theorem}

\begin{proof}
  If $X$ is an affine scheme then the theorem holds,
and $$X' = \DiffSpec(\mathcal O_X(X)).$$

Let us consider the non-affine case. Let $(X,\partial_X)$ be an
scheme with derivation, and let $\{U_i\}_{i\in\Lambda}$ be a
covering of $X$ by affine subsets. The ring of sections $\mathcal
O_X(U_i)$ is a differential ring for al $i\in\Lambda$, and its
spectrum $\Spec(\mathcal O(U_i))$ is canonically isomorphic to
$U_i$.

 For each $i\in \Lambda$ we take take $U_i'$ the differential
spectrum $\DiffSpec(\mathcal O_X(U_i))$, which is a topological
subspace of $U_i$. Then $U_i'\subset U_i \subset X$. Let us define
$X' = \bigcup_{i\in\Lambda} U'_i$. Thus, $X'$ is a locally
differential ringed space with the sheaf $\mathcal O_X|_{X'}$.

  Let us prove that $X'$ is a differential scheme.

  First, let us prove that $U_i\cap X' = U'_i$. By construction
we have, $U'_i \subset U_i\cap X'$. Let us consider $x\in U_i\cap
X'$. It means that for certain $j\in\Lambda$, $x \in U_i\cap U_j$,
and $x\in U'_j\subset U_j$. Let us consider an affine neighborhood
$U_x$ of $x$ contained in such intersection. Because the inclusion
$U_x\to U_j$, we have that $x\in U'_x = \DiffSpec(\mathcal
O_X(U_x))$. Then we have inclusions and restriction as follows:
$$\xymatrix{U_x \ar[r]\ar[rd] & U_i \\ & U_j}\quad\quad \xymatrix{\mathcal O_X(U_x) & \ar[l] \mathcal O_X(U_i)
\\ & \mathcal O_X(U_j) \ar[ul] }\quad\quad\xymatrix{U_x' \ar[r]\ar[rd] & U'_i \\ & U'_j}$$
We conclude that $x\in U'_i$.

Secondly, let us prove that for any affine subset $U$, the
intersection $U\cap X'$ is an affine differential scheme
$\DiffSpec(\mathcal O_X(U))$. Let $U$ be an affine subset, and let
us denote $U'$ the differential spectrum $\DiffSpec(\mathcal O_X(U))
$ that we consider as a subset of $U$. Let us consider $x\in U'$.
Then, for certain $i\in\Lambda$, $x\in U\cap U_i$. Let $U_x$ be an
affine neighborhood of $x$ such that $U_x\subset U\cap U_i$. Denote
by $U'_x$ the differential spectrum of $\mathcal O_X(U_x)$. We have
that $U'_x\subset U'_i$, and then $x\in U\cap X'$. Reciprocally let
us consider $x\in U\cap X'$. Then for certain $i\in\Lambda$ we have
$x\in U'_i$. By the same argument, we have that $x\in U$ is a prime
differential ideal of $\mathcal O_X(U)$.

  The derivation $\partial$ induces derivations on the structure
sheaf of $U\cap X$ for each affine open subset $U\subset X$. Then,
it induce a derivation $\partial\colon \mathcal O_{X'}\to\mathcal
O_{X'}$ and $\Diff(X,\partial) = (X', \mathcal
O_{X}|_{X'},\partial|_{X'})$ is a differential scheme.

  Finally, let us consider $f\colon(X,\partial_X)\to(Y,\partial_Y)$ a
morphism of schemes with derivation. If we assume that they are both
affine schemes, then the theorem holds. In the general case, we
cover $Y$ by affine subsets $\{U_i\}_{i\in\Lambda}$, and each fiber
$f^{-1}(U_i)$ by affine subsets $\{V_{ij}\}_{i\in\Lambda, j\in\Pi}$.
Then $f$ is induced by the family of differential ring morphisms
$$f^\sharp_{ij} \colon\mathcal O_Y(U_j) \to \mathcal O_X(V_{ij}).$$
These morphisms induce morphisms,
$$f'_{ij}\colon V_{ij}'\to U'_{i},$$
of locally differential ringed spaces which coincide on the
intersections, and then they induce a unique morphism,
$$f'\colon X'\to Y'.$$
\end{proof}

\index{differential!point}
\begin{definition}
  Let $(X,\partial)$ be an scheme with derivation. We will say that
$x\in X$ is a differential point if $x\in \Diff(X,\partial)$.
\end{definition}

\begin{corollary}
  Let us consider $(X,\partial)$ an scheme with derivation, and $x$ a point
of $X$. Then; the following are equivalent:
\begin{enumerate}
\item[(a)] $x\in X$ is a differential point.
\item[(b)] For each affine neighborhood $U$, $x$ correspond to a
differential ideal of $\mathcal O_X(U)$.
\item[(c)] The maximal ideal $\mathfrak m_x$ of the local ring $\mathcal O_{X,x}$ is a
differential ideal.
\item[(d)] The derivation $\partial$ induces
a structure of differential field in quotient field $\kappa(x)$.
\item[(e)] The derivation $\partial$ restricts to the Zariski
closure of $x$.
\end{enumerate}
\end{corollary}

\subsection{Split of Schemes with Derivation}

Let $Z$ be a scheme provided with the zero
derivation. Then we will write $Z$ instead of the pair $(Z,0)$.
Consider  a differential field $\mathcal K$ and let $\mathcal C$ be
its field of constants.

\index{split!of schemes with derivation}
\begin{definition}\label{C2DEFsplitSD}
 We say that a $\mathcal K$-scheme with derivation $(X,\partial)$ splits, if there is
a $\mathcal C$-scheme $Y$, and an isomorphism
$$\phi\colon(X,\partial) \xrightarrow{\sim} Y\times_{\mathcal C}
(\Spec(\mathcal K), \partial),$$ $\phi$ is called a splitting
isomorphism for $(X,\partial)$.
\end{definition}

\begin{definition}
  The space of constants $\Const(X,\partial)$ is locally ringed space defined as follows:
it is the topological subspace of differential points of $X$,
endowed with restriction of the sheaf of constant regular functions.
\end{definition}

\begin{proposition}
 Suppose $(X,\partial)$ is Keigher, then
$$\Const(X,\partial) = \Const(\Diff(X,\partial)).$$
\end{proposition}

\begin{proof}
  As topological subspaces of $X$ they coincide by construction.
Let $X' = \Diff(X,\partial)$. If $X$ is Keigher then $\mathcal
O_{X'}(U) = \lim_{\substack{\to \\ U\subseteq V}}\mathcal O_X(V)$
(see \cite{Ca0}). And because of that we have,
$$C\left(\lim_{\substack{\to \\ U\subseteq V}}\mathcal
O_X(V)\right) = \lim_{\substack{\to \\ U\subseteq V}} C_{\mathcal
O_X(V)},$$ and we finish.
\end{proof}

\index{scheme!with derivation!almost-constant}
\begin{definition}
  $(X,\partial)$ is almost-constant if $\Const(X,\partial)$ is a
scheme.
\end{definition}

\begin{proposition}
  If $(X,\partial)$ splits, then $\Diff(X,\partial)$ splits. If
$(X,\partial)$ is reduced and split, then it is almost-constant and
  $$(X,\partial) \xrightarrow{\sim} \Const(X,\partial) \times_{\mathcal
C} (\Spec(\mathcal K),\partial).$$
\end{proposition}

\begin{proof}
 Let us consider the splitting isomorphism $(X,\partial) \to Y\times_{\mathcal C}
(\Spec(\mathcal K), \partial)$. It is clear that
$\Diff(Y\times_{\mathcal C} (\Spec(\mathcal K), \partial)) = Y
\times_{\mathcal C}\DiffSpec(\mathcal K)$. Then the above splitting
isomorphism induces the splitting isomorphism of the differential
scheme $\Diff(X,\partial)$. If $X$ is reduced, then
$\Diff(X,\partial)$ is also reduced, and then we apply Proposition
\ref{Kov28.2}.
\end{proof}

%%%%%%%%%%%%%%%%%%%%%%%%%%%%%%%%%%%%%%%%%%%%%%%%%%%%%%%%%%%%%%%%%%%%%%%%%%%%
%                                                                          %
%      CHAPTER 3 - GALOIS THEORY                                           %
%                                                                          %
%%%%%%%%%%%%%%%%%%%%%%%%%%%%%%%%%%%%%%%%%%%%%%%%%%%%%%%%%%%%%%%%%%%%%%%%%%%%

\section{Galois theory of Algebraic Lie-Vessiot Systems}\label{C3}

  In this chapter we discuss the Galois theory of Lie-Vessiot
systems on algebraic homogeneous spaces. The field of functions
of the independent variable is here a differential
field $\mathcal K$ of characteristic zero and with a field of
constants $\mathcal C$ that we assume to be algebraically closed. 
We modelize algebraic Lie-Vessiot systems with coefficients
in $K$ as certain $\mathcal K$-schemes with derivation. We
study the general solution of algebraic Lie-Vessiot systems. 
It means that we study the differential extensions of $\mathcal K$ that allow us
to \emph{split} the Lie-Vessiot system, and the associated
automorphic system. We find that they are strongly normal extensions
in the sense of Kolchin \cite{Ko0}, and then we can apply Kovacic's
approach to Kolchin's differential Galois theory. In fact, the
Galois theory presented here should be seen as a generalization of
the classical Picard-Vessiot theory, obtained by replacing the
general linear group by an arbitrary algebraic group. However, 
the particular case of Picard-Vessiot theory contains all obstructions 
to solvability, because the
non-linear part of an algebraic group over $\mathcal C$ is an
abelian variety: abelian groups do not give obstruction to
integration by quadratures.

\subsection{Differential Algebraic Dynamical Systems}

  Here we establish a parallelism between
dynamical systems and differential algebraic terminology. 
\emph{
  From now on let us consider a differential field $\mathcal K$,
and $\mathcal C$ its field of constants. We assume that $\mathcal C$
is algebraically closed and of characteristic zero.}
We
modelize non-autonomous dynamical systemas as schemes
with derivation. The phase space is an algebraic variety $M$ over
the constant field $\mathcal C$, and the extended phase space
is $M_\mathcal K = M \times_{\mathcal C}\Spec(\mathcal K)$. Therefore,
non-autonomous dynamical system on $M$ with coefficients in 
$\mathcal K$ is a derivation on $M_{\mathcal K}$.

\begin{definition}
  A differential algebraic dynamical system is a
$\mathcal K$-scheme with derivation $(M,\partial_M)$ such that $M$
is an algebraic variety over $\mathcal K$. We say that
$(M,\partial_M)$ is non-autonomous if $\mathcal K$ is a non-constant
differential field.
\end{definition}

  There is a huge class of dynamical systems that can be seen as
differential algebraic dynamical systems, as polynomial or
meromorphic vector fields. It includes Lie-Vessiot systems in
algebraic homogeneous spaces, hence it also includes systems of
linear differential equations. Furthermore, a differential algebraic
study of a dynamical system is suitable in the most general case,
but results depend on the choice of an adequate differential field
$\mathcal K$.

  For a differential algebraic dynamical system $(M,\partial_M)$ we
have the associated differential scheme $\Diff(M,\partial_M)$. As a
topological space this differential scheme is the set of \emph{all
irreducible algebraic invariant subsets of the dynamical system}. By
algebraic, we mean that they are objects defined by algebraic
equations with coefficients in $\mathcal K$.

  Let us recall that for a $\mathcal
K$-algebra $\mathcal L$ we denote by $M(\mathcal L)$ the set of
$\mathcal L$-points of $M$. This sets consist of all the morphisms
of $\mathcal K$-schemes from $\Spec(\mathcal L)$ to $M$, or
equivalently, of all the rational points of the extended scheme
$$M_{\mathcal L} = M \times_{\mathcal K} \Spec\mathcal L.$$

\begin{definition}
  Let $(M,\partial_M)$ be a $\mathcal K$-scheme with derivation.
 We call rational solution of $(M,\partial_M)$ any rational
differential point $x\in\Diff(M,\partial_M)$. Let us consider a
differential extension $\mathcal K\subset \mathcal L$. A solution
with coefficients in $\mathcal L$ is an $\mathcal L$-point $x\in
M(\mathcal L)$ such that the morphism
$$x\colon (\Spec(\mathcal L),\partial)\to (M,\partial_M),$$
is a morphism of schemes with derivation. In such a case the image
$x(0)=\mathfrak x$ of the ideal $(0)\subset\mathcal L$ by $x$ is a differential point $\mathfrak
x\in \Diff(M,\partial_M)$ and its quotient field $\kappa(\mathfrak
x)$ is an intermediate extension,
$$\mathcal K \subset \kappa(\mathfrak x) \subset \mathcal L,$$
we say that $\kappa(\mathfrak x)$ is the differential field
generated by $x\in M(\mathcal L)$.
\end{definition}

  As in classical algebraic geometry, there is a one-to-one
correspondence between solutions with coefficients in $\mathcal L$
of $(M,\partial_M)$ and rational solutions  of the differential
algebraic dynamical system after a base change,
$(M,\partial_M)\times_{\mathcal K}(\Spec(\mathcal L),\partial)$.

\begin{definition}
  Let us consider two differential algebraic dynamical systems
over $\mathcal K$, $(M,\partial)$ and $(N,\partial)$. We say that
$(M,\partial)$ reduces to $(N,\partial)$ if there is an algebraic
variety $Z$ over $\mathcal C$ and,
$$(M,\partial) = (N,\partial) \times_{\mathcal C}  Z.$$
\end{definition}

  The notion of reduction is a generalization of the notion of split.
In particular, to split means reduction to $(\Spec(\mathcal
K),\partial)$.

  Given a differential algebraic dynamical system; what does it
mean to \emph{integrate} the dynamical system? As algebraists, we
shall use this term for writing down the general solution of the
dynamical system by terms of known operations, mainly algebraic
operations and quadratures. However, in the general context of
dynamical systems there is not a general definition for
\emph{integrability}. We are tempted to say that integrability is
equivalent to split. Notwithstanding, there are several situations
in which the general solution can be given, but there is not a
situation of split. For example, algebraically completely integrable
Hamiltonian systems \cite{AMV}. In such cases the flux is tangent to
a global lagrangian bundle, and the generic fibers of this bundle
are affine subsets of abelian varieties. It allows us to write down
the global solution by terms of Riemann theta functions and Jacobi's
inversion problem. However, this general solution can not be
expressed in terms of the splitting of a scheme with derivation.

  Split is the differential algebraic equivalent to \emph{Lie's
canonical form of a vector field}. The scheme with derivation
$Z\times_{\mathcal C} (\Spec(\mathcal K),\partial)$ should be seen
as an extended phase space, and $\partial$ as the derivative with
respect to the time parameter. The splitting morphism,
$$(M,\partial) \to Z\times_{\mathcal C} (\Spec(\mathcal
K),\partial),$$ can be seen as Lie's canonical form, usually
referred to, in dynamical system argot, as the \emph{flux box
reduction}. Then $Z$ is simultaneously the algebraic variety of
initial conditions, and the \emph{space of global solutions} of the
dynamical system. Our conclusion is that the split differential
algebraic dynamical systems are characterized by following the
property: \emph{its space of solutions is parameterized by a
scheme over the constants}.

  In the context of algebraic Lie-Vessiot systems we will see that algebraic
solvability of the problem, is equivalent to the notion of
\emph{split} (Theorem \ref{C3THE3.1.17}). And then, this notion
plays a fundamental role in our theory. We will see that
generically, a Lie-Vessiot equation does not split. If we want to
solve it, then we need to admit some new functions by means of a
differential extension of $\mathcal K\subset \mathcal L$. Thus, the
dynamical system splits after a base change to $\mathcal L$. The
Galois theory will provide us with the techniques for obtaining such
extensions and studying their algebraic properties (Proposition
\ref{C3PRO3.2.5}).

\subsection{Algebraic Lie-Vessiot Systems}

  From now on we will consider a fixed characteristic zero 
differential field $\mathcal K$ whose field of
constants $\mathcal C$ is algebraically closed.
Let $G$ be a $\mathcal C$-algebraic group, and
$M$ a faithful homogeneous $G$-space.

\begin{definition}
 A non-autonomous algebraic vector field $\vec X$ in $M$ with
 coefficients in $\mathcal K$ is an element of the vector space
 $\mathfrak X(M)\otimes_{\mathcal
C}\mathcal K$.
\end{definition}

A non-autonomous algebraic vector field $\vec X$ in $M$ is written
in the form,
$$\vec X = \sum_{i=1}^s f_i\vec X_i,$$
for certain elements $f_i\in \mathcal K$ and $\vec X_i\in\mathfrak X(M)$. We
define the \emph{derivation $\partial_{\vec X}$ associated to $\vec
X$} as the following derivation of the extended scheme $M_{\mathcal
K}$:
$$\partial_{\vec X}\colon  \mathcal K \otimes_{\mathcal C} \mathcal O_M \to 
\mathcal K  \otimes_{\mathcal C} \mathcal O_M,\quad a \otimes f \mapsto \partial a \otimes f
+ \sum_{i=1}^s(af_i\otimes \vec X_if).$$

\index{Lie-Vessiot system!algebraic}
\begin{definition}
A non-autonomous algebraic vector field $\vec X$ in $M$ with
coefficients in $\mathcal K$ is called a Lie-Vessiot vector field if
belongs to $\mathcal R(G,M)\otimes_{\mathcal C}
\mathcal K$. The differential algebraic dynamical system
$(M_{\mathcal K},\partial_{\vec X})$ is called a Lie-Vessiot system
in $M$ with coefficients in $\mathcal K$.
\end{definition}

  The group $G$ is, in particular, a faithful homogeneous $G$-space.
Let us recall that the Lie algebra of fundamental fields on the
group $G$ coincides with the Lie algebra of right invariant vector
field $\mathcal R(G)$. Then, a Lie-Vessiot vector field in $G$ with
coefficients in $\mathcal K$ is an element of $\mathcal
R(G)\otimes_{\mathcal C} \mathcal K$.

\begin{definition}
 We call automorphic vector fields to the Lie-Vessiot vector fields in $G$. An automorphic
vector field $\vec A$ in $G$ with coefficients in $\mathcal K$ is an
element of $\mathcal R(G)\otimes_{\mathcal C} \mathcal K$.
\end{definition}

The canonical isomorphism between $\mathcal R(G)$ and $\mathcal
R(G,M)$ allows us to translate Lie-Vessiot vector fields in $M$ to
automorphic vector fields in $G$.

\index{automorphic!system!algebraic}
\begin{definition}
  We call automorphic system associated to $(M,\partial_{\vec X})$ to the
Lie-Vessiot system $(G_{\mathcal K},\partial_{\vec A})$, where $\vec
A$ is the automorphic vector field whose corresponding Lie-Vessiot
vector field in $M$ is $\vec X$.
\end{definition}

\emph{
  From now on let $\vec X$ be a Lie-Vessiot vector field
in $M$, with coefficients in $\mathcal K$, and let $\vec A$  be the
associated automorphic vector field in $G$.}

\subsection{Logarithmic Derivative}

  A $\mathcal K$-point of the algebraic group $G$ has coefficients
in a differential field, so that it can be differentiated. The derivative
of a $\mathcal K$-point of $G$ gives a tangent vector at a $\mathcal K$-point
of $G_{\mathcal K}$. If we translate this tangent vector to a right invariant
vector field, we obtain the logarithmic derivative. In order to do so we identify
systematically the Lie algebra $\mathcal R(G)$ with the tangent
space $T_eG = \Der_{\mathcal C}(\mathcal O_{G,e}, \mathcal C)$. It
is also important to remark that the tangent space is compatible
with extensions of the base field in the following way:
$$\mathcal R(G)\otimes_{\mathcal C}\mathcal K \xrightarrow{\sim}
T_e(G_{\mathcal K}) = \Der_{\mathcal K}(\mathcal O_{G_{\mathcal
K},e}, \mathcal K).$$
 In classical algebraic geometry it is assumed that derivations of
$T_e(G_{\mathcal K})$ vanish on $\mathcal K$. However, automorphic
systems are by definition compatible with the derivation $\partial$
of $\mathcal K$. Thus, the restriction of an automorphic vector
field $\partial_{\vec A}$ to $e\in G_{\mathcal K}$ is not a tangent
vector of $T_e(G_{\mathcal K})$: it is shifted by $\partial$. We
have identifications of $\mathcal K$-vector spaces:
$$\xymatrix{\mathcal R(G)\otimes_{\mathcal C}\mathcal K \ar[r]^-{\sim}
& \mathcal R(G)\otimes_{\mathcal C} \mathcal K + \partial  \ar[r]^-{-\partial} & T_e(G_{\mathcal K}) \\
\vec A \ar[r] & \partial_{\vec A} = \partial + \vec A  \ar[r] & \vec
A_{e}}$$

  Let us consider $\sigma\in G(\mathcal K)$ and the canonical
  morphism $\sigma^\sharp$ of \emph{taking values in $\sigma$:}
$$\sigma^\sharp\colon \mathcal O_{G_{\mathcal K},{\sigma}} \to \mathcal K,
\quad f\mapsto f(\sigma).$$

Let us remember that there is a canonical form of extension of the
derivation $\partial$ in $\mathcal K$ to a derivation in
$G_{\mathcal K}$. We consider the direct product $G\times_{\mathcal
C}(\Spec(\mathcal K),\partial)$ in the category of schemes with
derivation. By abuse of notation \emph{we denote by $\partial$ this
canonical derivation in $G_{\mathcal K}$}. By construction we have
that $(G_{\mathcal K},\partial)$ splits -- the identity is the
splitting morphism -- and $\Const(G_{\mathcal K},\partial) = G$. Let
us consider the following \emph{non-commutative} diagram,
\begin{equation}\label{EnonComm}
\xymatrix{ \mathcal O_{G_{\mathcal K},{\sigma}}
\ar[rr]^-{\sigma^\sharp}\ar[d]_-{\partial} & & \mathcal K
\ar[d]^-{\partial}
\\ \mathcal O_{G_{\mathcal K},{\sigma}} \ar[rr]^-{\sigma^\sharp} & & \mathcal K}.
\end{equation}

\begin{lemma}
 The commutator $\sigma' = [\partial,\sigma^\sharp]$ of the diagram
\eqref{EnonComm} is a derivation vanishing on $\mathcal K$, and then
$\sigma'$ belong to the tangent space $T_\sigma (G_{\mathcal K})$
(id est, the space of derivations $\Der_{\mathcal K}(\mathcal
O_{G_{\mathcal K},\sigma},\mathcal K)$).
\end{lemma}

\begin{proof}
  $[\partial, \sigma^\sharp]$ is the difference between two
derivations, and then it is a derivation. Let us consider
$f\in\mathcal K\subset \mathcal O_{G_{\mathcal K}\sigma}$, then
$\sigma'(f) = \partial f - \partial f = 0$.
\end{proof}

  If $\sigma$ is a geometric point of $G_{\mathcal K}$, then
$R_{\sigma^{-1}}$ is a automorphism of $G_{\mathcal K}$ sending
$\sigma$ to $e$. It induces an isomorphism between the ring of germs
$\mathcal O_{G_{\mathcal K},\sigma}$ and $\mathcal O_{G_{\mathcal
K},e}$, and then an isomorphisms between the corresponding spaces of
derivations:
$$\xymatrix{T_\sigma(G_{\mathcal K}) \ar[rr]^-{R_{\sigma^{-1}}'} & &
T_e(G_{\mathcal K}) \simeq \mathcal R(G) \otimes_{\mathcal C}
\mathcal K}$$

\index{logarithmic derivative!algebraic}
\begin{definition}\label{DEFLogDerAlg}
  Let $\sigma$ be a geometric point of $G_{\mathcal K}$; we call
logarithmic derivative of $\sigma$, $l\partial(\sigma)$, to the
automorphic vector fiel
$R_{\sigma^{-1}}'([\partial,\sigma^\sharp])$. The logarithmic
derivative is then a map:
$$l\partial \colon G(\mathcal K) \to \mathcal R(G) \otimes_{\mathcal C}
\mathcal K.$$
\end{definition}

\begin{proposition}
Properties of logarithmic derivative:
\begin{enumerate}
\item[(1)] Logarithmic derivative is functorial in $\mathcal K$; for
each differential extension $\mathcal K\subset\mathcal L$ we have a
commutative diagram:
$$\xymatrix{G(\mathcal K) \ar[r]\ar[d]& \mathcal R(G) \otimes_{\mathcal C} \mathcal K \ar[d] \\
 G(\mathcal L) \ar[r]& \mathcal R(G) \otimes_{\mathcal C} \mathcal L}$$
\item[(2)] Let us consider $\sigma$ and $\tau$ in $G(\mathcal K)$:
$$l\partial(\sigma\tau) = l\partial(\sigma) +
\Adj_\sigma(l\partial(\tau))$$
\item[(3)] Let us consider $\sigma\in G(\mathcal K)$:
$$l\partial(\sigma^{-1}) = -\Adj_{\sigma}(l\partial(\sigma)).$$
\end{enumerate}
\end{proposition}

\begin{proof}
  (1) comes directly from the differential field extension, (2) 
comes from the right invariance, and (3) is corollary to (2).
\end{proof}

\subsection{Automorphic Equation}

\index{automorphic!equation}
\begin{theorem}
  Let us consider $\mathcal K\subset\mathcal L$ a differential
extension. Then  $\sigma\in G(\mathcal L)$ is a solution of the
differential algebraic dynamical system $(G_{\mathcal
K},\partial_{\vec A})$ if and only if $l\partial(\sigma) = \vec A$.
\end{theorem}

\begin{proof}
  Let us consider $\sigma\in G(\mathcal L)$, and let $\vec B$ be its
logarithmic derivative. The space $\mathcal R(G)\otimes_{\mathcal C}
\mathcal L$ is canonically identified with the Lie algebra of right
invariant vector fields on the \emph{base extended} $\mathcal
L$-algebraic group $G_{\mathcal L}$:
$$\mathcal R(G)\otimes_{\mathcal C}\mathcal L = \mathcal
R(G_{\mathcal L}).$$

By this identification, the automorphic vector field $\vec B$ is
seen as a derivation $\vec B$ of the structure sheaf $\mathcal
O_{G_{\mathcal L}}$. The germ $\vec B_{(\sigma)}$ at $\sigma$ of
$\vec B$ is a derivation of the ring $\mathcal O_{G_{\mathcal
L},\sigma}$. The composition with $\sigma^\sharp$ give us the
tangent vector $\vec B_\sigma\in T_{\sigma}(G_{\mathcal L})$:
$$\xymatrix{\mathcal O_{G_{\mathcal K},\sigma} \ar[r]^-{\vec B_{(\sigma)}}  \ar[rrd]_-{\vec
B_{\sigma}} & \mathcal O_{G_{\mathcal K},\sigma}
\ar[rd]^-{\sigma^\sharp} &  \\ & & \mathcal K}$$

The value of $\vec B$ at the identity point is, by definition,
$l\partial(\sigma)$. Since $\vec B$ is a right invariant vector
field we have $l\partial(\sigma) = R_{\sigma^{-1}}'(B_{\sigma}) =
\sigma^\sharp\circ \vec B_{(\sigma)} \circ R_{\sigma^{-1}}^\sharp$
hence $\vec B_{\sigma}$ is equal to the commutator $[\partial,
\sigma^\sharp]$ of Definition \ref{DEFLogDerAlg}. Then, $\vec
B_{(\sigma)}$ is the defect of the diagram \eqref{EnonComm};
therefore the following diagram commutes:
$$\xymatrix{ \mathcal O_{G_{\mathcal K},{\sigma}}
\ar[rr]^-{\sigma^\sharp}\ar[d]_-{\partial + \vec B_{(\sigma)}} & &
\mathcal K \ar[d]^-{\partial}
\\ \mathcal O_{G_{\mathcal K},{\sigma}} \ar[rr]^-{\sigma^\sharp} & & \mathcal
K}.$$ Furthermore, $\vec B$ is determined by the commutator $\vec
B_\sigma = [\partial,\sigma^{\sharp}]$ and then it is unique right
invariant vector field in $G_{\mathcal L}$ that forces the diagram
to commute.

  Let us note that the commutation of the above diagram holds
if and only if the kernel $\mathfrak m_{\sigma}$ of $\sigma^\sharp$
is a differential ideal. Then $\vec B$ is the unique right invariant
vector field in $G_{\mathcal L}$ such that the maximal ideal
$\mathfrak m_{\sigma}$ is a differential ideal. Let us note also
that, this derivation $\partial + \vec B_{\sigma}$ is the germ in
$\sigma$ of the automorphic derivation
$$\partial_{\vec B} = \partial + \vec B,$$
we conclude that $\vec B$, the logarithmic derivative of $\sigma$,
is the unique element of $\mathcal R(G)\otimes_{\mathcal C} \mathcal
L$ such that $\sigma$ is a differential point of $(G_{\mathcal L},
\partial_{\vec B})$.
\end{proof}

  Because of that we can substitute the automorphic system $\vec A$, for the
so-called \index{equation!automorphic} \emph{automorphic equation}:
\begin{equation}\label{EqAutomorphicAlgebraic}
l\partial(x) = \vec A
\end{equation}

\subsection{Solving Lie-Vessiot Systems}

\index{gauge!transformation}
\begin{definition}
  Let us consider $\sigma\in G(\mathcal K)$. We call gauge transformation
induced by $\sigma$ to the left translation $L_\sigma\colon
G_{\mathcal K}\to G_{\mathcal K}$.
\end{definition}

\begin{lemma}\label{lmLV1}
  $(G_{\mathcal K},\partial_{\vec A})$ splits if and only if the automorphic
equation \eqref{EqAutomorphicAlgebraic} has at least one solution in
$G(\mathcal K)$.
\end{lemma}

\begin{proof}
  Assume $(G_{\mathcal K},\partial_{\vec A})$ splits. Let us consider the
splitting isomorphism
$$\psi\colon(G_{\mathcal K},\partial_{\vec A})\to Z \times_{\mathcal C}
(\Spec(\mathcal K),\partial).$$
 Let $x$ be a $\mathcal C$-rational point of
$Z$. Let us denote by $x_{\mathcal K}$ the corresponding $\mathcal
K$-point of $G_{\mathcal K}$ obtained after the extension of the
base field. Thus, $\psi^{-1}(x_{\mathcal K})$ is a solution of
\eqref{EqAutomorphicAlgebraic}. Reciprocally, let us assume that
there exists a solution $\sigma$  of \eqref{EqAutomorphicAlgebraic}
in $G(\mathcal K)$. Let us consider the gauge transformation:
$$L_{\sigma^{-1}}\colon G_{\mathcal K} \to G_{\mathcal K}.$$
It applies $\sigma$ onto the identity element $e\in G_{\mathcal K}$.
But the logarithmic derivative $l\partial(e)$ vanishes, so that
$L_{\sigma^{-1}}$ transforms $\partial_{\vec A}$ into the canonical
derivation $\partial$. We conclude that $L_{\sigma^{-1}}$ is an
splitting isomorphism.
%$(G_{\mathcal K},\partial_{\vec A})\simeq (G_{\mathcal K},\partial)$.
\end{proof}

\begin{lemma}\label{lmLV2}
  Assume that $(G_{\mathcal K},\partial_{\vec A})$ splits. In such
case we can choose the splitting isomorphism between the gauge
transformations of $G_{\mathcal K}$. This gauge transformation
induces the split of any associated Lie-Vessiot system $(M_{\mathcal
K},\partial_{\vec X})$.
\end{lemma}

\begin{proof}
  We use the same argument as above. If it splits,
  $$s\colon (G_{\mathcal K},\partial_{\vec A})\to
  G \times_{\mathcal C} (\Spec(\mathcal K),\partial) = (G,\partial),$$
then the preimage of the identity element $s^{-1}(e) = \sigma$ is a
solution of the automorphic system. So that the gauge transformation
$L_{\sigma^{-1}}\colon \sigma\mapsto e$ maps solutions of
$(G_{\mathcal K},\partial_{\vec A})$ to solutions of $(G_{\mathcal
K},\partial)$ and it is an splitting isomorphism. For any associated
Lie-Vessiot system $(M_{\mathcal K}, \partial_{\vec X})$, and any
point $x_0\in M(\mathcal C)$ we have that $L_{\sigma}(x_0)$ is a
solution of $(M_{\mathcal K},\partial_{\vec X})$. So that
$L_{\sigma}$ sends solutions of the canonical derivation $\partial$
to solutions of $\partial_{\vec X}$. Thus, its inverse
$L_{\sigma^{-1}}$ is an splitting isomorphism for $(M_{\mathcal
K},\partial_{\vec X})$.
\end{proof}

\begin{lemma}\label{LmAlmostConstantSplit}
  Let $Z$ be a $\mathcal C$-algebraic variety and $(Z_{\mathcal
K},\vec D)$ a non-autonomous differential algebraic dynamical system
over $\mathcal K$. If $(Z_{\mathcal K},\vec D)$ splits then
$(Z_{\mathcal K},\vec D)$ is almost-constant and $\Const(Z_{\mathcal
K},\vec D) \simeq Z$.
\end{lemma}

\begin{proof}
  Assume that $(Z_{\mathcal K},\vec D)$ splits. It implies that
there exist an $\mathcal C$-scheme $Y$, such that  $Z_{\mathcal K} =
Y\times_{\mathcal C} \Spec(\mathcal K)$. We have that $Z_{\mathcal
K} \simeq Y_{\mathcal K}$, and then $Z\simeq Y$.
\end{proof}

\begin{lemma}\label{lmLV3.5}
  Let $Z$ be a reduced $\mathcal C$-scheme. There is a one-to-one
correspondence between  closed subschemes of $Z$ and  closed
subschemes with derivation of $(Z_{\mathcal K},\partial) = Z
\times_{\mathcal C}(\Spec(\mathcal K),\partial)$.
\end{lemma}

\begin{proof}
  First, let us consider the affine case. Assume $Z = \Spec \mathcal R$ for a
$\mathcal C$-algebra $\mathcal R$. The ring of constants
$C_{\mathcal R\otimes_{\mathcal C} \mathcal K}$ is $\mathcal R$
itself. It follows that $\Const(Z_{\mathcal K},\partial) = Z$. It
is clear that $\mathcal R\otimes_{\mathcal C}\mathcal
K$ is an almost-constant ring: each radical differential ideal is
generated by constants. Because of that there is an one-to-one
correspondence between radical ideals of $\mathcal R$ and radical
differential ideals of $\mathcal K$.

  In the non-affine case, let us consider $Y$ a  closed
sub-$\mathcal C$-scheme of $Z$. The canonical immersion $(Y_\mathcal
K,\partial)\subset (Z_{\mathcal K},\partial)$ identifies $Y$ with a
 closed sub-$\mathcal K$-scheme with derivation of
$(Z_{\mathcal K},\partial)$. Reciprocally, let $(\tilde
Y,\partial|_{\tilde Y})$ be a  closed sub-$\mathcal K$-scheme with
derivation of $(Z_{\mathcal K},\partial)$. Let us consider
$\{U_i\}_{i\in\Lambda}$ an affine covering of $Z$. The collection
$\{V_i\}_{i\in\Lambda}$ with $V_i = U_i\times_{\mathcal C} \mathcal
K$ is then an affine covering of $Z_{\mathcal K}$. Each intersection
$\tilde Y_i = \tilde Y|_{V_i}$ is an affine  closed sub-$\mathcal
K$-scheme of $V_i$. We are in the affine case: by the above argument
there are closed sub-$\mathcal C$-schemes $Y_i\subset U_i$ such that
$(\tilde Y_i,\partial|_{\tilde Y_i}) = Y_i \times_{\mathcal C}
(\Spec(\mathcal K),\partial)$. This family defines a covering of a
closed sub-$\mathcal C$-scheme $Y = \bigcup_{i\in\Lambda} Y_i$ of
$Z$.
\end{proof}

\begin{lemma}\label{lmLV4}
  Let $Z$ be a $\mathcal C$-algebraic variety and $(Z_{\mathcal
K},\vec D)$ a non autonomous algebraic dynamical system over
$\mathcal K$. Let $Y\subset Z$ a locally closed subvariety, and
assume that $\vec D$ is tangent to $Y$, so that $(Y_{\mathcal
K},\vec D|_Y)$ is a sub-$\mathcal K$-scheme with derivation. If
$(Z_{\mathcal K},\vec D)$ splits then $(Y_{\mathcal K},\vec D|_Y)$
splits.
\end{lemma}

\begin{proof}
  By substituting $Z$ for certain open subset we can assume that $Y$ is closed.
Let us consider the splitting isomorphism,
$$\psi\colon(Z_{\mathcal K},\vec D)
\to Z\times_{\mathcal C} (\Spec(\mathcal K),\partial).$$ The image
$\psi(Y_{\mathcal K},\vec D|_Y)$ is a locally closed subscheme with
derivation of  $Z\times_{\mathcal C}(\Spec(\mathcal K),\partial)$.
By Lemma \ref{lmLV3.5} it splits.
\end{proof}

\begin{lemma}\label{lmLV5}
  Assume that the action of $G$ on $M$ is faithful. Then $(G_{\mathcal
K},\partial_{\vec A})$ splits if and only if $(M_{\mathcal K},
\partial_{\vec X})$ splits.
\end{lemma}

\begin{proof}
  Lemma \ref{lmLV2} says that if $(G_{\mathcal K}, \partial_{\vec A})$ splits, then
$(M_{\mathcal K},\partial_{\vec X})$ splits. Reciprocally, let us
assume that $(M_{\mathcal K},\partial_{\vec X})$ splits. For each
positive number $r$ we consider the natural lifting to the cartesian
power $(M^r_{\mathcal K},\partial_{\vec X}^r)$. The splitting of
$(M_{\mathcal K},\partial_{\vec X})$ induces the splitting of those
cartesian powers differential algebraic dynamical system
$(M^r_{\mathcal K},\partial_{\vec X}^r)$. For $r$ big enough there
is a point $x\in M^r$ such that its orbit $O_x$ is a principal
homogeneous space \emph{isomorphic} to $G$. 
Then $(O_{x,\mathcal K},\partial_{\vec X})$
is a locally closed sub-$\mathcal K$-scheme with derivation of
$(M^r_{\mathcal K},\partial_{\vec X}^r)$. By Lemma \ref{lmLV4} it
splits. We also know that $(O_{x,\mathcal K},\partial_{\vec X})$ is
isomorphic to $(G_{\mathcal K},\partial_{\vec A})$. Finally,
$(G_{\mathcal K},\partial_{\vec A})$ splits.
\end{proof}

\begin{theorem}\label{C3THE3.1.17}
  Assume that the action of $G$ on $M$ is faithful. Then the
following are equivalent.
\begin{enumerate}
\item[(1)] The automorphic equation \eqref{EqAutomorphicAlgebraic}
has a solution in $G(\mathcal K)$
\item[(2)] $(G_{\mathcal K},\partial_{\vec A})$ splits.
\item[(3)] There is a gauge transformation of $G_{\mathcal K}$ sending $\vec
A$ to $0$.
\item[(4)] $(M_{\mathcal K},\partial_{\vec X})$ splits.
\item[(5)] $(G_{\mathcal K},\partial_{\vec A})$ splits, is almost-constant, and
$\Const(G_{\mathcal K},\partial_{\vec A}) \simeq G$.
\item[(6)] $(M_{\mathcal K},\partial_{\vec A})$ splits, is almost-constant, and
$\Const(M_{\mathcal K},\partial_{\vec X}) \simeq M$.
\end{enumerate}
\end{theorem}

\begin{proof}
  Equivalence between (1) and (2) comes from Lemma \ref{lmLV1}.
Equivalence between (2) and (3) comes from Lemma \ref{lmLV2}. (2)
and (4) are equivalent by Lemma \ref{lmLV5}. By Lemma
\ref{LmAlmostConstantSplit}, they all imply (5) and (6).
\end{proof}

\subsection{Splitting Field of an Automorphic System}

  Note that a differential extension $\mathcal K \subset
  \mathcal L$, induces a canonical inclusion,
  $$\mathcal R(G,M)\otimes_{\mathcal C}\mathcal K \subset
  \mathcal R(G,M)\otimes_{\mathcal C}\mathcal L;$$
  so that a Lie-Vessiot vector field with coefficients in $\mathcal
  K$ is a particular case of a Lie-Vessiot vector field
  with coefficients in $\mathcal L$. So that if $(M_{\mathcal
  K},\partial_{\vec X})$ is a Lie-Vessiot system, then $(M_{\mathcal
  L},\partial_{\vec X})$ makes sense.

\index{splitting extension}
\begin{definition}
  We say that a differential extension $\mathcal K\subset\mathcal L$
is a splitting extension for $(M_{\mathcal K},\partial_{\vec X})$ if
$(M_{\mathcal L},\partial_{\vec X})$ splits.
\end{definition}

   From theorem \ref{C3THE3.1.17}, we know that $\mathcal
K\subset\mathcal L$ is a splitting extension of
{\nolinebreak$(M_{\mathcal K},\partial_{\vec X})$} if and only it is
a splitting extension of $(G_{\mathcal K},\partial_{\vec A})$. Then
we will center our attention in the automorphic vector field $\vec
A$.

\subsection{Action of $G(\mathcal C)$ on $G_{\mathcal
K}$}\label{C3SS3.3.1}

  For each $\sigma\in G(\mathcal C)$, $R_{\sigma}$ is an
automorphism of $G_{\mathcal K}$. The composition law is an action
of $G$ on $G_{\mathcal K}$ by the right side,
  $$G_{\mathcal K} \times_{\mathcal C} G \to G_{\mathcal K}.$$
  The vector field $\vec A$ is right invariant, so that we expect the differential
points of $(G_{\mathcal K},\partial_{\vec A})$ to be invariant under
right translations. In fact, the above morphism is a morphism of
schemes with derivation,
  $$(G_{\mathcal K},\partial_{\vec A}) \times_{\mathcal C} G \to
(G_{\mathcal K},\partial_{\vec A}).$$ We apply the functor $\Diff$,
and then we obtain an action of the $\mathcal C$-algebraic group $G$
on the differential scheme $\Diff(G_{\mathcal K},\partial_{\vec
A})$,
$$\Diff(G_{\mathcal K},\partial_{\vec A}) \times_{\mathcal C} G \to
\Diff(G_{\mathcal K},\partial_{\vec A}).$$ Assume that $(G_{\mathcal
K},\partial_{\vec A})$ split. In such case, when we apply the
functor $\Const$ to the previous morphism, we obtain a morphism of
schemes,
$$\Const(G_{\mathcal K},\partial_{\vec A})\times_{\mathcal C} G \to
\Const(G_{\mathcal K},\partial_{\vec A}).$$ Because of the split we
already knew that $\Const(G_{\mathcal K},\partial_{\vec A})$ is a
$\mathcal C$-scheme isomorphic to $G$. Furthermore, the above
morphism says that the action of $G$ by the right side on this
$G$-scheme is canonical. We have proven the following:

\begin{lemma}
  Assume that $(G_{\mathcal K},\partial_{\vec A})$ splits. Then
$\Const(G_{\mathcal K},\partial_{\vec A})$ is a principal
$G$-homogeneous space by the right side.
\end{lemma}

\subsection{Existence and Uniqueness of the Splitting Field}\label{C3SSexistenceuniqueness}

\begin{lemma}\label{C3LEMcloseddiffpoint}
  There is a differential point $\mathfrak x\in \Diff(G_{\mathcal K},\partial_{\vec A})$ which
is closed in the Kolchin topology.
\end{lemma}

\begin{proof}
  Let us consider the generic point $p_0\in G_{\mathcal K}$. In
particular it is a differential point $p_0\in\Diff(G_{\mathcal
K},\partial_{\vec A})$. If $p_0$ is Kolchin closed, then we finish
and the result holds. If not, then the Kolchin closure of $p_0$
contains a differential point point $p_1$ such that $p_0$
specializes on it $p_0\to p_1$. We continue this process with $p_1$.
As $G_{\mathcal K}$ is an algebraic variety, and then a noetherian
scheme, this process finish in a finite number of steps and lead us
to a Kolchin closed point.
\end{proof}

\begin{lemma}
  Let $\mathfrak x \in \Diff(G_{\mathcal K},\partial_{\vec A})$ be a closed
differential point. Then its field of quotients $\kappa(\mathfrak
x)$ is a differential extension of $\mathcal K$ with the same field
of constants; $C_{\kappa(\mathfrak x)} = \mathcal C$.
\end{lemma}

\begin{proof}
  Reasoning by \emph{reductio ad absurdum} let us assume that there
exists $c\in C_{\kappa(\mathfrak x)}$ not in $\mathcal C$.
%By Lemma \ref{LmDisjoint}
Let us consider an affine open neighborhood $U$ of $\mathfrak x$ and
denote by $A$ its ring of regular functions. We identify $\mathfrak
x$ with a maximal differential ideal $\mathfrak x\subset A$. Denote
by $B$ the quotient ring $A/\mathfrak x$. $B$ is a differential
subring of the differential field $\kappa(\mathfrak x)$. By Lemma
\ref{LmDisjoint} there exist $b\in B$ such that the ring constants
$C_{B_b}$ -- of the localized ring $B_b$ -- is a finitely generated
$\mathcal C$-algebra. By reducing our original neighborhood $U$ --
removing the zeros of $b$ -- we can assume that $b$ is invertible
and then the localized ring $B_b$ is just $B$. $C_B$ is a
non-trivial finitely generated $\mathcal C$-algebra over $\mathcal
C$, because it contains an element $c$ not in $\mathcal C$. So that
there is a non-invertible element $c_2 \in C_B$. The principal ideal
$(c_2)$ is a non trivial differential ideal in $B$. Let us consider
a regular function $a_2$ such that $a_2(\mathfrak x) = c_2$. Then
$\partial_{\vec A} a_2 \in \mathfrak x$ and $(a,\mathfrak x)$ is a
non-trivial differential ideal of $A$ strictly containing $\mathfrak
x$. We arrive to contradiction with the maximality of $\mathfrak x$.
\end{proof}

\begin{proposition}\label{C3PRO3.2.5}
  Let $\mathfrak x \in \Diff(G_{\mathcal K},\partial_{\vec A})$ be a closed point. Then $\mathcal K\subset\mathcal
\kappa(\mathfrak x)$ is a splitting extension of $(G_{\mathcal
K},\vec A)$.
\end{proposition}

\begin{proof}
  Let $\mathfrak x$ be a closed point. Then the canonical morphism
$\mathfrak x^\sharp$ of \emph{taking values in $\mathfrak x$,}
$\mathfrak x^\sharp\colon \mathcal O_{G_{\mathcal K},\mathfrak x}\to
\kappa(\mathfrak x)$ is a morphism of differential rings. Let $U$ be
an affine neighborhood of the image of $\pi(\mathfrak x)$ by the
canonical projection $\pi\colon G_{\mathcal K}\to G$. By composition
we construct a morphism $\Spec(\kappa(\mathfrak x))\to U$,
$$\xymatrix{ \mathcal O_G(U)\ar[rr]^-{\sigma^\sharp} \ar[d]_-{\pi^\sharp}
& & \kappa(\mathfrak x) \\ \mathcal O_{G_{\mathcal K},\mathfrak x}
\ar[rru] ^-{\mathfrak x^\sharp}}.$$ The morphism $\sigma^\sharp$ is
the dual of a morphisms $\sigma$ from $\Spec(\kappa(\sigma))$ to
$U$. In other words, $\sigma$ is a point of $G(\kappa(\mathfrak
x))$. We consider $\sigma$ as a rational differential point of
$(G_{\kappa(\mathfrak x)},\partial_{\vec A})$, and then it is a
solution of the automorphic equation. By Lemma \ref{lmLV1},
$(G_{\kappa(\mathfrak x)},\partial_{\vec A})$ splits.
\end{proof}

\index{fundamental solution}
\begin{definition}
  We say that $\sigma$, as defined in the above proof, is the fundamental solution
of $\vec A$ associated with the closed differential point $\mathfrak
x$.
\end{definition}

  Let us consider the action of $G$ on $G_{\mathcal K}$ by
right translations. The derivation $\partial_{\vec A}$ is invariant
by right translations, and then it is a morphism of schemes with
derivation:
$$(G_{\mathcal K},\partial_{\vec A})\times_{\mathcal C} G \to
(G_{\mathcal K},\partial_{\vec A})$$ We apply the functor $\Diff$,
thus we obtain a morphism of differential schemes which is an
algebraic action of $G$ on the set of differential points.
$$\Diff(G_{\mathcal K},\partial_{\vec A})\times_{\mathcal C} G \to
\Diff(G_{\mathcal K},\partial_{\vec A})$$

\begin{proposition}\label{C3PROtransitivity}
  The action of $G(\mathcal C)$ on the set of closed points of
 $\Diff(G_{\mathcal K},\partial_{\vec A})$ is transitive.
\end{proposition}

\begin{proof}
  Let us consider a Kolchin closed point $\mathfrak x\in\Diff(G_{\mathcal
K},\partial_{\vec A})$. Let $\mathcal L$ be the rational field of
$\mathfrak x$. It is an splitting field for $(G_{\mathcal
K},\partial_{\vec A})$. We have that $(G_{\mathcal L},\partial_{\vec
A})$ splits, hence $\Diff(G_{\mathcal L},\partial_{\vec A})$ is an
almost-constant differential scheme. Thus $\Diff(G_{\mathcal
L},\partial_{\vec A})$ is homeomorphic to the principal homogeneous
$G$-space $\Const(G_{\mathcal L},\partial_{\vec A})$. The
differential extension $\mathcal K\subset\mathcal L$ induces a
commutative diagram of schemes with derivation,
$$\xymatrix{(G_{\mathcal L}, \partial_{\vec A}) \ar[rr] \ar[d]\times_{\mathcal C} G & &
(G_{\mathcal L},\partial_{\vec A}) \ar[d]^-{\pi_1}\\
(G_{\mathcal K}, \partial_{\vec A}) \times_{\mathcal C} G \ar[rr] &
& (G_{\mathcal K},\partial_{\vec A})}$$ and thus, a commutative
diagram of differential schemes,
$$\xymatrix{\Diff(G_{\mathcal L}, \partial_{\vec A}) \ar[rr] \ar[d]\times_{\mathcal C} G & &
\Diff(G_{\mathcal L},\partial_{\vec A}) \ar[d]^-{\pi_2}\\
\Diff(G_{\mathcal K}, \partial_{\vec A}) \times_{\mathcal C} G
\ar[rr] & & \Diff(G_{\mathcal K},\partial_{\vec A})}.$$ Let
$\mathfrak s$ be a Kolchin closed point of $\Diff(G_{\mathcal K},
\partial_{\vec A})$. The projection $\pi_2$ of the above diagram is
exhaustive. Consider any $\mathfrak p \in \pi_2^{-1}(\mathfrak s)$,
and let us consider a Kolchin closed point $x$ in the closure
$\overline{\{\mathfrak p\}}$. Thus, $\pi_2(x)$ is in the closure
$\overline{\{\mathfrak s\}}$. As $\mathfrak s$ is a Kolchin closed
point we know that $\pi_2(x) = \mathfrak s$. Hence, there is a
Kolchin closed point $x\in \Diff(G_{\mathcal L},\partial_{\vec A})$
such that $\pi_2(x)=\mathfrak s$.

Consider two Kolchin closed points $\mathfrak s, \mathfrak y\in
\Diff(G_{\mathcal K},\partial_{\vec A})$. Because of the above
argument there are two Kolchin closed points $x,y\in
\Diff(G_{\mathcal L},\partial_{\vec A})$ such that $\pi_2(x) =
\mathfrak s$ and $\pi_2(y)=\mathfrak y$. The set of Kolchin closed
points of $\Diff(G_{\mathcal L},\partial_{\vec A})$ is a $G(\mathcal
C)$-homogeneous space in the set theoretical sense. Then there is
$\sigma\in G(\mathcal C)$ such that $x\cdot\sigma = y$, and by the
commutativity of the diagram we have $\mathfrak s \cdot \sigma =
\mathfrak y$.
\end{proof}

\begin{corollary}\label{C3CORuniqueness}
  Let $\mathfrak x$ and $\mathfrak y$ be two closed points of
$\Diff(G_{\mathcal K},\partial_{\vec A})$. Then there exists an
invertible $\mathcal K$-isomorphism of differential fields
$\kappa(\mathfrak x)\simeq \kappa(\mathfrak y)$.
\end{corollary}

\begin{proof}
  There is a closed point $\sigma\in G$, such that $\mathfrak x\cdot
\sigma = \mathfrak y$. Then
$$R_\sigma\colon (G_{\mathcal K},\partial_{\vec A})\to (G_{\mathcal K},\partial_{\vec A})$$
is an automorphism that maps $\mathfrak x$ to $\mathfrak y$. Then it
induces an invertible  $\mathcal K$-isomorphism
$$R_{\sigma}^\sharp\colon\kappa(\mathfrak y) \to \kappa(\mathfrak
x).$$
\end{proof}

\index{Galois!extension}
\begin{definition}\label{C3DEFGaloisExt}
  For each closed point $\mathfrak x\in\Diff(G_{\mathcal
K},\partial_{\vec A})$ we say that the differential extension
$\mathcal K\subset \kappa(\mathfrak x)$ is a Galois extension
associated to the non-autonomous differential algebraic dynamical
system $(G_{\mathcal K},\partial_{\vec A})$.
\end{definition}

{\bf Notation. }\emph{
  As we have proven, all Galois extensions associated to
$(G_{\mathcal K},\partial_{\vec A})$ are isomorphic. From now on let
us choose a closed point $\mathfrak x$ and denote by $\mathcal
K\subset \mathcal L$ its corresponding Galois extension.}

\begin{proposition}\label{PrpSplitL}
  A Galois extension is a minimal splitting extension for
$(G_{\mathcal K},\partial_{\vec A})$ in the following sense: If
$\mathcal K\subset \mathcal S$ is any splitting extension for
$(G_{\mathcal K},\partial_{\vec A})$ then there is a $\mathcal
K$-isomorphism of differential fields $\mathcal
L\hookrightarrow\mathcal S$.
\end{proposition}

\begin{proof}
  If $\mathcal K\subset \mathcal S$ is an splitting extension, then
$(G_{\mathcal S},\partial_{\vec A})$ splits. Hence, for each Kolchin
closed differential point $x\in\Diff(G_{\mathcal S},\partial_{\vec
A})$ the rational field of $x$ is $\mathcal S$. Let us consider the
natural projection $\pi\colon(G_{\mathcal S},\partial_{\vec A})\to
(G_{\mathcal K},\partial_{\vec A})$. We can choose a Kolchin closed
point $x\in \Diff(G_{\mathcal K},\partial_{\vec A})$ such that
$\pi(x) = \mathfrak x$. We have a morphism of $\mathcal
K$-differential algebras between the corresponding rational fields
$\pi^{\sharp}\colon \mathcal L \to \mathcal S$.
\end{proof}

\index{Picard-Vessiot extension}
\begin{example}[Picard-Vessiot extensions]
Let us consider system of $n$ linear differential equations
$$\partial x = Ax,\quad A\in gl(n,\mathcal K),$$
and let us denote $a_{ij}$ for the matrix elements of $A$. The
algebraic construction of the Picard-Vessiot extension is done as
follows (cf. \cite{Ko1} and \cite{Vanderput}):

Let us consider the algebra $\mathcal K[u_{ij}, \Delta]$, being
$\Delta=|u_{ij}|^{-1}$ the inverse of the determinant. Note that it
is the algebra of regular functions on the affine group
$GL(n,\mathcal K)$. If is an affine group, and then it is isomorphic
to the spectrum $$GL(n,\mathcal K) = \Spec(\mathcal K[u_{ij},
\Delta]).$$ We define the following derivation,
$$\partial_{\vec A} u_{ij} = \sum_{k=1}^n a_{ik}u_{jk},$$
that gives to $\mathcal K[u_{ij},\Delta]$ the structure of
differential $\mathcal K$-algebra, and to  $(GL(n,\mathcal
K),\partial_{\vec A})$ the structure of automorphic system. The set
of Kolchin closed differential points od $\Diff(GL(n,\mathcal
K),\partial_{\vec A})$ is the set of maximal differential ideals of
$\mathcal R$. A Picard-Vessiot algebra is a quotient algebra
 $\mathcal K \subset \mathcal \mathcal K[u_{ij},
\Delta]/\mathfrak m$, and a Picard-Vessiot extension is a rational
differential field $\mathcal K \subset \kappa(\mathfrak m)$. It is
self-evident that the Picard-Vessiot extension is the particular
case of Galois extension when the considered group is the general
linear group.
\end{example}

\begin{lemma}\label{lmClosedPoint}
  Let $\mathcal K\subset \mathcal S$ be a splitting extension. The canonical
projection  $$\pi\colon \Diff(G_{\mathcal S},\partial_{\vec A}) \to
\Diff(G_{\mathcal K}, \partial_{\vec A})$$ is a closed map.
\end{lemma}

\begin{proof}
  It is enough to prove that the projection $\mathfrak y =\pi(y)$ of
a closed point $y\in \Diff(G_{\mathcal S},\partial_{\vec A})$ is a
closed point. Let us take a closed point $\mathfrak z\in
\overline{\{\mathfrak y\}}$. Then $\pi^{-1}(\mathfrak z)$ is closed
and there is a closed point $z\in \pi^{-1}(\mathfrak z)$.
$\Diff(G_{\mathcal S},
\partial_{\vec A})$ is a principal homogeneous $G$-space,
there is a $\sigma\in G(\mathcal C)$ such that $z\cdot \sigma = y$,
and then $\mathfrak z \cdot \sigma = \mathfrak y$. $G(\mathcal C)$
acts transitively in the space of closed points, and $\mathfrak z$
is closed, so that we have proven that $\mathfrak y$ is closed. In
fact $\mathfrak y $ and $\mathfrak z$ are the same differential
point.
\end{proof}

\begin{proposition}\label{PrpClosedProj}
  Let us consider any intermediate
differential extension, $\mathcal K \subset \mathcal F \subset
\mathcal S$, with $\mathcal K \subset \mathcal S$ an splitting
extension. The projection, $$\pi\colon\Diff(G_{\mathcal
F},\partial_{\vec A})\to \Diff(G_{\mathcal K},\partial_{\vec A}),$$
is a closed map.
\end{proposition}

\begin{proof}
  Let us consider the following diagram of projections:
$$\xymatrix{ \Diff(G_{\mathcal S},\partial_{\vec A})
\ar[rr]^-{\pi_1} \ar[rd]^-{\pi_2} & & \Diff(G_{\mathcal K},\partial_{\vec A}) \\
 & \Diff(G_{\mathcal F},\partial_{\vec A}) \ar[ru]^-{\pi} }$$ By
Lemma \ref{lmClosedPoint} $\pi_1$ and $\pi_2$ are closed and
surjective. Then $\pi$ is closed.
\end{proof}

\begin{lemma}\label{lmFSBaseChange}
  Let $\mathcal K \subset \mathcal F \subset \mathcal L$
be an intermediate differential extension of the Galois extension of
$(G_{\mathcal K},\partial_{\vec A})$, and $\sigma$ the fundamental
solution associated to $\mathfrak x$. Let us consider the sequence
of base changes,
$$\xymatrix{ \Diff(G_{\mathcal L},\partial_{\vec A}) \ar[r]^{\pi_1} & \Diff(G_{\mathcal F},\partial_{\vec A}) \ar[r]^{\pi_2} & \Diff(G_{\mathcal K},\partial_{\vec A}) \\
\quad \sigma \quad \ar[r] & \quad \mathfrak y \quad \ar[r] & \quad
 \mathfrak x \quad,}$$ then $\mathfrak y$ is closed in Kolchin topology, $\kappa(\mathfrak y)$ is the
Galois extension $\mathcal L$ and $\sigma$ is the fundamental
solution associated with $\mathfrak y$.
\end{lemma}

\begin{proof}
  By Proposition \ref{lmClosedPoint} $\pi_1$ is a closed map, so that
$\mathfrak y$ is a closed point. The chain of projections induces a
chain of differential extensions  $\kappa(\mathfrak x) \subseteq
\kappa(\mathfrak y) \subseteq \kappa(\sigma)$ but $\kappa(\mathfrak
x) = \kappa(\sigma)$, and then we have the equality.
% Let $\mathfrak z$ be a closed point in the
%Kolchin closure of $\mathfrak y$, then we have that
%$\kappa(\mathfrak z)$ is a split extension for $(G_{\mathcal
%F},\partial_{\vec A})$, and then by Proposition \ref{PrpSplitL} we
%have $\kappa(\mathfrak z) \subseteq \kappa(\mathfrak y)$, but this
%is suitable if and only if $\mathfrak z = \mathfrak y$.
\end{proof}

\subsection{Galois Group}

  Here we give a purely geometrical definition for the Galois group
associated to a Kolchin closed differential point. We prove strong
normality of the Galois extensions, and identify our
geometrically-defined Galois group with the group of automorphisms
of the Galois extension. Let us consider the action of $G$ on
$\Diff(G_{\mathcal K},\partial_{\vec A})$ shown in Subsection
\ref{C3SS3.3.1}:
$$\Diff(G_{\mathcal K},\partial_{\vec A}) \times_{\mathcal C} G \to
\Diff(G_{\mathcal K},\partial_{\vec A}).$$

\index{Galois!group}
\begin{definition}\label{C3DEFGalois}
  Let $\mathfrak x\in \Diff(G_{\mathcal K},\partial_{\vec A})$ be
a Kolchin closed differential point. We call Galois group of the
system $(G_{\mathcal K},\partial_{\vec A})$ in $\mathfrak x$ to the
isotropy subgroup of $\mathfrak x$ in $G$ by the above action, and
denote it by $\Gal_{\mathfrak x}(G_{\mathcal K},\partial_{\vec A})$.
\end{definition}

\begin{proposition}\label{C3PROalgebraicgroup}
  $\Gal_{\mathfrak x}(G_{\mathcal K},\partial_{\vec A})$ is an
algebraic subgroup of $G$.
\end{proposition}

\begin{proof}
Denote by $H_{\mathfrak x}$ the Galois group in $\mathfrak x$. Let
us consider the projection $\pi_1$ from $G_{\mathcal K}$ to $G$
induced by the extension $\mathcal C\subset \mathcal K$. Denote by
$x$ the point $\pi_1(\mathfrak x)$, and let $U$ be an affine
neighborhood of $x$. Then $U = G \setminus Y$ with $Y$ closed in
$G$.

$U_{\mathcal K}$ is an affine neighborhood of $\mathfrak x$ in
$G_{\mathcal K}$. We have that the ring of regular functions in
$U_{\mathcal K}$ is the tensor product $\mathcal
O_{G}(U)\otimes_{\mathcal C} \mathcal K$. We identify $\mathfrak x$
with a maximal prime differential ideal $\mathfrak x\subset \mathcal
O_G(U)\otimes_{\mathcal C}\mathcal K$. Let us consider a $\mathcal
C$-point $\sigma$ of $G$. Then, for each $f\in \mathcal
O_G(U)\otimes_{\mathcal C}\mathcal K$ we have that the right
translate $R_{\sigma}^\sharp(f)$ is in $\mathcal O_G(U\cdot
\sigma^{-1})\otimes_{\mathcal C}\mathcal K$.

The morphism
$$\pi_2\colon G \to G, \quad \sigma \mapsto R_{\sigma}(x),$$
is algebraic, and let $W$ be the complementary in $G$ of
$\pi_2^{-1}(Y)$,
$$W = G \setminus \pi_2^{-1}(Y),$$
$W$ is an open subset in $G$ verifying:
\begin{enumerate}
\item[(a)] for all $\sigma\in W(\mathcal C)$, $x\in U\cap U\cdot
\sigma^{-1}$,
\item[(b)] $H_{\mathfrak x}\subset W$.
\end{enumerate}
 We will prove that the equations of $H_{\mathfrak x}$ in $W$ are algebraic.
Let us consider $W_1$ an affine open subset in $W$. Let
$\{\xi_1,\ldots, \xi_r\}$ be a system of generators of $\mathcal
O_G(W)$ as $\mathcal C$-algebra. The composition is algebraic,
$$\pi_3\colon U \times_{\mathcal C} W_1 \to G, \quad (y, \sigma)\mapsto y\cdot\sigma,$$
and it induces a morphism,
$$\pi_3^\sharp \colon \mathcal O_{G,x} \to (\mathcal O_{G}(U)\otimes_{\mathcal
C} \mathcal O(W_1))_{\pi_3^{-1}(x)},$$ and then for each $f\in
\mathcal O_{G,x}$, $\pi_3^{\sharp}(f) = F(\xi)$, is a rational
function in the $\xi_i$ with coefficients in $\mathcal O_{G,x}$. We
identify $\mathfrak x$ with a prime ideal of $\mathcal
O_{G}(U)\otimes_{\mathcal C}\mathcal K$. We consider a system of
generators,
$$\mathfrak x = (\eta_1,\ldots,\eta_r), \quad \eta_i\in\mathcal
O_{G}(U)\otimes_{\mathcal C}\mathcal K.$$
 Property (b) says that by the natural inclusion,
$$j\colon \mathcal O_{G}(U) \otimes_{\mathcal C} \mathcal K \to (\mathcal O_{G}(U)\otimes_{\mathcal
C} \mathcal O(W_1))_{\pi_3^{-1}(x)}\otimes_{\mathcal C}  \mathcal
K,$$ $j(\mathfrak x)$ spans a non trivial ideal of $(\mathcal
O_{G}(U)\otimes_{\mathcal C} \mathcal
O(W_1))_{\pi_3^{-1}(x)}\otimes_{\mathcal C} \mathcal K$, and then we
have a commutative diagram:
$$\xymatrix{\mathcal O_{G}(U) \otimes_{\mathcal C} \mathcal K
\ar[rr]\ar[d]& & (\mathcal O_{G}(U)\otimes_{\mathcal C} \mathcal
O(W_1))_{\pi_3^{-1}(x)}\otimes_{\mathcal
C}  \mathcal K \ar[d]^-{\pi_4} \\
\kappa(\mathfrak x) \ar[rr] & & (\kappa(\mathfrak x)
\otimes_{\mathcal C} \mathcal O(W_1))_{\pi_3^{-1}(x)}}.$$
 An element $\sigma\in W_1$ stabilizes
$\mathfrak x$ if and only if $R_{\sigma}^\sharp(\eta_i) \in
\mathfrak x$, and this is so if and only if $\pi_4(j(\eta_i)) = 0$
for $i=1,\ldots, r$. Let us consider a basis
$\{e_\lambda\}_{\lambda\in\Lambda}$ of $\kappa(\mathfrak x)$ over
$\mathcal C$. For each $i$, we have a finite sum:
$$\pi_4(j(\eta_i)) = \frac{\sum_{\alpha} G_{i\alpha}(\xi)e_\alpha}{\sum_{\beta} H_{i\beta}(\xi)e_\beta},$$
and then $G_{i\alpha}(\xi)\in \mathcal O(W_1)$ are the algebraic
equations of $H_{\mathfrak x}$ in $W_1$.
\end{proof}

\begin{remark}\label{RemarkGaloisK}
  Let $\mathfrak x$ be a Kolchin closed differential point as above,
and $H\subset G$ the Galois group of $(G_{\mathcal K},\partial_{\vec
A})$ in $\mathfrak x$. Then $H_{\mathcal K} = H \times_{\mathcal
C}\Spec(\mathcal K)$ is the stabilizer subgroup of $\overline{\{\mathfrak x\}}$, 
the Zariski closure of $\mathfrak x$, by
the action of composition by the right side:
$$G_{\mathcal K}\times_{\mathcal K} G_{\mathcal K} \to
G_{\mathcal K}.$$ However, the morphisms $R_{\sigma}$ for $\sigma\in
H_{\mathcal K}$ are not in general morphisms of schemes with
derivation. In the same sense, for any field extension $\mathcal
K\subset \mathcal L$, $H_{\mathcal L}\subset G_{\mathcal L}$ is the
stabilizer group of $\overline{\pi^{-1}(\mathfrak x)}$, the Zariski closure
of the preimage of $\mathfrak x$, where $\pi$ is the
natural projection from $G_{\mathcal L}$ to $G_{\mathcal K}$. This
means that $H_{\mathcal L}$ stabilizes the fiber, in the following
sense: for each $\mathcal L$-point $\sigma\in H_{\mathcal L}$,
$R_{\sigma}\colon G_{\mathcal L}\to G_{\mathcal L}$ induces,
$$R_{\sigma}|_{\overline{\pi^{-1}(\mathfrak x)}}\colon \overline{\pi^{-1}(\mathfrak x)}\to
\overline{\pi^{-1}(\mathfrak x)}.$$

\end{remark}

\begin{proposition}\label{C3PROcongugated}
Consider two Kolchin closed differential points $\mathfrak x,
\mathfrak y$ in $\Diff(G_{\mathcal K},
\partial_{\vec A})$. The groups $\Gal_{\mathfrak x}(G_{\mathcal
K},\partial_{\vec A})$ and $\Gal_{\mathfrak y}(G_{\mathcal
K},\partial_{\vec A})$ are isomorphic conjugated algebraic subgroups
of $G$.
\end{proposition}

\begin{proof}
  The group of $\mathcal C$-points of $G$ acts transitively in the set of closed
differential points. Hence, there exists $\sigma\in G(\mathcal C)$
with $\mathfrak x \cdot \sigma = \mathfrak y$, and then
$H_{\mathfrak x} \cdot \sigma = \sigma \cdot H_{\mathfrak y}$.
\end{proof}

\begin{theorem}\label{C3THEsne}
  The Galois extensions associated to
$(G_{\mathcal K},\partial_{\vec A})$ are strongly normal extensions.
\end{theorem}

\begin{proof}
  Let us consider a Galois extension $\mathcal K \subset \mathcal L$. Thus,
$\mathcal L$ is the rational field of certain Kolchin closed
differential point that we denote by $\mathfrak x$. Let us consider
$\sigma\in G_{\mathcal L}$ the fundamental solution associated to
$\mathfrak x$. We have that $\sigma$ projects onto $\mathfrak x$ and
the gauge transformation $L_{\sigma^{-1}}$ is a splitting morphism.
We define the morphism $\psi$ of schemes with derivation trough the
following commutative diagram:
$$\xymatrix{(G_{\mathcal L}, \partial_{\vec A})\ar[rr]^-{\pi} & & (G_{\mathcal K},\partial_{\vec A}) \\
(G_{\mathcal L}, \partial) = G \times_{\mathcal C} (\Spec(\mathcal
L),\partial) \ar[u]^-{L_{\sigma}}\ar[urr]_-{\psi}}$$

Denote by $H$ the Galois group in $\mathfrak x$. We have that
$(H_{\mathcal L},\partial)\subset (G_{\mathcal L},\partial)$ is a
closed subscheme with derivation. The group $H_\mathcal L$ is the
preimage of $H$ by the projection from $G_{\mathcal L}$ to $G$. By
remark \ref{RemarkGaloisK} $H_{\mathcal L}$ is the stabilizer of the
$\overline{\pi^{-1}(\mathfrak x)}$ in $G_{\mathcal L}$. It means that for any
point $z$ of $G_{\mathcal L}$ whose projection is addherent to $\mathfrak x$ and any
$\mathcal L$-point $\tau$ of $H_{\mathcal L}$, the right translate
$z\cdot\tau$ is also addherent to $\mathfrak x$. In particular we have
that $ \psi(\tau) = \mathfrak x$, and then
$$(H_{\mathcal L},\partial) \subset \overline{\psi^{-1}(\mathfrak x)}.$$
 Reciprocally, let us consider an $\mathcal L$-point $\tau \in \psi^{-1}(\mathfrak x)$.
Therefore $\pi(\sigma\cdot \tau)$ is addherent to $\mathfrak x$. The following
diagram is commutative:
$$\xymatrix{G_{\mathcal L} \times_{\mathcal L} G_{\mathcal L} \ar[rr] \ar[d] & & G_{\mathcal L}\ar[d]\\
G_{\mathcal K} \times_{\mathcal K} G_{\mathcal K} \ar[rr] & &
G_{\mathcal K}}$$ We deduce that, for any other preimage
$\bar\sigma$ of $\mathfrak x$ by $\pi$, the right translated
$\bar\sigma\cdot\tau$ also projects onto $\overline{\{\mathfrak x\}}$. Thus, $\tau$
stabilizes $\overline{\pi^{-1}(\mathfrak x)}$, so that $\tau\in (H_{\mathcal
L},\partial)$. Finally we have the identity:
$$\psi^{-1}(\mathfrak x) = (H_{\mathcal L},\partial) =
H\times_{\mathcal C} (\Spec(\mathcal L),\partial).$$

On the other hand we apply the affine stalk formula (Proposition
\ref{CBformula}, that comes from the classical stalk formula,
Theorem \ref{StalkFormula}, in Appendix \ref{ApA}) to $\mathfrak x$.
We obtain the isomorphism:
$$\pi^{-1}(\mathfrak x) \simeq (\Spec(\mathcal L\otimes_{\mathcal K}
\mathcal L), \partial).$$

From the definition of $\psi$ we know that $L_{\sigma}$ gives us an
isomorphism between the fibers $\pi^{-1}(\mathfrak x)$ and
$\psi^{-1}(\mathfrak x)$. This restricted morphism
$L_\sigma|_{(H_{\mathcal L},\partial)}$ is a splitting morphism
$$\xymatrix{\overline{(\Spec(\mathcal
L\otimes_{\mathcal K} \mathcal L), \partial)}\ar[rr]^-{\pi} & & \{\mathfrak x\} \\
 H \times_{\mathcal C}(\Spec(\mathcal L), \partial) \ar[u]^-{L_\sigma|_{(H_{\mathcal L},\partial)}}
 \ar[urr]_-{\psi}}$$
of the tensor product $\mathcal L\otimes_{\mathcal K}\mathcal L$. All differential point
$\tau\in \overline{(\Spec(\mathcal L\otimes_{\mathcal K} \mathcal L, \partial)}$ must be
be in the preimage of $\mathfrak x$, because of the maximality of $\mathfrak x$ as
differential point if $G_{\mathcal L}$. If follows that
$\Diff(\Spec(\mathcal L\otimes_{\mathcal K} \mathcal L, \partial) = 
\DiffSpec(\mathcal L\otimes_{\mathcal K} \mathcal L)$. And then, we obtain an isomorphism
$$\DiffSpec(\mathcal L\otimes_{\mathcal K} \mathcal L) \to H\times_{\mathcal C}\DiffSpec(\mathcal L),$$
it follows that $\mathcal K \subset \mathcal L$ is strongly normal. 
\end{proof}

\begin{remark}\label{C3REM3.2.20}
Following \cite{Kov2}, $\DiffSpec(\mathcal L \otimes_{\mathcal
K}\mathcal L)$ is the set of admissible $\mathcal K$-isomorphism of
$\mathcal L$, modulo generic specialization. In the case of a
strongly normal extension $\mathcal K\subset \mathcal L$ the space
of constants  $\Const(\DiffSpec(\mathcal L \otimes_{\mathcal
K}\mathcal L))$ is an algebraic group and its closed points
correspond to differential $\mathcal K$-algebra automorphisms of
$\mathcal L$. Let us consider the previous splitting morphism,
$$H \times_{\mathcal C} (\Spec(\mathcal L), \partial) \to
(\Spec(\mathcal L \otimes_{\mathcal K}\mathcal L), \partial)$$ if we
apply the constant functor $\Const$, we obtain a isomorphism of
$\mathcal C$-algebraic varieties,
$$H \xrightarrow{s} \Gal(\mathcal L/\mathcal K),$$ where $H$ and $\Gal(\mathcal L
/\mathcal K)$ are algebraic groups. To each $\tau\in H$, we have
$\mathfrak x\cdot \tau = \mathfrak x$, and the $R_{\tau}^\sharp
\colon \mathcal L \to \mathcal L$. We have $R_{\tau}^\sharp \circ
R_{\bar\tau}^\sharp = R_{\tau\bar\tau}^\sharp$ and it realizes $H$
as a group of differential $\mathcal K$-algebra automorphisms of
$\mathcal L$.

\end{remark}

\begin{theorem}\label{C3THEauto}
  The Galois group $\Gal_{\mathfrak x}(G_{\mathcal K},\partial_{\vec
  A})$ is the group of differential $\mathcal
K$-algebra automorphisms of the Galois extension $\mathcal K \subset
\kappa(\mathfrak x)$.
\end{theorem}

\begin{proof}
  Denote, as above, by $H\subset G$ the Galois group and by
$\mathcal L$ the Galois extension $\kappa(\mathfrak x)$. We consider
the isomorphism $s$ stated in remark \ref{C3REM3.2.20}. Let us prove
that $s$ is an isomorphism of algebraic groups over $\mathcal C$,
and that for $\tau\in H(\mathcal C)$, $s(\tau)$ is the automorphism
$R^{\sharp}_{\tau}$ of $\mathcal L$, induced by the translation
$R_\tau$.

  We already know that $s$ is a scheme isomorphism. We have to prove that it is a group
morphism. For $\tau\in H$, let us compute $s(\tau)$. First, let us
denote by $\bar\tau$ the point of $H_{\mathcal L}$ obtained from
$\tau$ after the base extension from $\mathcal C$ to $\mathcal L$.
It is a differential point of $(H_{\mathcal L},\partial)$. Then
$L_{\sigma}(\bar\tau) = R_{\tau}(\sigma)\in\pi^{-1}(\mathfrak x)$.
We identify $R_{\tau}(\sigma)$ with a differential point of
$\pi^{-1}(\mathfrak x)$. By the stalk formula we have that
$\pi^{-1}(\mathfrak x) = (\Spec(\mathcal O_{G_{\mathcal K},\mathfrak
x}\otimes_{\mathcal K}\mathcal L),\partial)$. We identify
$R_\tau(\sigma)$ with a prime differential ideal of $\mathcal
O_{G_{\mathcal K},\mathfrak x}\otimes_{\mathcal K}\mathcal L$.
Because $\pi(R_{\tau}(\sigma)) = \mathfrak x$, the morphism
$R_{\tau}(\sigma)^\sharp$ factorizes,
$$\xymatrix{\mathcal O_{G_{\mathcal K},\mathfrak x}\otimes_{\mathcal K}\mathcal L
\ar[d]^-{\mathfrak x^{\sharp}\otimes Id} \ar[rrd]^-{R_\tau(\sigma)^{\sharp}} \\
\kappa(\mathfrak x) \otimes_{\mathcal K} \mathcal L \ar[rr]_-{\psi}
& & \mathcal L}$$ and then the kernel of $\psi$ is the prime
differential ideal defining the automorphism $s(\tau)$,
$$\psi(a\otimes b) = s(\tau)(a)\cdot b$$

Let us consider the right translation $R_{\tau}$,
$$\xymatrix{G_{\mathcal L} \ar[r]^-{R_\tau}\ar[d] & G_{\mathcal L} \ar[d] \\ G_{\mathcal K} \ar[r] &
G_{\mathcal K}}\quad\quad\xymatrix{\sigma \ar[r]^-{R_\tau}\ar[d]& L_{\sigma}(\bar\tau) \ar[d] \\
\mathfrak x \ar[r] & \mathfrak x}$$ we have a commutative diagram
between the local rings,
$$\xymatrix{\mathcal L &  \mathcal L \ar[l]_-{Id} \\ \mathcal  O_{G_{\mathcal L},\pi^{-1}(\mathfrak x)}
\ar[u]^-{\sigma^\sharp}
& \mathcal O_{G_{\mathcal L},\pi^{-1}(\mathfrak x)} \ar[l]_-{R_{\tau}^{\sharp}} \ar[u]_-{R_\tau(\sigma)^\sharp} \\
\mathcal O_{G_{\mathcal K},\mathfrak x} \ar[u] & \mathcal
O_{G_{\mathcal K},\mathfrak x} \ar[l]^-{R_{\tau}^\sharp} \ar[u] },$$
where $\mathcal O_{G,\pi^{-1}(\mathfrak x)} = \mathcal
O_{G_{\mathcal K},\mathfrak x}\otimes_{\mathcal K}\mathcal L$, and
the morphism $R_{\tau}^\sharp$ on these rings is defined as follows:
$$\mathcal O_{G_{\mathcal K},\mathfrak x}\otimes_{\mathcal K}\mathcal
L \to \mathcal O_{G_{\mathcal K},\mathfrak x}\otimes_{\mathcal
K}\mathcal L, \quad a\otimes b \mapsto R_{\tau}^\sharp(a)\cdot b.$$
It is then clear that morphism $\psi$ defined above sends,
$$\psi\colon (a\otimes b) \mapsto R_{\tau}^\sharp(a)\cdot b$$
and then its kernel defines the automorphism $R_{\tau}^\sharp$ and
we finally have found $R_{\tau}^\sharp = s(\tau)$.
\end{proof}

\index{Galois!correspondence}
\subsection{Galois Correspondence}

  There is a Galois correspondence for strongly normal extensions (theorem
\ref{ThGaloisCorrespondence}). It is naturally transported to the
context of algebraic automorphic systems. Let $\mathcal L$ be a
Galois extension, which is the rational field $\kappa(\mathfrak x)$
of a Kolchin closed point $\mathfrak x$ as above. Let $\mathcal F$
be an intermediate differential extension,
$$\mathcal K \subset \mathcal F \subset \mathcal L.$$
We make base extensions sequentially so that we obtain a sequence of
schemes with derivations,
$$(G_{\mathcal L},\partial_{\vec A}) \to (G_{\mathcal F},\partial_{\vec A})) \to (G_{\mathcal K},\partial_{\vec A}),$$
and the associated sequence of differential schemes,
$$\Diff(G_{\mathcal L},\partial_{\vec A}) \to \Diff(G_{\mathcal
F},\partial_{\vec A}) \to \Diff(G_{\mathcal K},\partial_{\vec A}).$$
Let $\sigma\in G(\mathcal L)$ be the fundamental solution induced by
$\mathfrak x$. We obtain a sequence of differential points:
$$\sigma \mapsto \mathfrak y \mapsto \mathfrak x.$$
They are Kolchin closed and $\sigma$ is the fundamental solution
associated to $\mathfrak x$ and $\mathfrak y$ (Lemma
\ref{lmFSBaseChange}). The stabilizer subgroup of $\mathfrak y$ is a
subgroup of the stabilizer subgroup of $\mathfrak x$. We have
inclusions of algebraic groups,
$$\Gal_{\mathfrak y}(G_{\mathcal F},\partial_{\vec A}) \subset
\Gal_{\mathfrak x}(G_{\mathcal K},\partial_{\vec A}) \subset G.$$ In
particular we have that $\mathcal K \subset \mathcal F$ is a
strongly normal extension if and only if
$$\Gal_{\mathfrak y}(G_{\mathcal F},\partial_{\vec A}) \lhd
\Gal_{\mathfrak x}(G_{\mathcal K},\partial_{\vec A}).$$

\begin{proposition}
  Assume that $\Gal_{\mathfrak x}(G_{\mathcal K},\partial_{\vec A})$ is the whole group $G$, and
$\mathcal K\subset \mathcal F$ is a strongly normal extension. Then
the quotient group $$\bar G = G/\Gal_{\mathfrak y}(G_{\mathcal
F},\partial_{\vec A})$$ exists. Let $\vec B$ be the projection of
$\vec A$ in $\mathcal R(\bar G)\otimes_{\mathcal C}\mathcal K$.
Then, there is a unique closed differential point $\mathfrak z\in
\Diff(\bar G_{\mathcal K},\partial_{\vec B})$, and,
$$\Gal_{\mathfrak z}(\bar G_{\mathcal K},\partial_{\vec B}) = \bar G.$$
\end{proposition}

\begin{proof}
  The quotient realizes itself as the group of automorphisms of
the differential $\mathcal K$-algebra $\mathcal F$. The extension
$\mathcal K \subset\mathcal F$ is strongly normal, and then this
group is algebraic by Galois correspondence (Theorem
\ref{ThGaloisCorrespondence}). The induced morphism
$$\pi\colon \Diff(G_{\mathcal
K},\partial_{\vec A}) \to \Diff(\bar G_{\mathcal K},\partial_{\vec
A})$$ restricts to the differential points, and it is surjective.
The hypothesis $\Gal_{\mathfrak x}(G_{\mathcal K},\vec A)=G$ implies
that $\Diff(G_{\mathcal K},\partial_{\vec A})$ consist in the only
point $\{\mathfrak x\}$, and then $\Diff(\bar G_{\mathcal
K},\partial_{\vec A}) = \{\mathfrak z\}$. Hence, $\mathfrak z$ is
the generic point of $G_{\mathcal K}$ and the Galois group is the
total group.
\end{proof}

Reciprocally let us consider an algebraic subgroup $H\subset
\Gal_{\mathfrak x}(G_{\mathcal K},\partial_{\vec A})$. Then $H$ is a
subgroup of differential $\mathcal K$-algebra automorphisms of
$\mathcal L$. Let $\mathcal F =\mathcal L^H$ be its field of
invariants. We have again a sequence of non-autonomous algebraic
dynamical systems
$$(G_{\mathcal L},\partial_{\vec A}) \to (G_{\mathcal F}, \partial_{\vec A})
\to (G_{\mathcal K}, \partial_{\vec A}).$$ Let again $\sigma$ be the
fundamental solution induced by $\mathfrak x$, we have the sequence
of closed differential points,
$$\sigma\mapsto \mathfrak y \mapsto \mathfrak x$$

\begin{proposition}\label{Prop318}
Let us consider an intermediate differential field,
$$\mathcal K\subset \mathcal F \subset \mathcal L,$$ as above, and $H = \Aut(\mathcal L/\mathcal
F)$, then
\begin{enumerate}
\item[(a)] $H$ is the Galois group $\Gal_{\mathfrak y}(G_{\mathcal F},\partial_{\vec A})
\subset \Gal_{\mathfrak x}(G_{\mathcal K},\partial_{\vec A})$.
\item[(b)] $\mathcal K\subset\mathcal F$ is strongly normal if and
only if $H\lhd \Gal_{\mathfrak x}(G_{\mathcal K},\partial_{\vec
A})$. In such case $\Aut(\mathcal F/\mathcal K) = \Gal_{\mathfrak
x}(G_{\mathcal K},\partial_{\vec A})/H$.
\end{enumerate}
\end{proposition}

\begin{proof} By considering the identification of the Galois group with
the group of automorphisms, the result is a direct translation of
the Galois correspondence for strongly normal extensions (see
\cite{Kov2} Theorem 20.5, Theorem \ref{ThGaloisCorrespondence} in
this text).
\end{proof}

  In particular, each algebraic group admits a unique normal subgroup
of finite index, the connected component of the identity. Let
$\Gal^0_{\mathfrak x}(G_{\mathcal K},\partial_{\vec A})$ be the
connected component of the identity of $\Gal_{\mathfrak
x}(G_{\mathcal K},\partial_{\vec A})$ and,
$$\Gal^1_{\mathfrak
x}(G_{\mathcal K},\partial_{\vec A}) = \Gal_{\mathfrak
x}(G_{\mathcal K},\partial_{\vec A})/\Gal^0_{\mathfrak
x}(G_{\mathcal K},\partial_{\vec A}),$$ which is a finite group. In
such case we have:
\begin{enumerate}
\item[(a)] The invariant field $\mathcal L^{\Gal^0_{\mathfrak x}(G_{\mathcal K},\partial_{\vec A})}$
is the relative algebraic closure ${\mathcal K^\circ}$ of $\mathcal
K$ in $\mathcal L$.
\item[(b)] $\mathcal K \subset {\mathcal K^\circ}$ is an algebraic
Galois extension of Galois group $\Gal^1_{\mathfrak x}(G_{\mathcal
K},\partial_{\vec A})$.
\item[(c)] $\Gal_{\mathfrak y}(G_{{\mathcal K^\circ}},\partial_{\vec
A}) = \Gal^0_{\mathfrak x}(G_{\mathcal K},\partial_{\vec A})$.
\end{enumerate}
Thus, we can set out:

\begin{proposition}
  $\mathcal K$ is relatively algebraically closed in $\mathcal L$ if
and only if its Galois group is connected.
\end{proposition}

\subsection{Galois Correspondence and Group Morphisms}

 Here, we relate the Galois correspondence and the projection
of automorphic vector fields through algebraic group morphisms. It
is self evident that a group morphism $\pi\colon G\to \bar G$ sends
an automorphic system $\vec A$ in $G$ with coefficients in $\mathcal
K$ to an automorphic system $\pi(\vec A)$ in $\bar G$ with
coefficients in $\mathcal K$. Furthermore we know that $\pi(\vec A)$
is an automorphic system in the image of $\pi$ which is a subgroup
of $\bar G$. By restricting our analysis to this image, we can
assume that $\pi$ is a surjective morphism.

\begin{theorem}\label{ThGM}
  Let $\pi \colon G\to \bar G$ be a surjective morphism of algebraic groups,
and $\vec B$ the projected automorphic system $\pi(\vec A)$. Then:
\begin{itemize}
\item[(1)] $\mathfrak y = \pi(\mathfrak x)$ is a closed differential
point of $\Diff(\bar G_{\mathcal K},\partial_{\vec B})$.
\item[(2)] $\kappa(\mathfrak y)$ is a strongly normal intermediate extension of
$\mathcal K \subset \kappa(\mathfrak y) \subset \mathcal L$.
\item[(3)] $\Gal_{\mathfrak y}(\bar G_{\mathcal K},\partial_{\vec B}) =
\Gal_{\mathfrak x}(G_{\mathcal K},\partial_{\vec A})/(\ker (\pi)\cap
\Gal_{\mathfrak x}(G_{\mathcal K},\partial_{\vec A}))$.
\item[(4)] Let $\mathfrak z$ be a Kolchin closed point of $(G_{\kappa(\mathfrak
y)},\partial_{\vec A})$ in the fiber of $\mathfrak x$. Then
$\Gal_{\mathfrak z}(G_{\kappa(\mathfrak y)},\partial_{\vec A}) =
\ker(\pi) \cap \Gal_{\mathfrak x}(G_{\mathcal K},\partial_{\vec A})$
\end{itemize}
\end{theorem}

\begin{proof} (1) Let $\mathfrak s$ be a closed point of
$\Diff(\bar G_{\mathcal K},\partial_{\vec B})$ adherent to
$\mathfrak y$. Then $\pi^{-1}(\mathfrak x)$ is a closed subset of
$\Diff(G_{\mathcal K},\partial_{\vec A})$ and it contains a closed
point $\mathfrak z$. $G(\mathcal C)$ acts transitively in the set of
closed points, and then there is $\tau\in G(\mathcal C)$ such as
$\mathfrak x = \mathfrak z \cdot \tau$. Thus, $\mathfrak y =
\mathfrak s\cdot \pi(\tau)$, so that $\mathfrak y$ is closed,
$\mathfrak s = \mathfrak x$, and furthermore $\pi(\tau) \in
\Gal_{\mathfrak y}(\bar G_{\mathcal K},\partial_{\vec B})$.

(2) $\pi^\sharp\colon \kappa(\mathfrak y)\to \mathcal L$ is a
differential $\mathcal K$-algebra morphism, and $\kappa(\mathfrak
y)$ is realized as an intermediate extension $\mathcal K \subset
\kappa(\mathfrak y) \subset \mathcal L$. It is a strongly normal if
and only if the subgroup of $\Gal(\mathcal L/\mathcal K)$ fixing
$\kappa(\mathfrak y)$ is a normal subgroup. We identify
$Gal(\mathcal L/\mathcal K)$ with $\Gal_{\mathfrak x}(G_{\mathcal
K}, \partial_{\vec A})$. Then $\tau$ fixes $\kappa(\mathfrak y)$ if
and only if $\pi(\tau) = e$. This subgroup fixing $\kappa(\mathfrak
y)$ is $\ker(\pi)\cap \Gal_{\mathfrak x}(G_{\mathcal K},
\partial_{\vec A})$. By hypothesis, $\ker(\pi)$ is a normal subgroup
of $G$, and then its intersection with $\Gal_{\mathfrak
x}(G_{\mathcal K}, \partial_{\vec A})$ is a normal subgroup.

Finally, be obtain (3) and (4) by Galois correspondence.
\end{proof}

\subsection{Lie Extension Structure on Intermediate Fields}

  Differential field approach to Lie-Vessiot systems was initiated
  by K. Nishioka, in terms of the notions of rational dependence on
  arbitrary constants and Lie extensions (see definitions \ref{DefRationalDependence} and
  \ref{DefLieExtension}). Here we relate our results
  with these notions.

\begin{theorem}\label{Nihioka}
  Assume one of the following:
\begin{enumerate}
\item[(a)]  $\mathcal K$ is algebraically closed.
\item[(b)] The Galois group of $(G_{\mathcal K},\partial_{\vec A})$
is $G$.
\end{enumerate}
  Let $y$ be a particular solution of $(M_{\mathcal K}, \partial_{\vec
  X})$ with coefficients in a differential field extension $\mathcal
  K\subset \mathcal R$. Assume that $\mathcal R$ is generated by
  $y$. Then:
  \begin{enumerate}
  \item[(i)] $\mathcal K \subset \mathcal R$ depends rationally on
  arbitrary constants.
  \item[(ii)] $\mathcal K \subset \mathcal R$ is a Lie extension.
  \end{enumerate}
\end{theorem}

\begin{proof}

  (i) $\mathcal R$ is an intermediate extension of the splitting
  field of the automorphic system which is a strongly normal
  extension. It is a stronger condition than the one of Definition
  \ref{DefRationalDependence}, thus $\mathcal R$ depends rationally
  on arbitrary constants.

  (ii) If $\mathcal K$ is algebraically closed, then the result
  comes directly from Theorem \ref{ThNishioka2}. For the case $(b)$,
  some analysis on the infinitesimal structure of $\mathcal R$
  is must be done. If the Galois group is $G$, then there are not non-trivial
  differential points in $G_{\mathcal K}$, nor in $M_{\mathcal K}$.
  Then $\mathcal R$ coincides with $\mathcal M(M_{\mathcal
  K})$, the field of meromorphic functions in $M_{\mathcal K}$.
  Fundamental vector fields of the action of
  $G$ on $M$ induce derivations of the corresponding fields of
  meromorphic functions so that we have a Lie algebra morphism,
  $$\mathcal R(G) \to \Der_{\mathcal K}(\mathcal R), \quad \vec A_i
  \mapsto \vec X_{i},$$
  and the derivation in $\partial$ in $\mathcal R$ is seen in
  $\mathcal M(\mathcal R)$ as the Lie-Vessiot system
  $$\bar \partial = \partial + \sum_{i=1}^r f_i \vec X_i.$$
  From that, we have that,
  $$[\bar \partial, \mathcal R(G)] \subset \mathcal R(G)
  \otimes_{\mathcal C} \mathcal K,$$
  and because the vector fields $\vec X_{i}$ span the tangent vector
  space to $M$, we have that the morphism,
  $$\mathcal R(G) \otimes_{\mathcal C} \mathcal R \to \Der_{\mathcal
  K}(\mathcal R)$$
  is surjective. According to Definition \ref{DefLieExtension} we conclude that $\mathcal R$
  is a Lie extension.
\end{proof}

\section{Algebraic Reduction and Integration}\label{C4}

  Here we present the algebraic theory of reduction and integration
of algebraic automorphic and Lie-Vessiot systems. Our main tool is
an algebraic version of Lie's reduction method, that we call
\emph{Lie-Kolchin} reduction. Once we have developed this tool we
explore different applications.

\subsection{Lie-Kolchin Reduction Method}

  In \cite{BMCharris}, when discussing the general topic of analytic
Lie-Vessiot systems, we have shown the Lie's method for reducing an
automorphic equation to certain subgroups, once we know certain
solution of a Lie-Vessiot associated system. This method is local,
because it is assumed that we can choose a suitable curve in the
group for the application of the algorithm. A germ of such a curve
exists, but it is not true that a suitable global curve exists in
the general case. In the algebraic realm we will find obstructions
to the applicability of this method, highly related to the structure
of principal homogeneous spaces over a non algebraically closed
field, and then to Galois cohomology.

  We will show that the application of the Lie's method in the algebraic
case leads us directly to Kolchin reduction theorem of a linear
differential system to the Lie algebra of its Galois group. Because
of this, we decided to use the nomenclature of \emph{Lie-Kolchin
reduction method}.

\subsection{Lie-Kolchin Reduction}

\emph{From now on, let us consider a differential field $\mathcal K$
of characteristic zero. The field of constant is $\mathcal C$, that
we assume to be algebraically closed. Let $G$ be an algebraic group
over $\mathcal C$, and let $\vec A$ be an algebraic automorphic
vector field in $G$ with coefficients in $\mathcal K$. We also fix a
Kolchin closed point $\mathfrak x$ of $\Diff(G_{\mathcal K},
\partial_{\vec A})$ and denote by $\mathcal L$ its associated Galois
extension.}

\begin{lemma}\label{LmPrincipalStalk}
  Let $G'\subset G$ be an algebraic subgroup, and let $M$ be the
quotient homogeneous space $G/G'$. Then:
\begin{enumerate}
\item[(a)] $M_{\mathcal K} = G_{\mathcal K}/G'_{\mathcal K}$
\item[(b)]  Let us consider the natural projection morphism $\pi_{\mathcal
K}\colon G_{\mathcal K} \to M_{\mathcal K}$. For each rational point
$x\in M_{\mathcal K}$, $\pi_{\mathcal K}^{-1}(x) \subset G_{\mathcal
K}$ is an homogeneous space of group $G'_{\mathcal K}$.
\end{enumerate}
\end{lemma}

\begin{proof} (a) $\mathcal C$ is algebraically closed, and then the geometric
quotient is universal; (a) is the fundamental property of geometric
universal quotients (see \cite{Sa}). (b) The isotropy subgroup $H_x$
of $x$ is certain algebraic subgroup isomorphic and conjugated with
$G'_\mathcal K$. The action of $(H_x)_{\mathcal K}$ on $G$ preserves
the stalk $\pi_{\mathcal K}^{-1}(x)$,
$$\psi\colon (H_x)_{\mathcal K} \times_{\mathcal K} \pi_{\mathcal K}^{-1}(x) \to \pi_{\mathcal
K}^{-1}(x),$$ the induced morphism
$$(\psi\times Id) \colon (H_x)_{\mathcal K} \times_{\mathcal K} \pi_{\mathcal K}^{-1}(x) \to \pi_{\mathcal
K}^{-1}(x)\times_{\mathcal K} \pi_{\mathcal K}^{-1}(x)$$ is the
restriction of the isomorphism
$$G_{\mathcal K} \times_{\mathcal K} G_{\mathcal K} \to G_{\mathcal
K} \times_{\mathcal K} G_{\mathcal K}, \quad (\tau,\sigma) \mapsto
(\tau\cdot\sigma, \sigma),$$ and then it is an isomorphism.
\end{proof}

 Let $M$ be an homogenous space over $G$, and $\vec X$ the Lie-Vessiot
vector field induced in $M$ by the automorphic vector field $\vec
A$. Let us fix a rational point $x_0$ of $M$ and denote by $H_{x_0}$
the isotropy subgroup at $x_0$.

\begin{lemma}\label{lmSolutionx0}
  Assume that $x_0\in M$ is a constant solution of $(M_{\mathcal K},
\partial _{\vec X})$. Then: $$\vec A \in \mathcal
R(H_{x_0})\otimes_{\mathcal C} \mathcal K.$$
\end{lemma}

\begin{proof}
  There is a solution $\tau$ of $\vec A$ with coefficients in $\mathcal L$ such
that $x_0 = \tau\cdot x_0$. Therefore $\tau \in (H_{x_0})_{\mathcal
L}$ and its logarithmic derivative is an automorphic vector field in
$H_{x_0}$,
$$\l\partial(\tau) \in \mathcal R(H_{x_0}) \otimes_{\mathcal C}
\mathcal L.$$ Taking into account that $l\partial(\tau) = \vec A$,
we obtain $\vec A \in \mathcal R(H_{x_0})\otimes_{\mathcal
C}\mathcal K.$
\end{proof}

\begin{theorem}[Main Result]\label{C4THE4.1.5}
  Let us assume that $(M_{\mathcal K}, \partial_{\vec X})$ has a
solution $x$ with coefficients in $\mathcal K$. If $H^1(H_{x_0},
\mathcal K)$ is trivial, then there exists a gauge transformation
$L_{\tau}$ of $G_{\mathcal K}$ that sends the automorphic vector
field $\vec A$ to:
$$\vec B = \Adj_{\tau}(\vec A) + l\partial(\tau),$$
with $\vec B \in \mathcal R(H_{x_0}) \otimes_{\mathcal C} \mathcal
K$ an automorphic vector field in $H_{x_0}$.
\end{theorem}

\begin{proof}
  Let us consider the canonical isomorphism $G/H_{x_0}\to M$ that sends the class $[\sigma]$ to
$\sigma\cdot x_0$. Now, let us consider the base extended morphism,
  $$\pi\colon G_{\mathcal K} \to M_{\mathcal
K}, \qquad \tau \mapsto \tau\cdot x_0.$$ We are under the hypothesis
of Lemma \ref{LmPrincipalStalk} (b). Therefore the stalk
$\pi^{-1}(x)$ is a principal homogeneous space of group
$(H_{\mathcal K})_x$ which is a subgroup of $G_{\mathcal K}$
conjugated to $(H_{x_0})_{\mathcal K}$. Because of the vanishing of
the Galois cohomology, there exist a rational point $\tau_1\in
\pi^{-1}(x)$, and then $\tau_1\cdot x_0 = x$. Define $\tau =
\tau_1^{-1}$. Let us consider the gauge transformation,
$$L_{\tau}\colon (G_{\mathcal K},\partial_{\vec A}) \to (G_{\mathcal
K}, \partial_{\vec B}) \quad\quad L_\tau\colon (M_{\mathcal
K},\partial_{\vec X}) \to (M_\mathcal K, \partial_{\vec Y}),$$ where
$\vec Y$ is the Lie-Vessiot vector field in $M$ induced by $\vec B$.
We have that $\tau\cdot x = x_0$ is a constant solution of
$(M_{\mathcal K},
\partial_{\vec Y})$. By Lemma \ref{lmSolutionx0},
$\vec B$ is an automorphic field in $H_{x_0}$.
\end{proof}

\begin{proposition}\label{PrpRationalX}
Assume that there is a rational point $x_0\in M$ such that
$\Gal_{\mathfrak x}(G_{\mathcal K},\partial_{\vec A})\subset
H_{x_0}$, then there exists a rational solution $x\in M(\mathcal K)$
of $\vec X$.
\end{proposition}

\begin{proof}
  Let us consider the fundamental solution $\sigma$ associated to
$\mathfrak x$. We consider it as an $\mathcal L$-point of $G$,
$$\sigma\colon \Spec(\mathcal L)\to G_{\mathcal K}.$$
It is determined by the canonical morphism of \emph{taking values in
$\sigma$},
$$\sigma^\sharp \colon \mathcal O_{G_{\mathcal K},\mathfrak x} \to \mathcal L = \kappa(\mathfrak x).$$

  Now, let us consider the projection $\pi\colon G\to M$, $\tau\mapsto
\tau\cdot x_0$. It induces a morphism $\pi\colon G_\mathcal
K(\mathcal L) \to M_{\mathcal K}(\mathcal L)$. Let us consider $x =
\pi(\sigma)$.  This point $x$ is an $\mathcal L$ point of $M$ and
then it is a morphism
$$x\colon \Spec(\mathcal L) \to M_{\mathcal K}.$$
Let $\bar x\in M_{\mathcal K}$ be the image of $x$; then $x$ is
determined by the morphism $x^\sharp$ defined by the following
composition:
$$\xymatrix{\mathcal O_{M_{\mathcal K},\bar x} \ar[r]^-{\pi^\sharp}\ar[rrd]_-{x^\sharp} &
\mathcal O_{G_{\mathcal K},\mathfrak x}\ar[rd]^-{\sigma^\sharp} \\ &
& \mathcal L}$$

We are going to prove that $x$ is a rational point of $M_\mathcal
K$. Let us consider $\tau \in \Gal_{\mathfrak x}(G_{\mathcal
K},\partial_{\vec A})$. Therefore we have $R_{\tau}(x) = x$, and the
following diagram is commutative:
$$\xymatrix{\mathcal  O_{M_{\mathcal K},\bar x} \ar[rrrd]^-{x^\sharp}\ar[rd]\ar[rddd]_-{x^\sharp}\\
& \mathcal O_{G_{\mathcal K},\mathfrak x}\ar[rr]_-{\sigma^{\sharp}}\ar[dd]^-{(\sigma\tau)^\sharp} & & \mathcal L \ar[lldd]^-{R_\tau^\sharp} \\
& &\\
&  \mathcal L }$$ For each $f\in \mathcal O_{X_{\mathcal K},\bar
x}$, we have $x^\sharp(f) = R_\tau^\sharp(x^\sharp(f))$. This
equality holds for all $\tau\in H_{x_0}$. Hence,  $x^\sharp(f)$ an
element of $\mathcal L$ that is invariant for any differential
$\mathcal K$-algebra automorphism of $\mathcal L$. In virtue of the
Galois correspondence the fixed field of $\mathcal L$ by the action
of ${\Gal_{\mathfrak x}(G_{\mathcal K},\partial_{\vec A})}$ is
$\mathcal K$ %(Theorem \ref{ThGaloisCorrespondence})
. Thus,
$x^\sharp(f) \in \mathcal K$.
\end{proof}

\begin{theorem}\label{ThKolchinH}
Let us consider an algebraic subgroup $G'$ of $G$ verifying:
\begin{itemize}
\item[(1)] $\Gal_{\mathfrak x}(G_{\mathcal K}, \partial_{\vec A} ) \subset
G'$,
\item[(2)] $H^1(H,\mathcal K)$ is trivial.
\end{itemize}
Then there exist a gauge isomorphism $L_{\tau}$ of $G$ with
coefficients in $\mathcal K$ reducing the automorphic system $\vec
A$ to an automorphic system in $H$,
$$\vec B = \Adj_{\tau}(\vec A) + l\partial(\tau),$$
belongs to $\mathcal R(G')\otimes_{\mathcal C} \mathcal K$.
\end{theorem}

\begin{proof}
  By Proposition \ref{PrpRationalX} there exists a rational
solution of the Lie-Vessiot system in $M$ associated to $\vec A$.
Theorem \ref{C4THE4.1.5} says that such a reduction exists.
\end{proof}

Denote by $\Gal_{\mathfrak x}^0(G_{\mathcal K},\partial_{\vec A})$
the connected component of the identity of the Galois group
$\Gal_{\mathfrak x}(G_{\mathcal K},\partial_{\vec A})$.

\begin{corollary}
  Let ${\mathcal K^\circ}$ be the
relatively algebraic closure of $\mathcal K$ in $\mathcal L$. Assume
that 
$H^1(\Gal^0_{\mathfrak x}(G_{\mathcal K},\partial_{\vec
A}),{\mathcal K^\circ})$ is trivial. Then there is a gauge
transformation $L_\tau$, $\tau$ with coefficients in ${\mathcal
K^\circ}$ such that
$$\vec B = \Adj_{\tau}(\vec A) + l\partial(\tau)$$
belongs to $\mathcal R(\Gal^0_{\mathfrak x}(G_{\mathcal
K},\partial_{\vec A}))\otimes_{\mathcal C}{\mathcal K^\circ}$.
\end{corollary}

\begin{proof}
  We know that the Galois group of the automorphic system with coefficients in
${\mathcal K^\circ}$ is precisely $\Gal^0_{\mathfrak x}(G_{\mathcal
K},\partial_{\vec A})$ (see, for instance, remark (c)  in \cite{BM}, below Proposition 18). We apply then Theorem \ref{ThKolchinH}.
\end{proof}

\begin{corollary}\label{LmConnectedH}
 If $H^1(\Gal_{\mathfrak
x}(G_{\mathcal K},\partial_{\vec A}),\mathcal K)$ is trivial then
$\Gal_{\mathfrak x}(G_{\mathcal K},\partial_{\vec A})$ is connected.
\end{corollary}

\begin{proof}
  If $H^1(\Gal_{\mathfrak
x}(G_{\mathcal K},\partial_{\vec A}),\mathcal K)$ is trivial, then
we can reduce the automorphic system to an automorphic system in
$\mathcal R(\Gal_{\mathfrak x}(G_{\mathcal K},\partial_{\vec A}))
\otimes_{\mathcal C}\mathcal K$. Note that $\Gal^0_{\mathfrak
x}(G_{\mathcal K},\partial_{\vec A})$ and $\Gal_{\mathfrak
x}(G_{\mathcal K},\partial_{\vec A})$ have the same Lie algebra.
Therefore the Galois group of the reduced equation is contained in
$\Gal^0_{\mathfrak x}(G_{\mathcal K},\partial_{\vec A})$.
\end{proof}

The following is an extension of the classical result of Kolchin on
the reduction a system of linear differential equations to the Lie
algebra of its Galois group \cite{Ko1}

\index{theorem!Kolchin of reduction}
\begin{theorem}[Kolchin]\label{C4THE4.1.10}
  Let us consider the relative
algebraic closure ${\mathcal K^\circ}$ of $\mathcal K$ in $\mathcal
L$. There is a gauge transformation $L_\tau$, $\tau$ with
coefficients in ${\mathcal K^\circ}$, such that,
$$\vec B = \Adj_{\tau}(\vec A) + l\partial(\tau)$$
belongs to $\mathcal R(\Gal_{\mathfrak x}(G_{\mathcal
K},\partial_{\vec A}))\otimes_{\mathcal C}{\mathcal K^\circ}$.
\end{theorem}

\begin{proof}
Denote by $H$ the Galois group $\Gal_{\mathfrak x}(G_{\mathcal
K},\partial_{\vec A})$. Let us consider $M = G/H$, and let us denote
by $x_0\in M$ the origin which is the class of $H$ in $M$. Let $\vec
Y$ be the Lie-Vessiot vector field in $M$ associated to $\vec A$. In
virtue of Proposition \ref{PrpRationalX}, the canonical projection
$G(\mathcal L)\to M(\mathcal L)$ sends the fundamental solution
$\sigma$ to a solution $x$ of $(M,\partial_{\vec Y})$ with
coefficients in $\mathcal K$. Let us consider the projection:
$$\pi\colon G_{\mathcal K}\to M_{\mathcal K}.$$
Lemma \ref{LmPrincipalStalk} says that the stalk $\pi^{-1}(x)$ is a
principal homogeneous space modeled over the group $H_{\mathcal K}$.
Let us denote by $P\subset G_{\mathcal K}$ such homogeneous space.
Note that $P$ is $\overline{\{\mathfrak x\}}$, the closure of
$\mathfrak x$ in Zariski topology. We have the isomorphism,
$$\psi\colon P\times_{\mathcal K} H_{\mathcal K} \to P \times_{\mathcal K} P,
\quad (\tau, g) \to (\tau, \tau g),$$

Let $\tau$ be a closed point of $P$. Its rational field
$\kappa(\tau)$ is an algebraic extension of $\mathcal K$. We have
that $x = \tau\cdot x_0$. Thus, we can apply Lie-Kolchin reduction
method. $L_{\tau^{-1}}$ is a gauge transformation with coefficients
in $\kappa(\tau)$:
$$L_{\tau^-1}\colon G_{\kappa(\tau)}\to G_{\kappa(\tau)},$$
that sends the automorphic vector field $\vec A$ to an automorphic
vector field $\vec B$ in $H$ with coefficients in $\kappa(\tau)$.

In order to finish the proof we have to see that $\kappa(\tau)$ is a
subfield of the relative algebraic closure $\mathcal K^\circ$ of
$\mathcal K$ in $\mathcal L$. It is enough to see that $\mathcal K
\subset \kappa(\tau)$ is an intermediate differential extension of
$\mathcal K\subset\mathcal L$. Furthermore, if $\kappa(\tau)$ is an
intermediate differential extension then it coincides with $\mathcal
K^\circ$ because of the Galois correspondence.

Let us consider then the following base extension and natural
projection,
$$P_{\kappa(\tau)} = P \times_{\mathcal K} \Spec(\kappa(\tau)), \quad \quad
\pi_1\colon P_{\kappa(\tau)}\to P.$$ The product $P_{\kappa(\tau)}$
is a principal homogeneous space modeled over $H_{\kappa(\tau)}$.
Moreover, $\tau$ induces a rational point of $P_{\kappa(\tau)}$.
Hence, the Galois cohomology cohomology class of $P_{\kappa(\tau)}$
is trivial, so that it is isomorphic to $H_{\kappa(\tau)}$ as
homogeneous space. $P_{\kappa(\tau)}$ has as many connected
components as $H_{\kappa(\tau)}$. We write it as the disjoint union
of its connected components.
$$P_{\kappa(\tau)} = \bigsqcup_{i\in\Lambda} P_i.$$
For each $i\in \Lambda$, the restriction $P_i\to P$ is an
isomorphism of $\mathcal K$-schemes, and $\pi_1$ is a trivial
covering. But each $P_i$ is a $\kappa(\tau)$-scheme, and then each
component induces in $P$ an structure of $\kappa(\tau)$-scheme.
Hence we have a realization of $\kappa(\tau)$ as intermediate
extension
$$\mathcal K \subset \kappa(\tau) \subset  \mathcal L.$$
Thus, $\kappa(\tau) = {\mathcal K^\circ}$.
\end{proof}

%\begin{remark} Note that in the case of affine groups, due to Chevalley's theorem
%, the hypothesis of existence
%of geometric quotient in Theorem \ref{C4THE4.1.10} is unnecessary.
%\end{remark}

\subsection{Integrability by Quadratures}

  To integrate an automorphic system by quadratures
means to write down a fundamental solution by terms of a formula.
This formula should involve the solutions of certain simpler
equations. We assume that we have a geometrical meccano to express
these solutions. We refer to elements of such a meccano as
\emph{quadratures}. Those simpler equations are like the building
blocks of our integrability theory. Depending of which simpler
equations we consider as \emph{integrable} we obtain different
theories integrability. In theory of Lie-Vessiot systems the
elements of our formulas are the \emph{exponential maps of Lie
groups} and \emph{indefinite integrals}.

 From a geometric point of view, it is reasonable to consider
automorphic systems in \emph{abelian groups} as \emph{integrable}.
Let us consider an abelian Lie group $G$. Then, the exponential map,
$$\exp\colon \mathcal R(G) \to G,$$
is a group morphism, and moreover, $\mathcal R(G)$ is the universal
covering of $G$. An automorphic equation,
$$\frac{d\log}{dt}(x) = \sum_{i=1}^n f_i(t)\vec A_i, \quad \vec A_i\in \mathcal R(G)$$
is integrated by the formula,
$$\sigma(t) = \exp\left(\sum_{i=1}^n \left(\int_{t_0}^t f_i(\xi) d\xi\right) \vec A_i\right).$$
This formula involves the integral of $t$ dependent functions, and
the exponential map of the Lie group. Assuming that we are able of
realize these operations \emph{a reasonable point of view is to
consider al automorphic equations in abelian groups integrable}.
This assumption is done in \cite{Ve2}, and followed in
\cite{Bryant}. On the other hand, the algebraic case has a new kind
of richness. An abelian Lie group splits in direct product of
circles an lines, but an abelian algebraic group can carry a higher
complexity, for example in the case of abelian varieties. In such
case the exponential map is the solution of the Abel-Jacobi
inversion problem. In \cite{Ko0} Kolchin develops a theory of
integrability generalizing Liouville integrability, in which just
quadratures in one dimensional abelian groups are allowed. It
reduces the case to quadratures in the additive group, the
multiplicative group and elliptic curves.

\subsection{Quadratures in the Additive Group}

Let us consider an automorphic equation in the additive group
$\mathcal C$. The additive group is its own Lie algebra, and the
logarithmic derivative is the usual derivative. Thus, the
automorphic equations are written in the following form:
\begin{equation}\label{EqAutomorphicAdditive}
\partial x = a, \quad a \in \mathcal K.
\end{equation}

\begin{definition}
  An extension of differential fields $\mathcal K \subset \mathcal
  L$ is an integral extension if $\mathcal L$ is $\mathcal K(b)$, with
  $\partial b \in \mathcal K$. We say that $b$ is an integral element over $\mathcal K$.
\end{definition}

  It is obvious that the Galois extension of equation
\eqref{EqAutomorphicAdditive} is an integral extension of $\mathcal
K$, with $b = \int a$. The additive group (of a field of
characteristic zero) has no algebraic subgroups. Therefore, if $a$
is algebraic over $\mathcal K$, then $a\in\mathcal K$. Hence we have
two different possibilities for integral extensions:
  \begin{itemize}
  \item $b\in \mathcal K$, $\Gal(\mathcal L/\mathcal K) = \{e\}$,
  \item $b\not\in\mathcal K$, $\Gal(\mathcal L/\mathcal K) =
  \mathcal C$.
  \end{itemize}

\subsection{Quadratures in the Multiplicative Group}

  Let us consider now an automorphic equation in the multiplicative
group. For the complex numbers $\mathbb C^*$ the exponential map is
the usual exponential. In the general case of an algebraically
closed field of characteristic zero, we can build the exponential
map for $\mathcal C^*$. However, it does not take values in
$\mathcal C^*$ but in a bigger group. We avoid such a construction,
and then we consider the exponential just as an algebraic symbol.
The logarithmic derivative in $\mathcal C^*$ coincides with the
classical notion of logarithmic derivative,
$$\mathcal K^* \to \mathcal K, \quad x \mapsto \frac{\partial
x}{x}.$$ The general automorphic equation in the multiplicative
group is written as follows:
\begin{equation}
\frac{\partial x}{x} = a,\quad a \in\mathcal K.
\end{equation}

\begin{definition}
  An extension of differential fields $\mathcal K \subset \mathcal
  L$ is an exponential extension if $\mathcal L = \mathcal K(b)$, with
  $\frac{\partial b}{b} \in \mathcal K$. We say that $b$ is an exponential element
  over $\mathcal K$.
\end{definition}

  $\mathcal C^*$ has cyclic finite subgroups. Then, we can obtain
exponential extensions that are algebraic. There appears the
following casuistic:

\begin{itemize}
\item $\Gal(\mathcal L/\mathcal K)$ is the multiplicative group $\mathcal C^*$
if $b$ is transcendent over $\mathcal K$.
\item $\Gal(\mathcal L/\mathcal K)$ is a cyclic group $(\mathbb Z_n)^*$ if $b^n\in\mathcal K$
for certain $n$. It means that there is $c\in\mathcal K$ that
$\frac{\partial c}{n c} = a$. In such case, $b^n = c$.
\end{itemize}

  Reciprocally, any algebraic Galois extension of $\mathcal K$ with a cyclic Galois
group is an exponential extension. Here, it is a an essential point
that $\mathcal C$ is algebraically closed.

\subsection{Quadratures in Abelian Varieties}

  Abelian varieties provide us examples of non linearizable
automorphic systems. For the following discussion, let us assume
that the constant field of $\mathcal K$ is the field of complex
numbers $\mathbb C$. Let $G$ be a complex abelian variety of complex
dimension $g$. Let us consider a basis of holomorphic differentials
$\omega_1,\ldots,\omega_g$, and $A_1,\ldots,A_g$,$B_1,\ldots,B_g$ a
basis of the homology of $G$, we can assume that $\int_{A_i}\omega_j
= \delta_{ij}$. Define the Jacobi-Abel map,
$$G \xrightarrow{\sim} \mathbb C^g/\Lambda, \quad  p \mapsto
\left(\int_e^p\omega_1,\ldots \int_e^p \omega_g\right).$$ The
exponential map is given by the exponential universal covering of
the torus and the inversion of the Jacobi-Abel map.
$$\xymatrix{\mathbb C^{g} \ar[d]_-{\exp} \ar[dr] \\
G \ar[r]^-{j} & \mathbb C^g/\Lambda }$$

A projective immersion of $G$ in $\mathbb P(\mathbb C,d)$, for $d$
big enough, is given by terms of \index{theta functions} theta
functions, $z \mapsto \left(\theta_0(z)\colon \ldots
\colon\theta_d(z)\right)$. Hence there are some homogeneous
polynomial constrains $\{P(\theta_0,\ldots,\theta_d)=0\}$. The
quotient $\frac{\theta_i}{\theta_j}$ defines a meromorphic abelian
function in $G$ (see \cite{Mum} Chapter 1, Section 3, p. 30). Let us
consider affine coordinates in $G$, $x_i =
\frac{\theta_i}{\theta_0}$. We can project the vector fields of
$\mathcal R(\mathbb C^g)$ to $G$,
$$\frac{\partial}{\partial z_i} \mapsto \sum_j F_{ij}(x_1,\ldots,x_d)\frac{\partial}{\partial
x_j}, \quad F_{ij}(x_1,\ldots,x_d) = \frac{\frac{\partial \theta
j}{\partial z_i}\theta_0 - \frac{\partial \theta_0}{\partial
z_i}\theta_j}{\theta_0^2}$$ being $F_{ij}$ abelian functions, and
then rational functions in the $x_j$. The automorphic system in
$\mathbb C^g$
$$\sum_i a_i \frac{\partial}{\partial z_i},\quad a_i\in \mathcal K$$
is seen in $A$ as a non linear system an $A$,
\begin{equation}\label{EqAA}
\dot x_j = \sum_i a_iF_{ij}(x_1,\ldots,x_d), \quad \{P(1,x_1,\ldots
x_d)=0\}.
\end{equation}
If $b_1,\ldots, b_d$ are integral elements over $\mathcal K$ such
that $\partial b_i = a_i$, then the solution of the automorphic
system \eqref{EqAA} is:
$$x_j = \frac{\theta_j(b)}{\theta_0(b)}, \quad\quad \left(\theta_0\left(b\right)
\colon\ldots \colon \theta_d\left(b\right) \right).$$

\begin{definition}
  A strongly normal extension $\mathcal K \subset \mathcal L$ whose
Galois group is an abelian variety is called an abelian extension.
\end{definition}

   For an automorphic system in an abelian variety $A$ we have that the
Galois group is an algebraic subgroup of $A$. Then its identity
component is an abelian variety. The Galois extension is then,
$$\mathcal K \subset {\mathcal K^\circ} \subset \mathcal L,$$
being $\mathcal K^\circ\subset \mathcal L$ an abelian extension.

\begin{example}
  Let us consider an algebraically completely integrable hamiltonian
system in the sense of Adler, Van Moerbecke and Vanhaecke (see
\cite{AMV}) $\{H,H_2,\ldots, H_n\}$ in $\mathbb C^{2n}$. Assume that
$\{H_i(x,y) = h_i\}$ are the equations of the affine part of an
abelian variety $G$. The Hamilton equations,
\begin{equation}\label{EqHamilton}
\dot x_i = \frac{\partial H}{\partial y_i}, \quad \dot y_i =
-\frac{\partial H}{\partial x_i}, \quad H_i(x,y) = h_i
\end{equation}
 are an automorphic system $\vec H$
in $G$ with constant coefficients $\mathcal K = \mathbb C$. In the
generic case, $G$ is a non-resonant torus, and then it is densely
filled by a solution curve of the equations \eqref{EqHamilton}. We
conclude that $(G,\partial_{\vec H})$ has not proper differential
points: its differential spectrum consist only of the generic point.
In such case, the Galois extension of the system is $\mathbb C
\subset \mathcal M(G)$, the field of meromorphic functions in $G$.
\end{example}

\begin{example} \emph{Automorphic systems in elliptic curves}:
  Let us examine the case of an elliptic curve $\mathcal E$ over
$\mathcal C$. Assume that $\mathcal E$ is given as a projective
subvariety of $\mathbb P(2, \mathcal C)$ in Weierstrass normal form.
$$t_0t_2^2 = 4t_1^3 - g_2t_0^2t_1 - g_3t_0^3$$
We take affine coordinates $x = \frac{t_1}{t_0}$ and $y =
\frac{t_2}{t_0}$. The Lie algebra $\mathcal R(\mathcal E)$ is then
generated by the vector field,
$$\vec v = y\frac{\partial}{\partial x} + (12x^2 - g_2)\frac{\partial}{\partial
y}$$

Every automorphic vector field in $\mathcal E$ with coefficients in
$\mathcal K$ is written in the form $a\vec v$ with $a\in\mathcal K$.
A solution of the automorphic equation is a point of $\mathcal E$
with values in the Galois extension $\mathcal L$. Such solution have
homogeneous coordinates $(1\colon \xi \colon \eta)$ such that $\eta
= a^{-1}\partial \xi$, and $\xi$ is a solution of the single
differential equation,
\begin{equation}\label{EqAE}
(\partial \xi)^2 = a^2(4\xi^2 - g_1\xi - g_2).
\end{equation}

If we know a particular solution $b$ of \eqref{EqAE} then we can
write down the general solution $(1\colon \xi \colon \eta)$ of the
automorphic equation by means of the addition law in $\mathcal E$
(see \cite{Ko0} p. 804 eq. 9), depending of an arbitrary point
$(1\colon x_0\colon y_0)\in \mathcal E(\mathcal C)$:
$$Sol\eqref{EqAE} \times \mathcal E(\mathcal C) \to \mathcal E(\mathcal L),\quad
(b, (1:x_0:y_0)) \mapsto (1:\xi:\eta)$$
\begin{equation}\label{Addition1}
\xi(x_0,y_0) =  -b - x_0 -\frac{1}{4}\left(\frac{\partial b -
ay_0}{a(b - x_0)}\right)^2
\end{equation}
\begin{equation}\label{Addition2}
\eta(x_0,y_0) =  -\frac{\partial b+ ay_0}{2a}+\frac{6}{2}(b +
x_0)\frac{\partial b - ay_0}{a(b -
x_0)}-\frac{1}{4}\left(\frac{\partial b - ay_0}{a(b -
x_0)}\right)^3.
\end{equation}
\end{example}

\begin{definition}
  Let $\mathcal K\subset \mathcal L$ a differential field extension. We say
that $b\in \mathcal L$ is a Weierstrassian element if there exist
$a\in \mathcal K$, and $g_1, g_2\in \mathcal C$, with the polynomial
$4x^3 - g_1 x - g_2$ having simple roots and such that, $(\partial
b)^2 = a^2(4b^2 - g_1 x - g_2)$. The differential extension
$\mathcal K \subset K(b,\partial b)$ is called an elliptic
extension.
\end{definition}

  The Galois extension of the automorphic equation \eqref{EqAE}
is an elliptic extension of $\mathcal K$. It can be transcendent or
algebraic. If it is transcendent then its Galois group is the
elliptic curve $\mathcal E$, if it is algebraic then its Galois
group is a finite subgroup of $\mathcal E$.

\index{equation!Weierstrass}
\begin{remark} Let us examine the case of complex numbers: assume that the field of
constants of $\mathcal K$ is $\mathbb C$. The solution of
Weierstrass equation is the elliptic function $\wp$, and it gives
rise to the universal covering of $\mathcal E$,
$$\pi\colon \mathbb C \to \mathcal E, \quad z\mapsto (1\colon\wp(z)\colon
\wp'(z)).$$ The automorphic vector field $a\vec v$ in $\mathcal E$
is the projection of the automorphic vector field
$a\frac{\partial}{\partial z}$ in $\mathbb C$. The solution of the
equation in the additive group is given by an integral element $\int
a$. Then the a solution of the projected system in $\mathcal E$ is
$(1\colon\wp(\int a)\colon\wp'(\int a))$. Then $b = \wp(\int a)$ is
the Weierstrass element of the Galois extension. Formulas
\eqref{Addition1} and \eqref{Addition2} are the addition formulas
for the Weierstrass $\wp$ and $\wp'$ functions.
\end{remark}

\begin{example} We obtain the previous situation in the case of one
degree of freedom, algebraic complete integrable hamiltonian
systems. Let us consider the pendulum equation:
\begin{equation}
\left.\begin{matrix} \dot x &=& y \\ \dot y &=&
\sin(x)\end{matrix}\quad \quad \right\} \quad \frac{y^2}{2} -
\cos(x) = h
\end{equation}
It is written as a simple ordinary differential equation depending
of the energy parameter $h$,
$$\left(\frac{dx}{dt}\right)^2 = 2h + 2\cos(x),$$
by setting $z = e^{ix}$, we obtain the algebraic form of such
equation, which is an automorphic equation in an elliptic curve for
all values of $h$ except for $h = \pm 1$;
$$\left(\frac{dz}{dt}\right)^2 = - z^3 -2hz^2 - 1.$$
The Weierstrass normal form is attained by setting $u = \frac{-z}{4}
- \frac{1}{6}h$;
$$\left(\frac{du}{dt}\right)^2 = 4u^3 - \frac{h^2}{3}u -
\left(\frac{h^3}{27}+\frac{1}{16}\right).$$ Hence, the general
solution is written in terms of the $\wp$ functions of invariants
$g_2 = \frac{h^2}{3}$ and $g_3 = \frac{h^3}{27}+\frac{1}{16}$, for
$h\neq \pm 1$:
$$z(t) = -4\wp(t + t_0) - \frac{2}{3}h\quad ;\quad  x(t) = \log\left( -4\wp(t-t_0) - \frac{4h+3\pi i}{6}
\right).$$
\end{example}

\subsection{Liouville and Kolchin Integrability}

\index{liouvillian extension} \index{liouvillian extension!strict}
\index{Kolchin extension}
\begin{definition}\label{C4DEF4.2.8}
  Let $\mathcal K\subset \mathcal F$ a differential field extension.
Let us break it up into a tower of differential fields:
$$\mathcal K = \mathcal F_0
\subset \mathcal F_1 \subset \ldots \subset \mathcal F_d = \mathcal
L.$$ We say that $\mathcal K \subset \mathcal F$ is $\ldots$
\begin{enumerate}
\item[(1)] $\ldots$ a Liouvillian extension if the differential fields $\mathcal F_i$ can
be chosen in such way that $\mathcal F_i\subset \mathcal F_{i+1}$ is
an algebraic, exponential or integral extension.
\item[(2)] $\ldots$ a strict-Liouvillian extension if the differential fields $\mathcal F_i$ can
be chosen in such way that $\mathcal F_i\subset \mathcal F_{i+1}$ is
an exponential or integral extension.
\item[(3)] $\ldots$ a Kolchin extension the differential fields $\mathcal F_i$ can
be chosen in such way that $\mathcal L_i\subset \mathcal F_{i+1}$ is
algebraic, elliptic, exponential or integral extension.
\end{enumerate}
\end{definition}

  Liouvillian and strict-Liouvillian extensions are Picard-Vessiot
extensions. An elliptic curve can not be a subquotient of an affine
group. Hence, if $\mathcal K\subset \mathcal F$ is a Kolchin
extension and $\Gal(\mathcal F/\mathcal K)$ is an affine group, then
it is a Liouville extension. From this perspective, the following
classical result is almost self evident:

\index{theorem!Drach-Kolchin}
\begin{theorem}[Drach-Kolchin]\label{C4THE4.2.9}
  Let $\mathcal K$ be a field of meromorphic functions of the
complex plane $\mathbb C$. Assume that the Weierstrass's $\wp$
function is not algebraic over $\mathcal K$. Then $\wp$ is not the
solution of any linear differential equation with coefficients in
$\mathcal K$.
\end{theorem}

\begin{proof}
  Let us assume that this equation exist, and let $\mathcal K
  \subset \mathcal F$ na associated its Galois extension. Its Galois group
  $\Gal(\mathcal F/\mathcal K)$ is an affine group. We have an
  intermediate extension:
  $$\mathcal K \subset \mathcal K(\wp, \wp') \subset \mathcal F,$$
  This intermediate extension $\mathcal K \subset \mathcal K(\wp, \wp')$ is strongly normal
  and its Galois group is an elliptic curve. Thus, there is a normal subgroup $H\lhd
  \Gal(\mathcal F/\mathcal K)$ and an exact sequence,
  $$ 0 \to H \to \Gal(\mathcal F/\mathcal K)  \to \mathcal E \to 0$$
  but the quotient group of an affine group is an affine group, and then
  $\mathcal E$ is affine.
\end{proof}

  From the Galois correspondence and some elemental properties of algebraic
groups we also have immediately the characterization of Liouvillian
and Kolchin extensions in terms of their Galois groups.

\begin{proposition}\label{C4PRO4.2.10}
Let $\mathcal K \subset \mathcal L$ be  a strongly normal extension.
\begin{enumerate}
\item[(1)]   $\mathcal K \subset \mathcal L$ is a Kolchin extension if and
only if there
  is a sequence of normal subgroups in $\Gal(\mathcal L/\mathcal K)$,
  $$H_0 \lhd H_1 \lhd \ldots \lhd H_n = \Gal(\mathcal L/\mathcal K),$$
  such that $\dim_{\mathcal C} H_i/H_{i+1} \leq 1$.
\item[(2)] $\mathcal K \subset \mathcal L$ is a strict-Liouville extension if
  and only if $\Gal(\mathcal L/\mathcal K)$ is an affine solvable
  group.
\item[(3)] $\mathcal K \subset \mathcal L$ is a Liouvillian extension
if and only if the identity component $\Gal^0(\mathcal L/\mathcal
K)$ is a linear solvable group.
\end{enumerate}
\end{proposition}

\begin{proof}
  For (1) and (3) see \cite{Ko0}. Let us proof that linear solvable
  Galois group implies strict Liouville.  Let us consider a
  resolution of the Galois group $H_0\lhd \ldots H_n$ such that each quotient
  $H_{i+1}/H_i$ is a cyclic group, a multiplicative group or an
  additive group. This resolution exist by means of Lie-Kolchin
  theorem. This resolution split the extension $\mathcal K\subset \mathcal L$ in a tower of
  differential fields
  $$\mathcal K_n \subset \mathcal K_{n-1} \subset \ldots \mathcal
  \subset \mathcal K_0.,$$
  Each differential extension of the tower is an exponential, integral or algebraic
  extension with cyclic Galois group. But an algebraic extension
  with cyclic group is a radical extension. The field $\mathcal C$ is
  algebraically closed, hence such radical extension is generated by the
  radical $\sqrt[n]{a}$ of a non-constant element of $a$, and then
  it is the Picard-Vessiot extension of the equation,
  $$\partial x = \frac{\partial a}{n a}x,$$
  which is an exponential extension.
\end{proof}

\subsection{Integration by Quadratures in Solvable Groups}

Let us remind that along this chapter we are considering an
automorphic vector field $\vec A$ with coefficients in $\mathcal K$
in an algebraic group $G$ defined over $\mathcal C$. We also
consider a Kolchin closed differential point $\mathfrak x\in
\Diff(G_{\mathcal K},\partial_{\vec A})$ and the associated Galois
extension $\mathcal K \subset \mathcal L$. We are going to explain
the classical integration by quadratures in terms of Lie-Kolchin
reduction method and Galois correspondence.

  Let us consider a normal subgroup $H\lhd G$, and the quotient
group $\bar G = G/H$. Let $\mathfrak y$ be the projection in $\bar
G_{\mathcal K}$ of $\mathfrak x$. In virtue of Theorem \ref{ThGM} we
know that,
  $$\mathcal K \subset \mathcal \kappa(\mathfrak y) \subset \mathcal
  L,$$
is an intermediate strongly normal extension. Furthermore, the
Galois group in $\mathfrak y$ of the automorphic system with
coefficients if $\kappa(\mathfrak y)$ is the intersection of the
Galois group $\Gal_{\mathfrak x}(G_{\mathcal K},\partial_{\vec A})$
with $H$.

\begin{theorem}\label{C4THE4.2.11}
  Assume that there is a resolution of $G$,
  $$H_0 \lhd H_1 \lhd \ldots \lhd H_n = G,$$
  such that $\dim_{\mathcal C} H_i/H_{i+1} = 1$,
  then $\mathcal K \subset \mathcal L$ is a Kolchin extension.
\end{theorem}

\begin{proof}
  Let us consider the quotients $\bar G_i = H_{n-i+1}/H_{n-i}$. They are algebraic groups
of dimension one. Each $G_i$ is isomorphic to one of the following:
the additive group, the multiplicative group, or an elliptic curve.
Each one corresponds to an integral, exponential, or Weierstrassian
quadrature. We prove the theorem by induction in the length of the
resolution. Let us consider the projection $\pi\colon G \to
G/H_{n-1}$. Define $\mathfrak y = \pi(\mathfrak x)$ and let
$\mathcal K_1$ be the relative algebraic closure of
$\kappa(\mathfrak x)$ in $\mathcal L$. Then $\mathcal K\subset
\kappa(\mathfrak y)$ is an integral, exponential or elliptic
extension and $\kappa(\mathfrak y)\subset \mathcal K_1$ is an
algebraic extension. Hence, $\mathcal K \subset \mathcal K_1$ is a
Kolchin extension.

Let $\mathfrak z$ be a closed differential point of $(G_{\mathcal
K_1},\partial_{\vec A})$ in the fiber of $\mathfrak x$. By Theorem
\ref{ThGM} $\Gal_{\mathfrak z}(G_{\mathcal K_1},\partial_{\vec A})
\subset H_{n-1}$, and then by Theorem \ref{C4THE4.1.10} there is a
gauge transformation $L_{\tau}$ with coefficients in $\mathcal K_1$
reducing the automorphic field to an automorphic field in $H_{n-1}$.
Any Galois extension associated to this last equation is $\mathcal
K_1$-isomorphic to $\mathcal L$. By the induction hypothesis the
extension $\mathcal K_1 \subset \mathcal L$ is a Kolchin extension,
hence $\mathcal K \subset \mathcal L$ is a Kolchin extension.
\end{proof}

\begin{theorem}\label{C4THE4.2.12}
  Assume that $G$ is affine and solvable. Then
  $\mathcal K \subset \mathcal L$ is a strict-Liouville
  extension.
\end{theorem}

\begin{proof}
The Galois group is a subgroup of $G$, and then it is a solvable
group. The result comes from Proposition \ref{C4PRO4.2.10} (2)
together with Theorem \ref{C4THE4.2.11}.
\end{proof}

\begin{proposition}\label{C4PRO4.2.13}
  If there is a connected affine solvable group
 $H\subset G$ such that $\Gal_{\mathfrak x}(G_{\mathcal K},\partial_{\vec A})\subset H$,
then $\mathcal K \subset \mathcal L$ is a strict-Liouville
extension.
\end{proposition}

\begin{proof}
  $H$ is connected affine solvable an then it has trivial Galois
  cohomology. We can reduce to the group $H$ by means of theorem
  \ref{ThKolchinH}. Hence, we are in the hypothesis of theorem
  \ref{C4THE4.2.12}.
\end{proof}

\subsection{Linearization}\label{C4SS4.2.3}

  There exist non-linear non-linearizable algebraic groups. 
An algebraic group that does not admit any linear representation
is called quasi-abelian. In other words, a quasi-abelian variety is an algebraic group
$G$ such that $\mathcal O_G(G) = \mathcal C$. Algebraic groups over an algebraic closed base field
$C$, which are complete and connected, are called abelian varieties. Since they are
complete varieties, they do not admit non-constant global regular functions and then
they are quasi-abelian.

  The following results give us the structure of the algebraic groups
by terms of linear and quasi-abelian algebraic groups. See, for instance
\cite{Sa}.

\begin{theorem}[Rayleigh decomposition]
  Let $G$ be an algebraic group. There is a unique subgroup $X\in G$ such
  that, $X$ is quasi-abelian and $G/X$ is an affine group.
\end{theorem}

\begin{theorem}[Chevalley-Barsotti-Sancho]\label{ThChevalleyBS}
  Let $G$ be a connected algebraic group over $\mathcal C$, with $\mathcal C$ an algebraically
  closed field of characteristic zero. Then there is a unique normal affine
  subgroup $N\subset G$ such that the quotient $G/N$ is an abelian
  variety.
\end{theorem}

\subsection{Reduction by means
Chevalley-Barsotti-Sancho Theorem}

  In virtue of Chevalley-Barsotti-Sancho theorem
(\ref{ThChevalleyBS} in appendix B), there is a unique linear normal
connected algebraic group $N \lhd G$ such that the quotient $G/N$
and is an abelian variety $V$. Let us consider the projection
$\pi\colon G \to V$. Let $\vec B$ be the projected automorphic
system  $\pi(\vec A)$ in $V$, and denote by $\mathfrak y$ the image
of $\mathfrak x$ by $\pi$. We state the following:

\begin{theorem}\label{C4THE4.2.14}
  Let $\mathcal M$ be the field of meromorphic functions in
$V_{\mathcal K}$. Assume that $\Gal_{\mathfrak y}(V_{\mathcal K},
\partial_{\vec B}) = V$, and one of the following hypothesis:
\begin{enumerate}
\item[(1)] $H^1(N,\mathcal M)$ is trivial.
\item[(2)] $\mathcal K$ is relatively algebraically closed in $\mathcal
L$.
\end{enumerate}
Then, there is a gauge transformation of $G$ with coefficients in
$\mathcal M$ reducing the automorphic system $\vec A$ to $N$.
\end{theorem}

\begin{proof}
  Let us consider $\vec A$ as an automorphic vector field in $G$
with coefficients in $\mathcal M$. By Galois correspondence we have:
  $$\Gal(\mathcal L/\mathcal M) \simeq \Gal_{\mathfrak x}(G_{\mathcal K},\partial_{\vec A})\cap N.$$
If hypothesis (1) holds, then the statement is a particular case of
Theorem \ref{ThKolchinH}. Let us prove the result in the case of
hypothesis (2). By Theorem \ref{C4THE4.1.10} there exists a gauge
transformations whose coefficients are algebraic over $\mathcal M$.
By hypothesis $\Gal_{\mathfrak x}(G_{\mathcal K},\partial_{\vec A})$
is connected. This group $\Gal_{\mathfrak x}(G_{\mathcal
K},\partial_{\vec A})$ realizes itself as a principal bundle over
$V$ whose structural group os $\Gal(\mathcal L/\mathcal M)$. It
implies that $\Gal(\mathcal L/\mathcal M)$ is also connected. So
that $\mathcal M$ is relatively algebraically closed in $\mathcal
L$. The coefficients of the considered gauge transformation are in
$\mathcal M$, as we wanted to prove.
\end{proof}

\subsection{Linearization by means of Adjoint
Representation}

We consider $GL(\mathcal R(G))$ the group of $\mathcal C$-linear
automorphisms of the Lie algebra $\mathcal R$. It is an algebraic
group over $\mathcal C$. The adjoint representation
$$\Adj\colon G\to GL(\mathcal R(G))$$
is a morphism of algebraic groups. It gives us a linearization of
the equations. Let us consider the center $\mathfrak Z(G)$ and the
exact sequence:
$$0 \to \mathfrak Z(G) \to G \to GL(\mathcal R(G)) \to 0$$
Denote by $\vec B$ the projection of the automorphic vector field
$\vec A$ by the morphism $\Adj$. It is a linear system and then its
Galois extension $\mathcal K \subset \mathcal P$ is a Picard-Vessiot
intermediate extension of $\mathcal K\subset \mathcal L$.

\begin{proposition}\label{C4PRO4.2.15}
 $\mathcal P \subset \mathcal L$ is a strongly normal extension and $\Gal(\mathcal L/\mathcal P)$ is
an abelian group.
\end{proposition}

\begin{proof}
  The extension $\mathcal P\subset \mathcal L$ is a Galois extension of $\vec A$ with coefficients in
$\mathcal P$, so that it is strongly normal. Its Galois group is, by
the Galois correspondence, the intersection of the Galois group of
$\Gal_{\mathfrak x}(G_{\mathcal K},\partial_{\vec A})$ with the
center $\mathfrak Z(G)$; it is an abelian group.
\end{proof}

\subsection{Linearization by means of Global Regular Functions}

 The ring of global regular functions $\Gamma(\mathcal O_{G},G)$ is a Hopf
 algebra, and then it spectrum is a linear algebraic group $L =
 \Spec(\Gamma(\mathcal O_G,G))$. The kernel $C$ of the canonical
 morphism $\pi\colon G\to L$ is, by definition a quasi-abelian variety (see \cite{Sa}). Let
 us consider the exact sequence:
$$ 0 \to C \to G \to L \to 0.$$
  We proceed as we did in Proposition \ref{C4PRO4.2.15}, and then we obtain
the following result.

\begin{proposition}\label{C4PRO4.2.16}
  Let $\mathcal K\subset \mathcal P$ be the Picard-Vessiot extension of
the automorphic system $\pi(\vec A)$ in $L$. Then $\mathcal P
\subset \mathcal L$ is a strongly normal extension, and the
connected component of the identity of its Galois group is a
quasi-abelian variety.
\end{proposition}

\section{Integrability of Linear Equations}

This section is devoted to the Liouville integrability of linear
differential equations. Since the development of Picard-Vessiot
system it is a rich field of research, let us cite some important
specialized literature \cite{Kov0}, \cite{SingerUlmer1},
\cite{SingerUlmer2}, \cite{UlmerWeil}, \cite{HoeijWeil},
\cite{HRUW}. Here, we adopt a slightly different point of view on
linear differential equations. We see them as automorphic systems.
It gives us some insight into the geometric mechanisms that allows
quadratures. In this way we are able to measure the solvability of
the Galois groups, in terms of equations in flag varieties and
grassmanians (Theorem \ref{C4THE4.3.2}). They are the natural
geometrical generalization of Riccati equations.

From now on let $G$ be a \emph{linear} connected algebraic group
over $\mathcal C$. We consider $\vec A$ an automorphic vector field
in $G$ with coefficients in $\mathcal K$.

\index{flag variety}
\subsection{Flag Variety}

 We call \emph{Borel subgroup} \index{Borel subgroup} of $G$ to any maximal connected
solvable group of $G$. Borel subgroups are all conjugated and
isomorphic subgroups. The quotient space $G/B$ is a complete variety
(see \cite{Sa} p. 163, th. 10.2).

\begin{definition}
  We call flag variety of $G$ to the homogeneous space quotient $G/B$, being
  $B$ a Borel subgroup of $G$.
\end{definition}

  The flag variety of $G$ is defined up to isomorphism of $G$-homogeneous spaces.
Let us consider $Flag(G)$ a flag variety of $G$, and let
$(Flag(G),\partial_{\vec F})$ be the induced Lie-Vessiot system.

  Let us see a natural generalization of the well-known
theorem of J. Liouville that relates the integrability by
Liouvillian functions of the second order linear homogeneous
differential equation with the existence of an algebraic solution of
an associated Riccati equation. This classical result is the
particular case of $GL(2,\mathbb C)$ in the following \emph{general
Liouville's theorem}.

\begin{theorem}\label{C4THE4.3.2}
  The Galois extension $\mathcal K \subset \mathcal L$ is
Liouvillian if and only if the flag Lie-Vessiot system
$(Flag(G),\partial_{\vec F})$ has an algebraic solution with
coefficients in ${\mathcal K^\circ}$, the algebraic relative closure
of $\mathcal K$ in $\mathcal L$.
\end{theorem}

\begin{proof}
  By the Galois correspondence we have that the Galois group of $(G_{{\mathcal K^\circ}},
  \partial_{\vec A})$ is the connected identity component of the
  Galois group of $(G_{\mathcal K},\partial_{\vec A})$. Assume that
  $(Flag(G),\partial_{\vec F})$ has an algebraic
  solution $x \in Flag(G)({\mathcal K^\circ})$. We are under the hypothesis of
  Theorem \ref{C4THE4.1.10}. There is a gauge transformation of $G_{\mathcal K^0}$ that
  send $\vec A$ to an automorphic vector field $\vec B$ in the Borel subgroup $B$. Then
  the Galois group of $\vec B$ with coefficients in $\mathcal K^0$ is contained in a Borel subgroup.
  Then the connected component of $\Gal_{\mathfrak x}(G_{\mathcal K},\partial_{\vec A})$ is solvable.

  Reciprocally, let us assume that $\mathcal K \subset \mathcal L$ is a Liouvillian extension.
  In such case the identity connected component of the Galois group is contained in a Borel subgroup $B$.
  By Proposition \ref{PrpRationalX} there is a solution with coefficients in $\mathcal K^\circ$ of
  $\vec F$.
\end{proof}

\subsection{Automorphic Equations in the General Linear Group}

\subsection{Grassmanians}

  Let us consider $E$ as $n$-dimensional vector space. Along this text
\emph{$m$-plane} will mean \emph{$m$-dimensional linear subspace}.
For all $m\leq n$ the linear group $GL(E)$ acts transitively in the
set of $m$-planes. For an $m$-plane $E_m$, the stabilizer subgroup
is an algebraic group, and then the set of $m$-planes define an
algebraic homogeneous space.

\index{grassmanian}
\begin{definition}
  We call grassmanian of $m$-planes of $E$, $\Gr(E,m)$,
  to the homogeneous space whose closed points are the $m$-planes of
  $E$. Denote $\Gr(\mathcal C,n,m)$ the grassmanian of
  $m$-planes of $\mathcal C^n$.
\end{definition}

\begin{example} $\Gr(\mathcal C,n,1)$ is the space of lines in
$\mathcal C^n$, and then if its the projective space of dimension
$n-1$, $\mathbb P(n-1,\mathcal C)$. The $\Gr(\mathcal C,n,n-1)$ is
the space os hyperplanes and then it is the dual projective space
$\mathbb P(n-1,\mathcal C)^*$.
\end{example}

In general, $m$-planes of $E$ are in one-to-one correspondence with
$(n-m)$-planes of the dual space $E^*$, and then we have the
projective duality
$$\Gr(E,m) \simeq \Gr(E^*, n-m).$$
The action of $GL(E)$ on $\Gr(E,m)$ is not faithful. Each scalar
matrix of the center of $GL(\mathcal C, n)$ fix all $m$-planes.
Thus, the non faithful action of $GL(E)$ is reduced to a faithful
action of the projective group $PGL(E)$.

  All grassmanian are projective varieties. There is a canonical
embedding of $\Gr(E,m)$ into the projective space of dimension
$\left(\substack{n\\ m}\right)-1$, called \emph{the pl\"ucker
embedding}:
$$\Gr(E,m) \to \mathbb P(E^{\wedge n}), \quad\langle
e_1,\ldots,e_m\rangle \mapsto \langle e_1\wedge
e_1\wedge\ldots\wedge e_m\rangle.$$
\index{Pl\"ucker!embedding}\index{Pl\"ucker!coordinates}

  For computation in the grassmanian spaces we will use
\emph{pl\"uckerian coordinates}. This system of coordinates is
subordinated to a basis in $E$. Thus, let us consider a basis
$\{e_1,\ldots,e_n\}$. Let $E_1 = \langle e_1,\ldots, e_m\rangle$ be
the $m$ plane spanned by the first $m$ elements of the basis, and
define $E_2 = \langle  e_{m+1}, \ldots e_{n}\rangle$ its
complementary. Let us consider the projection $\pi \colon E \to E_2$
of kernel $E_1$. We define the open subset $U \subset \Gr(E,m)$,
$$U = \{F\colon F\oplus E_2 = E\}.$$
For $F\in U$ the splitting of the space induces an isomorphism $i_F
\colon E_1 \to F$. We have an isomorphism
$$U \xrightarrow{\sim} \Hom_{\mathcal C}(E_1,E_2),\quad F \mapsto \pi\circ i_F.$$
We define the pl\"ukerian coordinates of $F$ as the matrix elements
of $\pi\circ i_F$ in the above mentioned basis. By permuting the
elements of the basis we construct a covering of $\Gr(E,m)$ by
$\left(\substack{n \\ m}\right)$ affine open subsets isomorphic to
$\mathcal C^{n(n-m)}$.

  Let us compute pl\"uckerian coordinates in $\Gr(\mathcal C, m, n)$
related to the canonical basis. Let us consider $F\in \Gr(\mathcal
C, m, n)$, and a basis of $F$, $\{\vec x_1,\ldots, \vec x_m\}$,
$\vec x_i = (x_{1i},\ldots,x_{ni})$. The matrix,
$$\left(\begin{matrix}x_{11}  & \ldots & x_{1m} \\
x_{21} & \ldots & x_{2m} \\ \vdots & \ddots & \vdots \\ x_{n1} &
\ldots & x_{nm}
\end{matrix}\right)$$
is of maximal rank. Thus, there is a non vanishing minor of rank
$m$. In particular, $F$ is in the open subset $U$ if and olny if the
minor corresponding to the first $m$ rows does not vanish. In such
case we define the numbers $\lambda_{ij}^{(m)}$
$$\left(\begin{matrix}x_{11}  & \ldots & x_{1m} \\
x_{21} & \ldots & x_{2m} \\ \vdots & \ddots & \vdots \\ x_{n1} &
\ldots & x_{nm}
\end{matrix}\right)\left(\begin{matrix}x_{11}  & \ldots & x_{1m} \\
\vdots & \ddots & \vdots \\ x_{m1} & \ldots & x_{mm}
\end{matrix} \right)^{-1} =
\left(\begin{matrix} 1  & \ldots & 0 \\ \vdots & \ddots & \vdots
\\ 0 & \ldots & 1 \\ \lambda^{(m)}_{11} & \ldots & \lambda^{(m)}_{1m} \\ \vdots
&  \ddots & \vdots \\ \lambda^{(m)}_{n-m,1} & \ldots &
\lambda^{(m)}_{n-m,m}
\end{matrix}\right)$$
that are the pl\"uckerian coordinates of $E_m\in \Gr(\mathcal
C,m,n)$ in the open affine subset $U$ related to the split of
$\mathcal C^n$ as $E_1\otimes E_2$.

\subsection{Flag Variety of the General Linear Group}

  A  flag of subspaces of $\mathcal C^n$, is a sequence,
  $$E_1 \subset E_2 \subset \ldots  \subset E_{n-1}, \quad \dim_{\mathcal C}E_i = i$$
of linear subspaces of $\mathcal C^n$. The space $Flag(\mathcal
C,n)$ of flags of $\mathcal C^n$ is an homogeneous space of
$GL(\mathcal C,n)$, and it is faithful for the action of
$PGL(\mathcal C, n)$. There is a canonical morphism,
$$Flag(\mathcal C, n) \to \prod_{m=1}^{n-1} \Gr(\mathcal C,n,m),\quad E_1 \subset E_2 \subset E_{n-1} \mapsto (E_1,\ldots,
E_{n-1}).$$

  By Lie-Kolchin theorem the isotropy subgroup of a flag is also a
Borel subgroup. Then, we can state $Flag(\mathcal C,n)$ is the flag
variety of the general linear group. Let us introduce a system of
coordinates in $Flag(\mathcal C,n)$. Let us consider $\{e_1,\ldots,
e_n\}$ the canonical basis of $\mathcal C^n$. Each $\sigma\in
GL(\mathcal C,n)$ defines a flag $F(\sigma)$ as follows:
$$\langle\sigma(e_1)\rangle \subset
\langle\sigma(e_1),\sigma(e_2)\rangle \subset \ldots \subset
\langle\sigma(e_1),\ldots,\sigma(e_{n-1})\rangle.$$

  There is a canonical flag corresponding to the identity element.
Its isotropy group is precisely $T(\mathcal C,n)$ the group of upper
triangular matrices. Then two matrices $A, B \in GL(\mathcal C,n)$
define the same flag if and only if $A = BU$ for certain $U\in
T(\mathcal C,n)$. Then let us consider the affine subset of
$GL(\mathcal C,n)$ of matrices with non vanishing principal minors.
For such a matrix there exist a unique $LU$ decomposition such that
$U\in T(\mathcal C,n)$ and is a lower triangular matrix as follows,
$$A = \left(\begin{matrix} 1 & 0 & \ldots & 0 \\ \lambda_{21} & 1 & \ldots & 0 \\
\vdots & \vdots & \ddots & \vdots \\ \lambda_{n1} & \lambda_{n2} &
\ldots & 1 \end{matrix}\right)U$$

Hence the matrix elements $\lambda_i$ define a system of affine
coordinates in $Flag(\mathcal C,n)$, in certain affine open subset.
We construct an open covering of the flag space by permutating the
vectors of the canonical base. The canonical morphism
$$Flag(\mathcal C, n) \to \prod_m \Gr(\mathcal C,m,n)$$
is easily written in pl\"uckerian coordinates:
$$\lambda_{ij}^{(m)} = \lambda_{i+m,j} -
\sum_{k=1}^{m}\lambda_{i+m,k}\lambda_{kj}.$$

\subsection{Matrix Riccati Equations}

Let us consider an homogeneous linear differential equation
$$\dot x = Ax, \quad A\in gl(\mathcal K, n).$$
It is seen as an automorphic system that induces Lie-Vessiot systems
in each homogeneous space. Let us compute the induced Lie-Vessiot
systems in the grassmanian spaces. First, the linear system induces
a linear system in $(\mathcal C^n)^m$.
\begin{equation}\label{EqMatL}
\dot X = AX,
\end{equation}
where $X$ is a $n\times m$ matrix. We write $X = \left(\substack{ U
\\ Y }\right)$, being $U$ a $m\times m$ matrix and $Y$ a $(n-m)\times m$
matrix. $\Lambda_m = YU^{-1}$ is the matrix of pl\"uckerian
coordinates of the space generated by the $m$ column vectors of the
matrix $X$. Then, $\dot \Lambda_m = \dot Y U^{-1} - \Lambda_m \dot U
U^{-1}$. If we decompose the matrix $A$ in four submatrices
$$A = \left(\begin{matrix} A_{11} & A_{12} \\ A_{21} & A_{22}
\end{matrix}\right)$$
being $A_{11}$ of type $m\times m$, $A_{12}$ of type $m\times
(n-m)$, $A_{21}$ of type $(n-m)\times m$, and $A_{22}$ if type
$m\times m$. Them the matrix linear equation \eqref{EqMatL} splits
as a system of matrix linear differential equations,
$$\dot U = A_{11} + A_{12}Y, \quad \dot Y = A_{21}U + A_{22}Y,$$
from which we obtain the differential equation for affine
coordinates in the grassmanian,
\begin{equation}\label{EqRiccMat}
\dot\Lambda_m = A_{21} + A_{22}\Lambda_m - \Lambda_m A_{11} -
\Lambda_m A_{12} \Lambda_m
\end{equation}
which is a quadratic system. We call such a system a \emph{matrix
Riccati equation} \index{equation!matrix Riccati} associated to the
linear system.
$$\Lambda_m = \left(\begin{matrix}\lambda^{(m)}_{11} &\ldots & \lambda^{(m)}_{1,m} \\ \vdots & \ddots & \vdots \\
\lambda^{(m)}_{n-m,1} & \ldots & \lambda^{(m)}_{n-m,m}
\end{matrix}\right)$$
$$\dot \lambda^{(m)}_{ij} = a_{m+i,j} +
\sum_{k=1}^{n-m}a_{m+i,m+k}\lambda^{(m)}_{kj} - \sum_{k=1}^m
\lambda^{(m)}_{ik}a_{kj} - \sum_{\substack{k = 1 \ldots m \\
r = 1 \ldots n-m}}\lambda^{(m)}_{ik}a_{k,r+m}\lambda^{(m)}_{rj}$$

\begin{example} Let us compute the matrix Riccati
equations associated to the general linear system of rank $2$ and
$3$. First, let us consider a general linear system of rank $2$,
  $$\dot x_1 = a_{11} x_1 + a_{12} x_2, \quad \dot x_2 = a_{21} x_1
  + a_{22} x_2.$$
There is one only grassmanian $\Gr(\mathcal C,1,2)$, which is
precisely the projective line. The associated matrix Riccati
equation is an ordinary Riccati equation
  $$\dot x = a_{21} + (a_{22} - a_{11})x - a_{12}x^2.$$

In the case of a general system of rank $3$,
  $$\left(\begin{matrix} \dot x_1 \\ \dot x_2 \\ \dot x_3\end{matrix} \right) =
  \left(\begin{matrix} a_{11} & a_{12} & a_{13} \\ a_{21} & a_{22} & a_{23} \\ a_{31} & a_{32} & a_{33}
  \end{matrix} \right) \left(\begin{matrix} x_1 \\ x_2 \\ x_3 \end{matrix} \right)$$
\end{example}
  there are two grassmanian spaces, $\Gr(\mathcal C, 1, 3)$ and
  $\Gr(\mathcal C, 2, 3)$, being the projective plane $\mathbb
  P^2(\mathcal C)$ and the projective dual plane $\mathbb
  P^2(\mathcal C)^*$ respectively. Then we obtain two quadratic
  systems,
  $$\mathbb P(2,\mathcal C) \quad \left\{\begin{matrix}
  \dot x &=& a_{21} + (a_{22}-a_{11})x + a_{23}y - a_{12}x^2 -
  a_{13}xy
  \\
  \dot y &=& a_{31} + (a_{33}-a_{11}) + a_{32}x - a_{13}y^2 -
  a_{12}xy
  \end{matrix}\right.$$

  $$\mathbb P(2,\mathcal C)^*\left\{\begin{matrix}
  \dot \xi &=& a_{31} + (a_{33}-a_{11})\xi + a_{21}\eta - a_{23}\xi\eta -
  a_{13}\xi^2
  \\
  \dot \eta &=& a_{32} + (a_{33}-a_{22})\eta + a_{12}\xi - a_{13}\xi\eta -
  a_{23}\eta^2
  \end{matrix}\right.$$
called the associated \emph{projective Riccati equations}.
\index{equation!projective Riccati}

\subsection{Flag Equation}

\index{equation!flag} From the relation between pl\"uckerian
coordinates and affine coordinates in the flag variety we can deduce
the equations of the induced Lie-Vessiot system in $Flag(\mathcal
C,n)$, from the matrix Riccati equations. We will obtain a Riccati
quadratic equation for $n = 2$, and a cubic system for $n\geq 3$.
$$\dot \lambda_{ij} = a_{ij} + \sum_{k=j+1}^{n}a_{ik}\lambda_{kj} - \sum_{k=1}^{j}\lambda_{ik}a_{kj} +
\sum_{k=1}^{j}\sum_{r=k+1}^{j}\lambda_{ir}\lambda_{rk}a_{kj}
$$ $$- \sum_{k=1}^j\sum_{r=j+1}^{n}\lambda_{ik}a_{kr}\lambda_{rj}+
\sum_{k=1}^{j}\sum_{r=j+1}^{n}\sum_{s=k+1}^{j}\lambda_{is}\lambda_{sk}a_{kr}\lambda_{rj},$$
Setting $\lambda_{ii} = 1$ for all $i$, we can simplify these
equations.
\begin{equation}\label{EqFlagEq}
\dot{\lambda_{ij}} = \sum_{k=j}^{n}a_{ik}\lambda_{kj} -
\sum_{k=1}^j\sum_{r=j}^n \lambda_{ik}a_{kr}\lambda_{rj} +
\sum_{k=1}^j\sum_{r=k+1}^j\sum_{s=j}^n
\lambda_{ir}\lambda_{rk}a_{ks}\lambda_{sj}
\end{equation}
Such as cubic system can be seen as a hierarchy of projective
Riccati equations. The equation corresponding to the first column
$\lambda_{i1}$, $i=2\ldots,n$ is a projective Riccati equation in
$\mathbb P(n-1,\mathcal C)$. The equation corresponding to the
second column is a projective Riccati equation in $\mathbb
P(n-2,\mathcal C(\lambda_{i1}))$, and so on.

\begin{example}
  Let us compute the flag equation for the general differential linear
system of rank $3$. Denote $x = \lambda_{21}$, $y=\lambda_{31}$, $z
=\lambda_{32}$.
\begin{equation}\label{C4EQ4.7}
\left\{\begin{array}{ccl} \dot x &=& a_{21} + (a_{22}- a_{11})x +
a_{23}y  - a_{12}x^2 - a_{13}xy \\
\dot y &=& a_{31} + a_{32}x + (a_{33} - a_{11})y - a_{12}xy -
a_{13}y^2\end{array}\right.
\end{equation}
$$\dot z = a_{32} - a_{12}y + (a_{33} -a_{22} +a_{12}y - a_{13}y )z + (a_{13}y - a_{23})z^2.$$
\end{example}

\subsection{Equations in the Special Orthogonal Group}

Automorphic equations in special orthogonal group have been deeply
studied since 19th century \cite{Ve3}, \cite{Darboux}. In particular
Darboux related these equation with Riccati equation. He stated that
the integration of \eqref{EqSO3} is reduced to the integration of
\eqref{RicSO3}. Here we show that the Flag equation of an
automorphic equation in $SO(\mathcal C, 3)$ is precisely the Riccati
equation, and then the solutions of \eqref{EqSO3} are Liouvillian if
and only if there are algebraic solutions for \eqref{RicSO3}

  The Lie algebra $so(3,\mathcal C)$ is the algebra of skew-symmetric
matrices of $gl(\mathcal C, 3)$. Then an automorphic system in
$SO(3, \mathcal C)$ is written in the following form.
\begin{equation}\label{EqSO3}
   \left(\begin{matrix} \dot x_0\\ \dot x_1 \\ \dot
   x_2\end{matrix}\right) = \left(\begin{matrix} & a & b \\ -a & & c \\ -b & -c \end{matrix} \right)
   \left(\begin{matrix} x_0 \\ x_1 \\ x_2 \end{matrix}\right)\quad
   \quad a,b,c\in\mathcal K,
\end{equation}
\emph{where the void spaces represent the vanishing elements in the
matrix}.

\subsection{On the Structure of the Special Orthogonal Group}

 The special orthogonal group is the group of linear transformations
preserving the quadratic form $x_0^2 + x_1^2 + x_2^2$. Let us
consider the non degenerated quadric in the projective space
$S_2\subset\mathbb P(3,\mathcal C)$, defined by homogeneous equation
$\{t_0^2+t_1^2+t_2^2-t_3^2=0\}$. In affine coordinates $x_i =
\frac{t_i}{t_3}$, its affine part is a sphere of radius $1$. Thus
$SO(3, \mathcal C)$ is a subgroup of algebraic automorphisms of the
quadric; $SO(3)\subset Aut(S_2)$.

Each non degenerate quadric in the projective space over an
algebraically closed field is a hyperbolic ruled surface. It has two
systems of generatrices, being each system parameterized by a
projective line. Denote $P_1$, $P_2$ these projective lines. $p\in
P_1$, and $q\in P_2$ are lines $S_2$, and they intersect in a unique
point $s(p,q)\in p\cap q$. We have a decomposition of $S_2$ which is
a particular case of \emph{Segre isomorphism},
$$P_1 \times_{\mathcal C} P_2
\xrightarrow{\sim} S_2 \subset \mathbb P(3,\mathcal C)$$
$$((u_0:u_1),(v_0:v_1))\mapsto (t_0:t_1:t_2:t_3) \left\{\begin{matrix} t_0 &=& u_0v_1+u_1v_0 \\
t_1 &=& u_1v_1-u_0v_0 \\ t_2 &=& i(u_1v_1 + u_0v_0) \\ t_3&=& u_0v_1
- u_1v_0\end{matrix}\right.$$

  Let us consider any algebraic automorphism of $S_2$. $\tau\colon
S_2\to S_2$. In particular, it must carry a system of generatrices
to a system of generatrices. Let us denote $P_1$, $P_2$ to the two
system of generatrices of $S_2$. Hence, $\tau$ is induces by a pair
of projective transformations $(\tau_1,\tau_2)$, where
$$\tau_1\colon P_1 \to P_1, \quad \tau_2\colon P_2 \to P_2$$
or
$$\tau_1\colon P_1\to P_2, \quad \tau_2\colon P_2\to P_1.$$
We conclude that the group of automorphism of $S_2$ is isomorphic to
the following algebraic group,
$$\Aut(S_2) = PGL(1,\mathcal C) \times_{\mathcal C} PGL(1,\mathcal C)
\times_{\mathcal C} \mathbb Z/2\mathbb Z.$$

  Let us compute the image of the canonical monomorphism \linebreak
$SO(3,\mathcal C)\subset \Aut(S_2)$. We take affine coordinates in
the pair of projective lines, $x = \frac{u_0}{u_1}$, $y =
\frac{v_0}{v_1}$. This is the system of \emph{symmetric coordinates}
of the sphere introduced by Darboux \cite{Darboux}.
\begin{equation}\label{SymCoor1}
x_0 = \frac{1-xy}{x-y} \quad x_1 = i\frac{1+xy}{x-y} \quad x_2 =
\frac{x+y}{x-y}
\end{equation}
\begin{equation}\label{SymCoor2}
 x = \frac{x_0+ix_1}{1-x_2} \quad y =
\frac{x_2-1}{x_1 - ix_2}.
\end{equation}
Let us write a general element of $SO(3,\mathcal C)$ in affine
coordinates,
$$R_{\lambda,\mu,\nu} =
\left(\begin{matrix} 1 & & \\ & \frac{\lambda+\lambda^{-1}}{2} & \frac{\lambda^{-1}-\lambda}{2i}\\
& \frac{\lambda-\lambda^{-1}}{2i} & \frac{\lambda+\lambda^{-1}}{2}
\end{matrix}\right) \left(\begin{matrix} \frac{\mu+\mu^{-1}}{2} &  \frac{\mu^{-1}-\mu}{2i} &
\\ \frac{\mu-\mu^{-1}}{2i} & \frac{\mu+\mu^{-1}}{2} & \\ & & 1
\end{matrix}\right) \left(\begin{matrix} 1 & & \\ & \frac{\nu+\nu^{-1}}{2} & \frac{\nu^{-1}-\nu}{2i} \\
 & \frac{\nu-\nu^{-1}}{2i} & \frac{\nu+\nu^{-1}}{2} \end{matrix}\right)$$
where, in the complex case $\lambda = e^{i\alpha}$, $\mu =
e^{i\beta}$, $\nu = e^{i\gamma}$ are the exponentials of the Euler
angles. Direct computation gives us,
$$R_{\lambda, \mu, \nu}\left\{\begin{matrix}
x&\mapsto & \frac{(\lambda\mu\nu + \lambda\nu + \mu\nu - \nu +
\lambda\mu - \lambda + \mu +1 )x +
 \lambda\mu\nu + \lambda\nu + \mu\nu - \nu - \lambda\mu + \lambda -
 \mu -1}{(\lambda\mu\nu + \lambda\nu - \mu\nu + \nu + \lambda\mu - \lambda - \mu -1)x +
 \lambda\mu\nu + \lambda\nu - \mu\nu + \nu - \lambda\mu + \lambda + \mu +1} = r_{\lambda,\mu,\nu}(x)\\
y&\mapsto & \frac{(\lambda\mu\nu + \lambda\nu + \mu\nu - \nu +
\lambda\mu - \lambda + \mu +1 )y +
 \lambda\mu\nu + \lambda\nu + \mu\nu - \nu - \lambda\mu + \lambda -
 \mu -1}{(\lambda\mu\nu + \lambda\nu - \mu\nu + \nu + \lambda\mu - \lambda - \mu -1)y +
 \lambda\mu\nu + \lambda\nu - \mu\nu + \nu - \lambda\mu + \lambda + \mu +1} = r_{\lambda,\mu,\nu}(y)\end{matrix}\right.
$$
and then $R_{\lambda,\mu,\nu}$ induces the same projective
transformation $r_{\lambda,\mu,\nu}$ for $x$ and $y$. Hence,
$$SO(3) \subseteq PGL(1,\mathcal C) \subset Aut(S_2).$$
In particular, we have the following formulae for rotations around
euclidean axis:
\begin{align}\label{:EqRot1}
\left(\begin{matrix} 1 & & \\ & \frac{\lambda+\lambda^{-1}}{2} & \frac{\lambda^{-1}-\lambda}{2i}\\
& \frac{\lambda-\lambda^{-1}}{2i} & \frac{\lambda+\lambda^{-1}}{2}
\end{matrix}\right) &\colon x \mapsto \frac{(\lambda+1)x + (\lambda -1)}{(\lambda-1) x + (\lambda +
1)}
\\\label{:EqRot2}
\left(\begin{matrix} \frac{\lambda+\lambda^{-1}}{2} &
\frac{\lambda^{-1}-\lambda}{2i} &
\\ \frac{\lambda-\lambda^{-1}}{2i} & \frac{\lambda+\lambda^{-1}}{2} & \\ & & 1
\end{matrix}\right) &\colon x \mapsto \lambda x
\\\label{:EqRot3}
\left(\begin{matrix} \frac{\lambda+\lambda^{-1}}{2} & &
\frac{\lambda^{-1}-\lambda}{2i} \\ & 1 & \\
\frac{\lambda-\lambda^{-1}}{2i} & & \frac{\lambda+\lambda^{-1}}{2}
\end{matrix}\right) &\colon x \mapsto
\frac{(\lambda+\lambda^{-1}+1/2)x-i(\lambda-\lambda^{-1})}{i(\lambda^{-1}-\lambda)x
- (\lambda+\lambda^{-1}+1/2)}
\end{align}
An the following formulae for the induced Lie algebra morphism --
the are computed by derivation of previous formulae with $\lambda =
1 + i\varepsilon$ --. Here the Lie algebra $pgl(1,\mathcal C)$ is
identified with $sl(2,\mathcal C)$:
\begin{align*}
\left(\begin{matrix} & 1 &
\\ -1 & & \\ & & 0 \end{matrix}\right)&\mapsto\left(\begin{matrix} \frac{i}{2} & \\ & \frac{-i}{2}\end{matrix}\right)
\\
\left(\begin{matrix} & & 1
\\ & 0 & \\ -1 & & \end{matrix}\right)&\mapsto\left(\begin{matrix} & \frac{1}{2} \\ -\frac{1}{2} & \end{matrix}\right)
\\
\left(\begin{matrix} 0 & &
\\ & & 1 \\ & -1 & \end{matrix}\right)&\mapsto\left(\begin{matrix} & -\frac{i}{2} \\ - \frac{i}{2} \end{matrix}\right)
\end{align*}

Reciprocally, a projective transformation
$$x \mapsto \frac{u_{11} x + u_{12}}{u_{21}x+u_{22}}; \quad y
\mapsto \frac{u_{11}y + u_{12}}{u_{21}y+u_{22}},$$ induces a linear
transformation in the affine coordinates $x_0,x_1,x_2$ (see
\cite{Darboux} p. 34). $SO(\mathcal C,3)$ is precisely the group of
automorphisms of $S_2$ that are linear in those coordinates. We have
proven the following proposition which is due to Darboux.

\begin{proposition}
  The special orthogonal group $SO(3,\mathcal C)$ over an
  algebraically closed field is isomorphic to the projective general
  group $PGL(1,\mathcal C)$. The isomorphism is given by formulae
  \eqref{:EqRot1},  \eqref{:EqRot2},  \eqref{:EqRot3}.
\end{proposition}

\subsection{Flag Equation}

  The flag variety of $SO(3,\mathcal C)$ is a projective line.
Any of the Darboux symmetric coordinates,
$$x \colon  S_2 \to P_1$$
gives us a realization of the action of $SO(3)$ on $P_1$. By
substituting the equation \eqref{EqSO3} in the identities
\eqref{SymCoor1}, \eqref{SymCoor2} we deduce the Riccati
differential equation satisfied by this symmetric coordinate, which
is the flag equation of equation \eqref{EqSO3}:
\begin{equation}\label{RicSO3}
\dot x = \frac{-b-ic}{2}-iax+\frac{-b +ic}{2}x^2.
\end{equation}

In \cite{Darboux}, Darboux reduces the integration of the equation
\eqref{EqSO3} to finding two different particular solutions of the
Riccati equation \eqref{RicSO3}. By application of our
generalization of Liouville's theorem we obtain an stronger result.

\index{theorem!Darboux on rigid movements}
\begin{theorem}[Darboux]\label{C4THE4.3.8}
  The Galois extension of the equation \eqref{EqSO3} is a
Liouvillian extension of $\mathcal K$ if and only if the Riccati
equation \eqref{RicSO3} has an algebraic solution.
\end{theorem}

\begin{proof}
 It is a particular case of Theorem \ref{C4THE4.3.2}.
\end{proof}

\appendix

\section{Stalk formula for affine morphisms}\label{ApA}

\subsection{Stalk Formula for Ring Morphisms}

  Let us consider a ring morphism $\varphi\colon\mathcal R\to\mathcal
  R'$, and $\mathfrak a\subset \mathcal R$ an ideal. We write
  $\varphi(\mathfrak a)\cdot \mathcal R'$ for the ideal of $\mathcal
  R'$ spanned by the image of $\mathfrak a$ by $\varphi$.

\index{stalk formula}
\begin{theorem}[Stalk formula]\label{StalkFormula}
  Let us consider $x\in\Spec(\mathcal R)$. The stalk $(\varphi^*)^{-1}(x)\subset \Spec(\mathcal R')$ is
  homeomorphic to the spectrum of $$\mathcal R'_{\varphi(x)\cdot\mathcal
  R'}/\varphi(x)\cdot\mathcal R_{\varphi(x)\cdot\mathcal R'} =
  \left(\mathcal R'/\varphi(x)\cdot \mathcal R'\right)_{\varphi(x)\cdot \mathcal R'} =
  \mathcal R' \otimes_{\mathcal R} \kappa(x). $$
\end{theorem}

  Let us note that we do two different processes in the computation
of the stalk. First there is a process of localization: the spectrum
of $R'_{\varphi(x)\cdot \mathcal R'} = \mathcal R' \otimes_{\mathcal
R} \mathcal R_x$ is identified with the set of prime ideals
$y\subset \mathcal R'$ verifying $\varphi(y)\subseteq x$. Second
there is a process of restriction, the spectrum of $\mathcal
R'/\varphi(x)\cdot\mathcal R' = \mathcal R' \otimes_{\mathcal R}
\mathcal R/x$ is identified with the set of prime ideals
$y\subset\mathcal R'$ verifying $\varphi(y)\supseteq x$. These
processes commute. When we take both together we obtain $\mathcal
R'\otimes_{\mathcal R}\kappa(x)$. As expected, the canonical
morphism $\mathcal R'\to\mathcal R'\otimes_{\mathcal R}\kappa(x)$,
$a\mapsto a\otimes 1$ induces de immersion of the stalk into
$\Spec(\mathcal R')$.

\subsection{Stalk Formula for Change of Base Field}

\begin{definition}
  Let $X$ be an $k$-scheme, and $k\hookrightarrow \mathcal A$ a
$k$-algebra. We write $X(\mathcal A)$ for the set of $k$-scheme
homomorphisms $\Spec(\mathcal A)\to X$. The functor
$$X\colon \mathcal A \leadsto X(\mathcal A) = \Hom_k(\Spec(\mathcal A),
X)$$ of the category of $k$-algebras in the category of sets, is
called the functor of points of $X$. An element $x\in X(\mathcal A)$
is called an $\mathcal A$-point of $X$.
\end{definition}

  First, note that for each field extension $k\hookrightarrow K$
there is a map,
  $$X(K) \to X,\quad x \mapsto x((0)), \quad (0)\subset K$$
following this map, $X(k)$ is identified with the set of points of
$X$ whose rational field $\kappa(x)$ is $k$. We call these points
\emph{rational points of $X$}.

  For any field extension $k\subset K$, the map $X(K)\to X$ is
surjective onto the subset of points $x\in X$ for whom that there
exist a commutative diagram,
$$\xymatrix{k \ar[dr] \ar[rr] & & K \\ & \kappa(x) \ar[ru]}$$
and moreover, $X(K)$ is identified with the set of $K$-rational
points of the $K$-scheme $X_K$:
$$\xymatrix{ & X\times\Spec(K) \ar[d] \\ \Spec(K) \ar[ur]\ar[r] &
X}$$

If $X$ is of finite type, then $X(\bar{k})\mapsto |X|_{cl} \subset
|X|$ is surjective onto the subset of closed points of $X$.

\begin{theorem}
There is a canonical one-to-one correspondence between the set
$X(K)$ of $K$-points of $X$ and the set of rational points of the
extended scheme $X_K$.
\end{theorem}

\index{base change formula}
\begin{proposition}[Base change formula]\label{CBformula}
  Let $X$ be a $k$-scheme, $x\in X$, and $k\subset \mathcal A$ a
$k$-algebra. The stalk $\pi^{-1}(x)$ of $x$ by $\pi\colon
X_{\mathcal A}\to X$, is isomorphic to $\Spec(\kappa(x)\otimes_k
\mathcal A)$.
\end{proposition}

\begin{proof}
First, assume that $X=Spec(\mathcal B)$ is affine. Then, by stalk
formula, we have
$$\pi^{-1}(x) = \Spec(\mathcal A\otimes_k B \otimes_{\mathcal B}
\otimes \kappa(x)) = \Spec(\mathcal A \otimes_k \kappa(x)),$$ the
homeomorphism is induced by the ring morphism
$$A\otimes_k B \to A \otimes \kappa(x), \quad a\otimes f \mapsto a \otimes f(x).$$

If $X$ is not affine, then we cover it with affine subsets $U_i$. If
$\pi(y) = x$, and $x\in U_i$, then $y\in U_i\times_k\Spec(\mathcal
A)$ and the previous argument is sufficient.
\end{proof}

\subsection*{Acknowledgements}

  This research of both authors has been partially financed
by MCyT-FEDER Grant MTM2006-00478 of spanish goverment. The
first author is also supported by {\sc Civilizar}, the research 
agency of Universidad Sergio Arboleda. We also acknowledge 
prof. J.-P. Ramis and prof. E. Paul for their support 
during the visit of the first author to Laboratoire Emile Picard. 
We are also in debt with J. Mu\~noz of Universidad de Salamanca for 
his continuous help and support. We thank also P. Acosta,
T. Lazaro and C. Pantazi who shared with us
the seminar of algebraic methods in differential equations in Barcelona.

\bigskip

{\sc\noindent David Bl\'azquez-Sanz \\
Escuela de Matem\'aticas\\
Universidad Sergio Arboleda\\
Calle 74, no. 14-14 \\
Bogot\'a, Colombia\\
}
E-mail: {\tt david.blazquez-sanz@usa.edu.co}

\bigskip

{\sc\noindent Juan Jos\'e Morales-Ruiz \\
Departamento de Inform\'atica y Matem\'aticas\\
Escuela de Caminos Canales y Puertos\\
Universidad Polit\'ecnica de Madrid
Madrid, Espa\~na\\
}
E-mail: {\tt juan.morales-ruiz@upm.es}

%\printindex


\begin{thebibliography}{99}

\bibitem%[AMV2002]
{AMV}{\sc M. Adler, P. van Moerbeke, P.
Vanhaecke},{\it ``Algebraic Complete Integrable Systems, Painlev\'e
Architecture and Lie Algebras'',} Springer Verlag, 2002.

%\bibitem%[A-H1994]
%{Athorne1994}{\sc C. Athorne, T. Hartl, }{\it Solvable structures and hidden symmetries, }
%Journal of Physics A, 27 (1995), pp. 3463--3474.

%\bibitem%[Ath1997]
%{Ath1997}{\sc C. Athorne, }{\it Symmetries of Linear ordinary differential equations,}
%J. Phys. A; Math. Gen {\bf 30} (1997) pp. 4639--4649.

%\bibitem%[Ath1998]
%{Ath1998}{\sc C. Athorne, }{\it On the Lie symmetry algebra of general ordinary differential equation,}
%J. Phys. A; Math. Gen {\bf31} (1998) pp. 6605--6614.

\bibitem%[Be2008]
{Be2008}{\sc F. Benoist,}{\it D-algebraic geometry,}{ Preprint 2008.}


\bibitem%[BM2008a]
{BM2008}{\sc D. Bl\'azquez-Sanz, J. Morales-Ruiz}{\it Local and Global Aspects of Lie Superposition
Theorem}, preprint. 

\bibitem%[BM2008b]
{BMCharris}{\sc D. Bl\'azquez-Sanz, J. Morales-Ruiz,}{\it Lie's Reduction Method
and Differential Galois Theory in the Complex Analytic Context}, Preprint 2009. 

\bibitem%[Bi1962]
{BB} {\sc A. Bia{\l}ynicki-Birula, }{\it On Galois theory of fields with
operators, } Amer. J. Math. 84 (1962), pp. 89--109.

\bibitem%[Br1991]
{Bryant} {\sc R. L. Bryant, }{\it An introduction to Lie
Groups and Symplectic Geometry, } Lectures at the R.G.I. in Park
City (Utah) 1991.

\bibitem%[Bu1986]
{Bu} {\sc A. Buium, }{\it Differential Function Fields
and Moduli of Algebraic Varietes. } Lecture Notes in Mathematics,
Springer Verlag, 1986.

\bibitem%[CGM2000]
{Carinena2000} {\sc J. F. Cari\~nena, J. Grabowski, G. Marmo, }
{\it ``Lie-Scheffers systems: a geometric approach,'' } Napoli
Series on Physics and Astrophysics. Bibliopolis, Naples, 2000.

%\bibitem%[CGM2007]
%{Ca-Gr-Ma} {\sc J. F. Cari\~nena, J. Grabowski, G.
%Marmo, }{\it Superposition rules, Lie theorem, and Partial
%Differential equations,} Rep. Math. Phys. 60 (2007), no. 2, pp.
%237--258.

%\bibitem%[CGR2001]
%{Carinena2001} {\sc J. F. Cari\~nena, J. Grabowski, A. Ramos,} {\it Reduction of
%time-dependent systems admitting a superposition law,} Acta Appl.
%Math. 66 (2001), no. 1, pp. 67--87.

%\bibitem%[Ca-Ra2002]
%{Carinena2002} {\sc J. F. Cari\~nena, A. Ramos, }{\it A new geometric approach to Lie
%systems and physical applications. Symmetry and perturbation
%theory.,} Acta Appl. Math. 70 (2002), no. 1-3, pp. 43--69.

\bibitem%[Ca1990]
{Ca0} {\sc Carra' Ferro,} {\it Kolchin schemes,} J.
Pure and Applied Algebra {\bf 63} (1990), pp. 13--27.

%\bibitem%[Ca1947]
%{Cartan} {\sc E. Cartan,} {\it L'oeuvre scientifique
%de M. Ernest Vessiot }, Bull. S. M. F., tome 75 (1947), pp. 1-8.%

%\bibitem%[Ca2006]
%{Casale1} {\sc G. Casale, }{\it Feuilletages singuliers de codimension un, groupo\"{\i}de
%de Galois et int\'egrales premi\`eres. } Ann. Inst. Fourier
%(Grenoble) 56 (2006), no. 3, pp. 735--779.

%\bibitem%[Ca2007]
%{Casale2} {\sc G. Casale, }{\it The Galois groupoid of Picard-Painlev VI equation.
%Algebraic, analytic and geometric aspects of complex differential
%equations and their deformations. Painlev hierarchies.} Res. Inst.
%Math. Sci. (RIMS), Kyoto, (2007), pp. 15--20.

\bibitem%[Da1894]
{Darboux} {\sc G. Darboux,}
{\it ``Le\c{c}ons sur la th\'eorie g\'en\'erale des surfaces, I, II,''}
ditions Jacques Gabay, Sceaux, 1993, (Reprint).

%\bibitem%[FH1991]
%{FultonHarris} {\sc W. Fulton, J. Harris }{\it
%``Representation Theory, a first course,''} Graduate Texts in
%Mathematics, Springer Verlag 1991.

%\bibitem%[EGA]
%{EGA} {\sc A. Grothendieck, J. Diedonn\`e, }{\it
%``El\'ements de G\'eometrie Alg\'ebrique I, II, III, IV,''}
%1961--1967.

\bibitem%[SGA3]
{SGA3} {\sc A. Grothendieck, }{\it ``Sch\`emas en groupes, 1962--1964 (Group schemes),''}
Lecture Notes in Mathematics 151, 152 and 153, Springer Verlag 1970.

%\bibitem%[Gu1893]
%{Guldberg} {\sc A. Guldberg, }{\it Sur les \'equations diff\'erentialles ordinaries
%qui poss\`edent un syst\`eme fondamental d'int\`egrales,} Compt.
%Rend. Acad. Sci. Paris, T CXVI (1893), pp. 964--965.

%\bibitem%[Ha1977]
%{Ha} {\sc R. Hartshorne, }{\it ``Algebraic
%Geometry,''} Springer Verlag 1977.

%\bibitem%[HPW1999]
%{HPW} {\sc M. Havl\'i\v{c}ek, S. Po\v{s}ta, P. Winternitz, }{\it Nonlinear superposition
%formulas based on imprimitive group action, }J. Math. Phys. vol 40
%num 6 (1999), pp. 3104--3122.

\bibitem%[Ho-We1997]
{HoeijWeil} {\sc Mark van Hoeij, Jacques-Arthur
Weil, }{\it An algorithm for computing invariants of differential
Galois groups, } J. Pure App. Algebra 117 \& 118 (1997), pp.
353--379.

\bibitem%[HRUW1999]
{HRUW} {\sc Mark van Hoeij, Jean-Franois Ragot, Felix Ulmer, Jacques-Arthur Weil, }
{\it Liouvillian solutions of linear differential equations of order
three and higher.} J. Symbolic Comput. {\bf 28} (1999), no. 4-5, pp.
589--609.

\bibitem%[Hu1975]
{Humphreys}{\sc J. E. Humphreys, }{\it ``Linear
Algebraic Groups''}, Graduate Texts in Mathematics, Springer Verlag
1975.

\bibitem%[Ka1957]
{Ka} {\sc I. Kaplansky, }{\it ``An introduction to differential algebra'', }
Hermann, Paris 1957.

\bibitem%[Ke1982]
{Ke1} {\sc W. F. Keigher, }{\it Differential
Schemes and Premodels of Differential Fields,} J. of Algebra {\bf
79} (1982), pp. 37--50.

\bibitem%[Ke1983]
{Ke2} {\sc W. F. Keigher, }{\it On the structure
presheaf of a differential ring,} J. Pure and Applied Algebra {\bf
27} (1883), pp. 163--172.

\bibitem%[Ko1953]
{Ko0} {\sc E. R. Kolchin, }{\it Galois Theory of Differential Fields, }
Amer. J. Math., Vol. 75, No. 4. (Oct., 1953), pp. 753--824.

%\bibitem%[Ko-La1958]
%{KolchinLang} {\sc E. R. Kolchin, S. Lang, } {\it Algebraic
%Groups and Galois theory of differential fields,} Amer. J. Math.,
%Vol. 80, 1958

\bibitem%[Ko1973]
{Ko1} {\sc E. R. Kolchin, }{\it Differential Algebra and
Algebraic Groups, } Academic Press, New York 1973.

\bibitem%[Kov1983]
{Kov0} {\sc Jerald J. Kovacic, }{\it An algorithm for solving
second order linear homogeneous differential equations, } J.
Symbolic Comput. {\bf 2} (1986), pp. 3--43.

\bibitem%[Kov2002]
{Kov1} {\sc Jerald J. Kovacic, }{\it Differential
Schemes, } Proceedings of the International Workshop, Rutgers
University, Newark. World Scientific Publishing Co., River Edge, NL,
2002.

\bibitem%[Kov2003]
{Kov2} {\sc Jerald J. Kovacic, }{\it The Differential
Galois Theory of Strongly Normal Extensions, } Transactions of the
AMS, Vol. 355, Number 11, pp. 4475--4522

\bibitem%[Kov2006]
{Kov3} {\sc Jerald J. Kovacic, }{\it Geometric
Characterization of Strongly Normal Extensions, } Transactions of
the AMS, Vol. 358, Number 9, pp. 4135--4157

%\bibitem%[Lei2006]
%{Lei} {\sc Lei, Jinzhi, }{\it Nonlinear differential Galois
%theory, } Preprint, arXiv:math.CA/0608492v1 20 Aug 2006 (36 p).

%\bibitem%[Lie1888]
%{Lie1888} {\sc S. Lie, }{\it Allgemeine
%Untersuchungen \"uber Differentialgleichungen, die eine
%continuierliche endliche Gruppe gestatten,} Math. Ann. Bd. 25
%(1888), S. 71--151

%\bibitem%[Lie1893.a]
%{Lie1893} {\sc s. Lie, }{\it \"Uber
%Differentialgleichungen die Fundamentalintegrale besitzen,} Lepziger
%Berichte 1893, S. 341.

%\bibitem%[Lie1893.b]
%{Lie1893b} {\sc S. Lie, }{\it Vorlesungen \"uber continuerliche Gruppen mit
%Geometrischen und anderen Anwendungen,} Edited and revised by G.
%Scheffers, Teubner, Leizpig, 1893.

\bibitem%[Lie1893.c]
{Lie1893c} {\sc S. Lie, }{\it Sur les \'equations
diff\'erentielles ordinaries, qui poss\`eddent des systemes
fondamentaux d'integrales,} Compt. Rend. Acad. Sci. Paris, T CXVI
(1893), pp. 11233-1235.

%\bibitem%[Lie]
%{Lie2} {\sc S. Lie, G. Scheffers, } {\it ``Vorlesungen \"uber Differentialgleichungen
%mit bekannten Infinitesimalen Transformationen,'' } Reprinted by
%Chelsea books, N.Y. 1967.

\bibitem%[Lv1838]
{Liouville} {\sc J. Liouville, }  {\it M\'emoire sur l'integration de une classe
de \'equations diff\'erentielles du second ordre en quantit\'es
finies explicit\'es, } J. Math. Pures Appl. 4 (1839) pp. 423--456.

\bibitem%[Mlg2001]
{Malgrange1} {\sc B. Malgrange,} {\it ``Le
grupo\"ide de Galois d'un feuilletage,''} Monographie {\bf 38} vol
{\bf 2} de L'ensegnaiment math\'ematique (2001).

%\bibitem%[Mlg2002]
%{Malgrange2} {\sc B. Malgrange,} {\it On
%non-linear Galois differential theory, } Chin. Ann. Math. Ser. B
%{\bf 23}, num 2 (2002).

%\bibitem%[Ma1972]
%{Malliavin} {\sc P. Malliavin, }{\it ``G\'eom\'etrie
%diff\'erentielle intrins\`eque'',} Hermann, Paris 1972.

%\bibitem%[Maa1995]
%{Maamache1995}{\sc. M. Maamache,}{\it Ermakow systems, exact solution, and geometrical
%angles and phases,} Phys. Rev. A, \textbf{52}, 2 (1995), 936--940.

%\bibitem%[Mon1994]
%{Montgomery}{\sc S. Montgomery,}{\it ``Hopf algebras and their actions on rings,''} Dekker 1994.

%\bibitem%[Mo-Ra2001.a]
%{MoralesRamis1}
%{\sc J.J. Morales-Ruiz, J. P. Ramis, }{\it Galoisian obstructions to
%integrability of hamiltonian systems, I}, Methods and Applications
%of Analysis \textbf{8} (2001), pp. 33--95.

%\bibitem%[Mo-Ra2001.b]
%{MoralesRamis2}
%{\sc J.J. Morales-Ruiz, J. P. Ramis, }{\it Galoisian obstructions to
%integrability of hamiltonian systems, II}, Methods and Applications
%of Analysis \textbf{8} (2001), pp. 97--112.

\bibitem%[Mo1999]
{MoralesBook}{\sc J.J. Morales-Ruiz, }
{\it ``Differential Galois Theory and Non-Integrability of
hamiltonian systems,''} Progress in Mathematics {\bf 179},
Birkh\"auser (1999).

%\bibitem%[MRS2008]
%{MRS}
%{\sc J. J. Morales-Ruiz, J. P. Ramis, C. Sim\'o,} {\it Integrability
%of hamiltonian systems and differential Galois groups of higher
%variational equations, } to appear in Ann. Scient. \'Ec. Norm. Sup.

\bibitem%[Mum1970]
{Mum} {\sc D. Mumford, }{\it Abelian Varieties, }
Tata Institute of fundamental Research, Bombay 1970.

\bibitem%[Ni1989.a]
{Ni1} {\sc K. Nishioka, }{\it Differential Algebraic
function fields depending rationally on arbitrary constants, }
Nagoya Math. J. Vol. 113 (1989), pp. 173--179.

\bibitem%[Ni1989.b]
{Ni2} {\sc K. Nishioka, }{\it General solutions
depending rationally on arbitrary constants, } Nagoya Math. J. Vol.
113 (1989), pp. 1--6.

\bibitem%[Ni1997]
{Ni3} {\sc K. Nishioka, }{\it Lie Extensions, }
Proc. Jap. Acad., {\bf 73}, Ser. A (1997).

%\bibitem%[No1956]
%{Nomizu} {\sc K. Nomizu, }{\it ``Lie groups and differential geometry,''}
%Mathematical Society of Japan, Tokyo, 1956.

%\bibitem%[PB2004]
%{PerezBiswas} {\sc R. Perez-Marco, K. Biswas,} {\it
%Log-Riemann surfaces, } Preprint.

%\bibitem%[Pi2004]
%{Pillay} {\sc A. Pillay,} {\it
%Algebraic D-groups and differential Galois theory, } Pacific journal
%of mathematics, ISSN 0030-8730, Vol. 216, 2, 2004 , pp. 343--360.

%\bibitem%[RM1990]
%{RamisMartinet} {\sc J.-P. Ramis, J. Martinet,}{ Th\'eorie de Galois
%diff\'erentielle et resommation. } Computer algebra and differential
%equations, 117--214, Comput. Math. Appl., Academic Press, London,
%1990.

%\bibitem%[Ri1948]
%{Ritt1948} {\sc J. F. Ritt,}{\it ``Integration in Finite terms,''}
%Columbia University Press, 1948.

\bibitem%[Ri1950]
{Ritt1950} {\sc J.F. Ritt,}{\it ``Differential Algebra,''} Dover, 1950.

\bibitem%[Ro1963]
{Rosenlicht1963}{\sc M. Rosenlicht, }{\it
A remark on Quotient Spaces,} 
An. da Acad. Brasileira de Ciencies, v. 35, n. 4, 1963.

\bibitem%[Sa2001]
{Sa} {\sc C. Sancho de Salas, }{\it ``Grupos Algebraicos y teor\'ia de
invariantes'', } Sociedad Matem\'atica Mexicana, 2001.

%\bibitem%[S-P1994]
%{ShafarevichIV} {\sc A. N. Parshin, I. R. Shafarevich
%(Eds.),}{\it ``Algebraic Geometry IV: Linear Algebraic Groups,
%Invariant Theory''}, Encyclopaedia of Mathematical Sciences, Vol.
%55, Springer Verlag, 1995.

%\bibitem%[GAGA]
%{GAGA} {\sc J.P. Serre, }{\it G\'eometrie alg\'ebrique et g\'eometrie
%analytique, } Ann. Inst. Fourier, Grenoble (1956) pp. 1--42.

\bibitem%[Se1964]
{Serre} {\sc J. P. Serre, }{\it ``Cohomologie%
Galoisienne'', } Lecture Notes in Mathematics 5, Springer Verlag,
5th Ed. Rep. 1997, (1st Ed 1964).

\bibitem%[Si1990]
{Sibuya} {\sc Y. Sibuya, }{\it ``Linear Differential Equations in the Complex Domain:
Problems of Analytic Continuation.''} Transl. of Math. Monogr. 82,
Am. Math. Soc. Providence, Rodhe Island 1990.

%\bibitem%[Sh-Wi1984]
%{ShniderWinternitz} {\sc S. Shnider, P Winternitz, }{\it Classification of systems of nonlinear
%ordinary equations with superposition laws. } J. Math, Phys. {\bf
%25}(11), 1984, pp. 3155--3165.

%\bibitem%[Sh-Wi1985]
%{ShorineWinternitz} {\sc M. Shorine, P. Winternitz, }{\it
%Superposition Laws for Solutions of Differential Matrix Riccati
%Equations Arising in Control Theory.} IEEE Transcactions on
%automatic control, Vol. AC-30, NO. 3, 1985, pp. 266--272.

\bibitem%[Si-Ul1993.a]
{SingerUlmer1} {\sc M. F. Singer, F. Ulmer,
}{\it Galois Groups of Second and Third Order Linear Differential
Equations}, J. Symbolic Comput. (1993) {\bf 11}, pp. 1--36.

\bibitem%[Si-Ul1993.b]
{SingerUlmer2} {\sc M. F. Singer, F. Ulmer,
}{\it Liouvillian and Algebraic Solutions of Second and Third Order
Linear Differential Equations}, J. Symbolic Comput, (1997) {\bf 11},
pp. 37--73.

\bibitem%[Ul-We1994]
{UlmerWeil} {\sc F. Ulmer, J. A. Weil, }{\it Note on Kovacic's algorithm}
J. Symbolic Comput. {\bf 22}, pp. 179--200.

\bibitem%[Um1985]
{Umemura1985} {\sc H. Umemura, }{\it Birrational
automorphism groups and differential equations, } Proc.
Fraco-Japanese colloquium on differential equations, Strasbourg,
1985.

%\bibitem%[Um1996]
%{Umemura} {\sc H. Umemura, } {\it Galois theory of algebraic and differential
%equations,}  Nagoya Math. J., 144 (1996) pp. 1--58.

%\bibitem%[Um1997]
%{Umemura1997} {\sc H. Umemura, } {\it Lie-Drach-Vessiot theory -
%infinite-dimensional differential Galois theory,} CR-geometry and
%overdetermined systems, Osaka, 1994 (Math. Soc. Japan, Tokyo, 1997),
%pp. 364--385.

%\bibitem%[Va1996]
%{Vanhaecke} {\sc P. Vanhaecke,} {\it ``Integrable
%systems in the realm of Algebraic Geometry,''} Lecture Notes in
%Mathematics, Springer Verlag 1996.

\bibitem%[Va-Si2003]
{Vanderput} {\sc M. Vanderput, M. Singer,} {\it ``Galois theory of linear
differential equations''}, Grundlehren der Mathematischen
Wissenschaften [Fundamental Principles of Mathematical Sciences],
328. Springer-Verlag, Berlin, 2003.

\bibitem%[Ve1892]
{Ve1} {\sc E. Vessiot, }{\it Sur l'int\'egration des
quations diffrentielles lin\'eaires.} Annales Scientifiques de
l'Ecole Normale Sup\'erieure Sr. 3, 9 (1892), pp. 197--280.

\bibitem%[Ve1893.a]
{Ve2} {\sc E. Vessiot, }{\it Sur une classe d'equations
diff\'erentielles. } Annales scientifiques de l'E.N.S., 10 (1893),
pp. 53--64.

\bibitem%[Ve1893.b]
{Ve3} {\sc E. Vessiot, }{\it Sur une classe
syst\`emes d'\'equations diff\'erentielles ordinaires. } Compt.
Rend. Acad. Sci. Paris, T. CXVI (1893), pp. 1112--1114.


\bibitem%[Ve1894]
{Ve4} {\sc E. Vessiot, }{\it
Sur les syst\`emes d'\'equations diff\'erentielles du premier ordre
qui ont des syst\`emes fondamentaux d'int\'egrales.} Annales de la
facult\'e des sciences de Toulouse Sr. 1, 8 no. 3 (1894), pp.
H1--H33.

\bibitem%[Ve1904]
{Ve1904} {\sc E. Vessiot, }{\it Sur la th\'eorie de Galois et ses
diverses g\'en\'eralisations.} Annales scientifiques de l'cole
Normale Suprieure Sr. 3, 21 (1904), pp. 9--85.

\bibitem%[Ve1940]
{Ve1940} {\sc E. Vessiot, }{\it Sur la r\'eductibilit\'e des syst\`emes automorphes dont
le groupe d'automorphie est un groupe continu fini simplement
transitif.} Annales scientifiques de l'cole Normale Suprieure Sr.
3, 57 (1940), pp. 1--60.
\end{thebibliography}
\end{document}